\documentclass{amsart}

\usepackage[all]{xy}

\usepackage{hyperref}

\usepackage{cleveref}

\usepackage{amssymb}

\usepackage{xcolor}

\usepackage{chngcntr}
\usepackage{apptools}
\AtAppendix{\counterwithin{theorem}{subsection}}

\usepackage{relsize}
\usepackage[bbgreekl]{mathbbol}
\usepackage{amsfonts}
\usepackage{todonotes}

\usepackage{mathtools}

\DeclareSymbolFontAlphabet{\mathbb}{AMSb}
\DeclareSymbolFontAlphabet{\mathbbl}{bbold}

\theoremstyle{plain}

\newtheorem{theorem}{Theorem}[section]
\newtheorem{proposition}[theorem]{Proposition}
\newtheorem{lemma}[theorem]{Lemma}
\newtheorem{corollary}[theorem]{Corollary}

\theoremstyle{definition}
\newtheorem{definition}[theorem]{Definition}
\newtheorem{example}[theorem]{Example}
\newtheorem{remark}[theorem]{Remark}

\newcommand{\B}{\mathbb{B}}
\newcommand{\BC}{\mathrm{BC}}
\newcommand{\Bun}{\mathrm{Bun}}
\newcommand{\Coh}{{\mathcal{C}oh}}

\newcommand{\D}{\mathbb{D}}

\newcommand{\dR}{\mathrm{dR}}
\newcommand{\et}{\mathrm{{\acute{e}t}}}

\newcommand{\F}{\mathbb{F}}

\newcommand{\Fpbar}{\overline{\F}_p}

\newcommand{\GL}{\mathrm{GL}}

\newcommand{\Hom}{\mathrm{Hom}}
\newcommand{\Id}{\textrm{Id}}

\renewcommand{\inf}{\mathrm{inf}}

\newcommand{\N}{\mathbb{N}}
\newcommand{\Perf}{\mathrm{Perf}}

\newcommand{\Q}{\mathbb{Q}}

\newcommand{\R}{\mathbb{R}}

\newcommand{\Spa}{\mathrm{Spa}}
\newcommand{\Spd}{\mathrm{Spd}}
\newcommand{\Spec}{\mathrm{Spec}}

\newcommand{\Z}{\mathbb{Z}}

\newcommand{\bs}{\blacksquare}

\newcommand{\Qlbar}{\overline{\Q}_\ell}

\newcommand{\red}[1]{{\color{black} #1}}

\begin{document}

\title{A Fourier transform for Banach-Colmez spaces}

\author{Johannes Ansch\"{u}tz}
\address{Mathematisches Institut, Universit\"at Bonn, Endenicher Allee 60, 53115 Bonn, Deutschland}
\email{ja@math.uni-bonn.de}

\author{Arthur-C\'esar Le Bras}
\address{Universit\'e Sorbonne Paris Nord, LAGA, C.N.R.S., UMR 7539, 93430 Villetaneuse,
France}
\email{lebras@math.univ-paris13.fr}

\begin{abstract}
We define and study an $\ell$-adic Fourier transform for a relative version of Banach-Colmez spaces (over a perfectoid space which is not necessarily a geometric point), which can be thought of as some analytic analogue of the $\ell$-adic Fourier transform for unipotent perfect group schemes. We explicitly describe it on some examples. 
\end{abstract}

\begin{abstract}
We define and study an $\ell$-adic Fourier transform for a relative version of Banach-Colmez spaces (over a perfectoid space which is not necessarily a geometric point), which can be thought of as some analytic analogue of the $\ell$-adic Fourier transform for unipotent perfect group schemes. We explicitly describe it in some examples. 
\end{abstract}

\maketitle

\tableofcontents

\section{Introduction}
\label{sec:introduction}
Let $p$ and $\ell$ be two distinct primes.

Let $k$ be a characteristic $p$ perfect field and let $\psi: \mathbb{F}_p \to \Qlbar^\times$ be a fixed non-trivial additive character. Let $S$ be a $k$-scheme, and let $\mathcal{E}$ be a vector bundle of constant rank $d$ on $S$. Let $p_{\mathcal{E}} : \mathbb{V}(\mathcal{E}) \to S$ be the total space of this vector bundle, and $p_{\mathcal{E}^\vee} : \mathbb{V}(\mathcal{E}^\vee) \to S$, where $\mathcal{E}^\vee$ is the $\mathcal{O}_S$-linear dual of $\mathcal{E}$. Denote by $\pi : \mathbb{V}(\mathcal{E}^\vee) \times_S \mathbb{V}(\mathcal{E}) \to \mathbb{V}(\mathcal{E})$, $\pi^\vee: \mathbb{V}(\mathcal{E}^\vee) \times_S \mathbb{V}(\mathcal{E}) \to \mathbb{V}(\mathcal{E}^\vee)$ the two projections. Finally, let $\alpha: \mathbb{V}(\mathcal{E}^\vee) \times_S \mathbb{V}(\mathcal{E}) \to \mathbb{A}_S^1$ coming from the duality pairing between $\mathcal{E}$ and its dual. The \textit{geometric Fourier transform} for $\mathcal{E}$, invented by Deligne in a 1976 letter to Kazhdan, is the functor
$$
\mathcal{F}_\psi : D_c^b(\mathbb{V}(\mathcal{E}), \Qlbar) \to D_c^b(\mathbb{V}(\mathcal{E}^\vee), \Qlbar)
$$
defined by
$$
\mathcal{F}_\psi(-) := R\pi^\vee_! (\pi^\ast (-) \otimes \alpha^\ast \mathcal{L}_\psi)[d],
$$ 
where $\mathcal{L}_\psi$ is the rank $1$ $\Qlbar$-local system on $\mathbb{A}_S^1$ deduced from $\psi$ by Artin-Schreier theory. When $S=\mathrm{Spec}(k)$ and $\mathcal{E}=\mathcal{O}_S^d$ (so that $\mathbb{V}(\mathcal{E})=\mathbb{A}_k^d$), the trace function of the geometric Fourier transform of an object $A \in D_c^b(\mathbb{A}_k^d,\Qlbar)$ which is defined over $\mathbb{F}_q$, for some power $q$ of $p$, agrees up to sign with the Fourier transform, in the function-theoretic sense, of the trace function of $A$, justifying the name. 

The geometric Fourier transform has been studied extensively by Laumon, \cite{laumon_transformation_de_fourier}, who gave several striking applications of it. The key property of the Fourier transform is that it induces an equivalence of triangulated categories, identifying the subcategories of perverse sheaves on both sides.

The possibility of extending this Fourier transform to a more general setting was already noticed by Deligne in the same letter. Assume that $S$ is a perfect scheme over $k$, which is harmless as far as categories of \'etale $\ell$-adic sheaves are concerned. Let $\mathcal{G}$ be a connected unipotent perfect group scheme over $S$ (i.e. a commutative perfect group scheme over $S$ whose pullback to any perfect field-valued point $s$ of $S$ admits a composition series with quotients isomorphic to the perfection of $\mathbb{A}_{k(s)}^1$), with Serre dual\footnote{All the sheaves $\mathcal{E}xt_{\mathcal{S}}^i(\mathcal{G}, \Q_p/\Z_p)$, $i\neq 1$, vanish.}
$$
\mathcal{G}^\vee= \mathcal{E}xt_{\mathcal{S}}^1(\mathcal{G}, \Q_p/\Z_p),
$$
where $\mathcal{S}$ denotes the perfect \'etale site of $S$. (For more on unipotent perfect group schemes, we refer the reader to \cite[\S 2]{milne1976duality} or \cite[\S 2]{berthelot_le_theoreme_de_dualite_plate_pour_les_surfaces}.) If one fixes an additive character $\psi: \Q_p/\Z_p \to \Qlbar^\times$, one can define as before a Fourier transform functor
$$
\mathcal{F}_\psi : D_c^b(\mathcal{G},\Qlbar) \to D_c^b(\mathcal{G}^\vee, \Qlbar),
$$
using the natural pairing
$$
\mathcal{G} \times_S \mathcal{G}^\vee \to [S/(\Q_p/\Z_p)]
$$
coming from the definition of the Serre dual. Since any vector bundle $\mathcal{E}$ on $S$ gives rise to a connected unipotent perfect group scheme $\mathcal{G}_{\mathcal{E}}$, which is simply the perfection of $\mathbb{V}(\mathcal{E})$, this generalizes the Fourier transform defined before\footnote{A computation (cf. e.g. \cite[Proof of Lemma 2.2]{milne1976duality}) shows that the group scheme attached to $\mathcal{E}^\vee= \mathcal{H}om_{\mathcal{O}_S}(\mathcal{E},\mathcal{O}_S)$ agrees with the cokernel of $\varphi-1: \mathcal{H}om_{\mathcal{S}}(\mathcal{G}_{\mathcal{E}},\mathbb{A}_S^1) \to \mathcal{H}om_{\mathcal{S}}(\mathcal{G}_{\mathcal{E}},\mathbb{A}_S^1)$ so that, by the Artin-Schreier sequence and Breen's computations \cite{breen_extensions_du_groupe_additif_sur_le_site_parfait} of self-extensions of the affine line on the perfect \'etale site, $\mathcal{G}_{\mathcal{E}}^\vee \cong \mathcal{G}_{\mathcal{E}^\vee}$, showing that the notions of duality indeed match.}. The basic properties of this Fourier transform were established in Saibi's thesis \cite{saibi1996transformation}. This formalism could be extended to a more general class of perfect unipotent group schemes (perfect unipotent group schemes are perfect group schemes whose pullback to any perfect field-valued point $s$ of $S$ admits a composition series with quotients isomorphic to $\Z/p$ or the perfection of $\mathbb{A}_{k(s)}^1$), not necessarily connected. The price to pay is that one then has to work with complexes in degrees $[-1,0]$ of unipotent perfect group schemes whose cohomology in degree $-1$ is \'etale: this comes from the fact that
$$
\mathcal{H}om_{\mathcal{S}}(\Z/p, \Q_p/\Z_p) \cong \Z/p, \quad \mathcal{E}xt_{\mathcal{S}}^i(\Z/p, \Q_p/\Z_p)=0,~ i>0,
$$ 
so that the duality for \'etale groups comes with a different shift than the one for connected groups.
\\

Let $E$ be a local field of residue characteristic $p$. Our goal in this text is to study an analogue of these constructions when $S$ is assumed to be a perfectoid space in characteristic $p$ (thus living in Huber's world of adic spaces), or even a small v-stack (in the sense of Scholze \cite{scholze_etale_cohomology_of_diamonds}), rather than a perfect scheme. For these analytic objects, the \'etale topology gets replaced by the v-topology, and the role of unipotent perfect group schemes is played by Banach-Colmez spaces for the local field $E$.

Banach-Colmez spaces were first defined, \cite{colmez2002espaces}, when $E=\Q_p$ and $S=\mathrm{Spa}(C)$, with $C$ a complete algebraically closed non-archimedean valued field of characteristic $p$, is a geometric point: they are v-sheaves which are all ``obtained" from $\underline{\Q_p}$ and the adic affine line over $C^{\sharp}$, for a fixed untilt $C^\sharp$ of $C$, which is quite reminiscent of what happens for unipotent perfect group schemes and motivates our analogy. The main result of \cite{le_bras_espaces_de_banach_colmez_et_faisceaux_coherents_sur_la_courbe_de_fargues_fontaine} shows that the category of Banach-Colmez spaces is closely related to the category of coherent sheaves on the Fargues-Fontaine curve. For a local field $E$ other than $\Q_p$ and for a general perfectoid space $S$ in characteristic $p$, a category of Banach-Colmez spaces (for $E$) over $S$ has not been defined and will not be needed for our purposes: rather, we restrict our attention to the v-sheaves of $\underline{E}$-vector spaces coming from (some) coherent sheaves on the relative Fargues-Fontaine curve $X_{S,E}$, in a sense which will be made precise in the first part of the paper, where we discuss coherent sheaves on relative Fargues-Fontaine curves and some of their properties. In fact, since the same dichotomy connected/\'etale as for unipotent perfect group schemes persists in this new setting, it is also necessary to switch from sheaves to stacks. We therefore define a notion of \textit{very nice stack in $E$-vector spaces} and a duality functor making them dualizable, and verify that Banach-Colmez spaces attached to coherent sheaves with only positive slopes or Banach-Colmez spaces attached to semi-stable vector bundles of slope zero (i.e., pro-\'etale $\underline{E}$-local systems) are examples of such. This is a non-trivial statement, which in particular involves a computation of extensions between certain v-sheaves of $\underline{E}$-vector spaces, generalizing (and streamlining) the partial computations done in \cite{le_bras_espaces_de_banach_colmez_et_faisceaux_coherents_sur_la_courbe_de_fargues_fontaine}. Continuing our analogy, Banach-Colmez spaces attached to coherent sheaves with only positive slopes and their duals take the role of connected unipotent groups, while pro-\'etale $\underline{E}$-local systems take the role of \'etale unipotent groups. 

We define the Fourier transform for very nice stacks in $E$-vector spaces. Since the formalism of $\Qlbar$-sheaves on small (Artin) v-stacks, \cite[\S VII]{fargues2021geometrization}, is more complicated to work with than its classical counterpart (e.g., excision fails for solid coefficients), we set up the definitions for \'etale sheaves of $\Lambda$-modules, where $\Lambda$ is a $\Z/\ell^n$-algebra, for some $n \geq 1$, as developed in \cite{scholze_etale_cohomology_of_diamonds}. This Fourier transform satisfies the same formal properties as the other ones, the most important being that it is (essentially) involutive and therefore gives rise to an equivalence of categories commuting with Verdier duality. However, we have nothing to say about preservation of perversity by the Fourier transform -- an important and useful property in the algebraic setting -- since there is currently no developed theory of perverse sheaves on diamonds or Artin v-stacks.

In the last part of the paper, we study more in depth three examples. 
\begin{itemize}
\item The first one is the case of finite dimensional $E$-vector spaces,\ cf.\ \Cref{sec:the-case-of-E-vector-spaces}. We prove that in this situation, the Fourier transform is, unsuprisingly, closely related to the function-theoretic Fourier transform for locally constant functions. A variant of this example is shown to geometrize some constructions of Bernstein-Zelevinsky \cite{bernstein_zelevinsky_representations_of_the_group_glnf}. 
\item The next example is the case of the affine line over a fixed untilt of $S$, when $S$ is a geometric point, cf.\ \Cref{the-case-of-the-affine-line-in-characteristic-zero}. A Fourier transform had been defined already in this setting by Ramero \cite{ramero1998class}; it coincides with ours, which can thus be seen as a generalization of Ramero's Fourier transform to a larger class of stacks in $E$-vector spaces.
\item The final example we analyze is the Banach-Colmez space attached to the line bundle $\mathcal{O}(1)$,\ cf.\ \Cref{sec:case-from-bcmathc}. This Banach-Colmez space is an interesting geometric object: seen over $k$, it is the adic spectrum of ring of integers in a positive characteristic local field (hence, \'etale sheaves on it are related to representations of the Galois group of this field), but after base change to a geometric point $\mathrm{Spa}(C)$, it becomes isomorphic to the (perfectoid) punctured open unit disk over $C$, a nice smooth adic curve, up to perfection. Exploiting this together with Huber's adic version of the Grothendieck-Ogg-Shafarevich formula (\cite{huber2001swan}), we can compute explicitly the rank of the Fourier transform. Remarkably, the result is related to formal degrees of smooth irreducible representations of the non-split quaternion algebra over $E$.
\end{itemize} 

Our main motivation for introducing this Fourier transform comes from our attempt to understand some constructions of Laumon in the geometric Langlands program in the setting of the Fargues-Fontaine curve. This will be explored in another paper and is not developed here, but is briefly mentioned in \Cref{sec-an-inductive-construction} and motivates some of the computations we make, e.g. in \Cref{sec:the-case-of-E-vector-spaces} and \Cref{sec:case-from-bcmathc}.

\subsection*{Notations and conventions}
\label{sec:notations-and-conventions}
Let $p$ be a prime. In all this text, we fix $E$, which is either a finite extension of $\Q_p$, with residue field $\mathbb{F}_q$, or the field of Laurent series $\mathbb{F}_q((\pi))$, where $q$ is a power of $p$. We let $\mathcal{O}_E$ be the ring of integers of $E$ and $\pi$ be a uniformizer of $E$.

The category of perfectoid spaces over a small v-stack $S$ will be denoted $\Perf_S$. The v-site of $S$ is the opposite of this category endowed with the v-topology and is denoted $S_v$. In this text, all v-stacks will be on $\Perf_{\mathbb{F}_q}$. 

One can attach to any $S\in \Perf_{\mathbb{F}_q}$ the relative Fargues-Fontaine curve $X_{S,E}$. We will always abbreviate it to $X_S$; in other words, all the Fargues-Fontaine curves appearing in this text will be Fargues-Fontaine curves for the local field $E$.

We fix a prime $\ell \neq p$, and a $\Z/\ell^n$-algebra $\Lambda$, $n\geq 1$, which will serve as our coefficient ring. We will denote by the same letter the associated condensed ring.

Whenever we write ``cohomologically smooth", we always mean ``$\ell$-cohomologically smooth".

\subsection*{Acknowledgements} We thank Sebastian Bartling, David Hansen, Fabian Hebestreit, Guy Henniart, Lucas Mann, C\'edric P\'epin and Peter Scholze heartily for discussions about the content of this paper. We also thank Laurent Fargues for sharing with us several years ago his great observation that the \textit{stack} of self-extensions of $\mathcal{O}$ by $\mathcal{O}$ can be used to geometrize characters of the additive group $E$. The second author would like to thank Dustin Clausen for offering him the opportunity to lecture on this paper at the University of Copenhagen in the fall 2021, and all the participants there for their questions.

\section{Perfect complexes on relative Fargues-Fontaine curves}
\label{sec-coherent-sheaves-on-relative-ff-curves}
In this section we prove some foundational results on perfect complexes on perfectoid spaces and relative Fargues-Fontaine curves, e.g., v-descent, building on work of Andreychev \cite{andreychev2021pseudocoherent} and Mann \cite{lucas_mann_in_progress}. As an application we introduce \textit{flat coherent sheaves} on relative Fargues-Fontaine curves, and show that they satisfy v-descent.

\subsection{v-descent for perfect complexes on perfectoid spaces}
\label{subsection-v-descent-perfectoid-spaces}

Let $X$ be a perfectoid space (not necessarily over $\F_p$). We denote by
\[
  \mathcal{P}erf(X)=\mathcal{P}erf(X, \mathcal{O})
\]
the ($\infty$-)category of perfect complexes of $\mathcal{O}_X$-modules on $X$, i.e., those complexes which locally for the analytic topology are quasi-isomorphic to a bounded complex of finite, locally free $\mathcal{O}_X$-modules.
If $X=\Spa(R,R^+)$ is affinoid, then by \cite[Theorem 5.43, Lemma 5.46, Lemma 5.47]{andreychev2021pseudocoherent} the natural functor
\[
  \mathcal{P}erf(R)\to \mathcal{P}erf(X)
\]
is an equivalence, where the left-hand side denotes the ($\infty$-)category of perfect complexes of $R$-modules.

If $a\leq b$ are two integers (or $\pm \infty$), we say that a perfect complex has Tor-amplitude in $[a,b]$ if it can locally be represented by a complex whose terms in degrees outside $[a,b]$ are zero. We will denote by 
\[
  \mathcal{P}erf^{[a,b]}(X)
\]
the subcategory of $\mathcal{P}erf(X)$ formed by perfect complexes with Tor-amplitude in $[a,b]$.

Let $X$ be a perfectoid space, mapping to a totally disconnected space $W$ (for example, $W=\mathrm{Spa}(K)$, with $K$ a non-archimedean field). Let 
\[
  f\colon Y\to X
\]
be a v-cover. We want to prove first that perfect complexes descend along $f$. More precisely, consider the \v{C}ech nerve
\[
  Y^{\bullet/X}
\]
of $f$, i.e., the simplicial perfectoid space with $n$-simplices given by the $(n+1)$-fold fiber product
\[
 Y^{n/X}:=Y\times_XY\times\ldots \times_X Y
\]
Let us denote by
\[
  f_n\colon Y^{n/X}\to X 
\]
the natural projection.

The following theorem generalizes \cite[Lemma 17.1.8]{scholze2020berkeley} about v-descent of vector bundles to v-descent of perfect complexes.

\begin{theorem}
  \label{sec:desc-perf-compl-1-descent-of-perfect-complexes-on-perfectoid-spaces}
  With the above assumptions and notations, the canonical functor
  \[
    \Phi\colon \mathcal{P}erf(X)\to \varprojlim\limits_{\Delta} \mathcal{P}erf(Y^{\bullet/X})
  \]
  is an equivalence.
\end{theorem}
\begin{proof}
  Let $a,b\in \Z$ with $a\leq b$. As pullback of perfect complexes preserves amplitude in $[a,b]$ and the amplitude descends along v-covers we equivalently have to show the claim with $\mathcal{P}erf(-)$ replaced everywhere by its version $\mathcal{P}erf^{[a,b]}(-)$ of perfect complexes with amplitude in $[a,b]$.
  If
  \[
    X=\bigcup\limits_{i\in I}X_i
  \]
  is an open covering of $X$ by affinoid perfectoid spaces $X_i$, then
  $\mathcal{P}erf^{[a,b]}(X)$ identifies with a limit of the $\infty$-categories $\mathcal{P}erf^{[a,b]}(U)$ for $U$ a finite intersection of some of the $X_i$, and similarly
  $\mathcal{P}erf^{[a,b]}(Y)$ identifies with a limit of the $\mathcal{P}erf^{[a,b]}(Y\times_X U)$ (this does not need the results of \cite{andreychev2021pseudocoherent} yet, but follows formally from the definition of perfect complexes via $\mathcal{O}_X$-modules).
  Commuting limits, we can therefore reduce to the case that $W$ and $X=\Spa(R,R^+)$ are affinoid perfectoid. By refining the v-cover $Y\to X$ we can furthermore assume $Y=\Spa(S,S^+)$ is affinoid perfectoid. In this case we can apply \cite[Theorem 1.4]{andreychev2021pseudocoherent} which identifies perfect complexes on affinoid perfectoid spaces with perfect complexes over the ring of global sections.
  
  Let us first check that $\Phi$ is fully faithful.
  The fiber products
\[
  Y^{n/X}\cong \Spa(S_n,S_n^+)
\]
are again affinoid perfectoid with
\[
  S_n:=S\hat{\otimes}_R\hat{\otimes}_R \ldots \hat{\otimes}_R S
\]
the $(n+1)$-fold completed tensor product of $S$ over $R$.
Let $K_1, K_2\in \mathcal{P}erf^{[a,b]}(R)$ be two objects. Then we have to see that the natural morphism
\[
  R\Hom_{R}(K_1,K_2)\to \varprojlim\limits R\Hom_{S_n}(S_n\otimes_R^L K_1, S_n\otimes_R^L K_2)
\]
is an isomorphism in $D(R)$. Indeed, the right-hand side identifies with the spectrum (in the topological sense) of homomorphisms from $\Phi(K_1)$ to $\Phi(K_2)$. If $K_1=R, K_2=R$, then the claim follows from exactness of the complex
\[
  R\to \underset{=S}{S_0}\to S_1 \to S_2\to ...., 
\]
which is proven in \cite[Proposition 8.8]{scholze_etale_cohomology_of_diamonds}.
By d{\'e}vissage in $\mathcal{P}erf^{[a,b]}(R)$ and passage to direct summands we can therefore conclude the same for $K_1,K_2\in \mathcal{P}erf^{[a,b]}(R)$ arbitrary.
This finishes the proof of fully faithfulness of $\Phi$ in the affinoid perfectoid case.

Let $g\colon \tilde{X}\to X$ be an open covering by rational subsets.
Using analytic descent for $g$ and its base change to $Y^{n/X}$, cf. \cite[Theorem 1.4]{andreychev2021pseudocoherent}, and applying fully faithfulness to the constituents of the \v{C}ech nerve for $g\colon \tilde{X}\to X$ it suffices to construct the preimage of some given object
\[
  N^{\bullet}\in \varprojlim\limits_\Delta \mathcal{P}erf^{[a,b]}(Y^{\bullet/X})
\]
locally in the analytic topology on $X$.

For proving essential surjectivity of $\Phi$, we first assume that $X=\Spa(K,K^+)$ is the spectrum of an affinoid perfectoid field. Because $Y\to X$ is a v-cover the space $Y$ is non-empty and thus after a refinement of $Y$ we may assume that $Y=\Spa(L,L^+)$ is the spectrum of an affinoid field, too. In this case the canonical truncations (in $D(L)$) of an object in $\mathcal{P}erf^{[a,b]}(Y)$ lie again in $\mathcal{P}erf^{[a,b]}(Y)$ and acquire canonical descent data. By \cite[Lemma 17.1.8]{scholze2020berkeley} we can therefore conclude descent in this case.

To deal with the general case, we will need to argue integrally. In the following, we call an ``integral model'' of a perfect complex
\[
  K\in \mathcal{P}erf(Z)\cong \mathcal{P}erf(B)
\]
on an affinoid perfectoid space $Z=\Spa(B,B^+)$ a pair of a perfect complex $K^+\in \mathcal{P}erf(B^+)$ and an isomorphism $K^+\otimes_{B^+}^L B \cong K$.

Now assume that $X$ is arbitrary, pick a point $x\in X$, and let $k(x)$ be the completed residue field at $x$. As we saw above it suffices to descend $N^\bullet$ in a neighborhood of $x$.
By the previous case we can descend the restriction $N_x$ of $N$ to the fiber $Y_x$ of $Y\to X$ over $x$, i.e., the descent datum
\[
  N_x^\bullet\in \varprojlim\limits_\Delta \mathcal{P}erf^{[a,b]}(Y^{\bullet/x}_x)
\]
obtained by base change of $N$ is effective. Let $M\in \mathcal{P}erf(k(x))$ be the descent of $N_x^\bullet$. Then we can choose an integral model $M^+\in \mathcal{P}erf^{[a,b]}(k(x)^+)$ of $M$.
Pulling back to $Y^{\bullet/x}_x$ we get a componentwise integral model
\[
  N_x^{+,\bullet}\in \varprojlim\limits_{\Delta} \mathcal{P}erf^{[a,b]}(Y^{\bullet/x}_x,\mathcal{O}^+),
\]
of $N_x^\bullet$.
For an open neighborhood $U\in X$ of $x$ set $Y_U:=Y\times_X U$ with associated \v{C}ech nerve
\[
  Y_U^{\bullet/U}.
\]
We denote the componentwise restriction of the descent datum $N$ to $Y^{\bullet/U}_U$ by
\[
  N_U^{\bullet/U}.
\]
By \Cref{sec:desc-perf-compl-1-integral-models-commute-with-colimits}, we can find for each $n\in \N$ an open rational neighborhood $U\subseteq X$ and an integral model
\[
  N_{U}^{+,n/U}
\]
of $N_U^{n/U}$ whose restriction to $Y_x^{n/x}$ is $N_x^{+, n/x}$. As all the considered perfect complexes have uniform amplitude contained in $[a,b]$ we can choose $U$ independently of $n$. Applying again \Cref{sec:desc-perf-compl-1-integral-models-commute-with-colimits} we can after shrinking $U$ further assume that the descent datum on $N_U^{\bullet/U}$ restricts to $N_U^{+,\bullet/U}$ because the descent datum of $N_x^{\bullet/x}$ restricts to $N_x^{+,\bullet/x}$.
More formally, by \Cref{sec:desc-perf-compl-1-integral-models-commute-with-colimits} and the fact that limits over $\Delta$ over uniformly truncated anima/$\infty$-categories commute with filtered colimits the functor
\[
  (B,B^+)\mapsto \varprojlim\limits_{\Delta} \mathcal{P}erf^{[a,b]}((B\hat{\otimes}_RS_\bullet)^+)\times_{\mathcal{P}erf^{[a,b]}(B\hat{\otimes}_R S_\bullet)} \ast
\]
on complete Tate-Huber $(R,R^+)$-algebras commutes with filtered colimits (here $\ast \to \mathcal{P}erf^{[a,b]}(B\hat{\otimes}_R S_\bullet)$ has image the base change of $N^\bullet$).

After replacing $X$ by $U$, $Y$ by $Y_U$ and $N^{\bullet/X}$ by $N^{\bullet/U}_U$ we can therefore assume that the whole descent datum $N^{\bullet}$ admits an integral model $N^{+,\bullet}$.

  We will use the following result of Mann, \cite{lucas_mann_in_progress}. Recall that a map $\mathrm{Spa}(B,B^{+}) \to \mathrm{Spa}(A,A^{+})$ of affinoid perfectoid spaces is \textit{weakly of perfectly finite type} if there is a surjection $A\langle X_1^{1/p^{\infty}}, \dots, X_n^{1/p^{\infty}} \rangle \to B$ for some $n\geq 1$.  

\begin{theorem}
\label{theorem-lucas}
Let $Z=\mathrm{Spa}(T,T^+)$ be an affinoid perfectoid space, weakly of perfectly finite type over the totally disconnected affinoid perfectoid space $W$, and let $\varpi$ be a pseudo-uniformizer of $T$. The category $$\mathcal{D}_{\bs}^a(Z,\mathcal{O}^{+}/\varpi), $$
introduced in \cite{lucas_mann_in_progress}, satisfies v-descent. Moreover, the full subcategory of dualizable objects in $\mathcal{D}_\bs^a(Z,\mathcal{O}^{+}/\varpi)$ is equivalent to the category $\mathcal{D}ual(T^{+a}/\varpi)$ of dualizable objects in the derived category of almost $T^{+a}/\varpi$-modules via the natural functor $\mathcal{D}(T^{+a}/\varpi)\to \mathcal{D}_\bs^a(Z,\mathcal{O}^{+}/\varpi)$.
\end{theorem} 

  Every map (resp.\ $v$-cover) of affinoid perfectoid spaces over $W$ is a cofiltered limits of maps (resp.\ $v$-covers) between affinoid perfectoid spaces which are weakly of perfectly finite type over $W$,\ cf. \cite{lucas_mann_in_progress}. In the following we will use that the functor $A\mapsto \mathcal{P}erf(A)$ on rings commutes with filtered colimits. Now, let $n\geq 1$ and fix a uniformizer $\varpi\in \mathcal{O}_W(W)$. We can write the $v$-cover $Y\to X$ as the cofiltered limit of $v$-covers $Y_i\to X_i=\Spa(R_i,R_i^+), i\in I,$ with $Y_i, X_i$ weakly of perfect finite type over $W$. As the descent datum $N^{+, \bullet}$ has uniformly bounded amplitude, there exists an $i\in I$, such that $N^{+,\bullet}/\pi^n$ is the base change of a descent datum $N^{+,\bullet}_{i,n}$ for $Y_i\to X_i$ of perfect $\mathcal{O}^+/\varpi^{n}$-modules. By \ref{theorem-lucas} (the almostification of) this descent datum descents to an almost dualizable complex $M^{+a}_{i,n}$ of almost $R^{+a}_i/\varpi^n$-modules. We let
  \[
    M^{+a}_n:=M^{+a}_{i,n}\otimes^{L}_{R_i^{+a}} R^{+a}.
  \]
  Constructing $M^{+a}_n$ inductively, we can arrange that we get isomorphisms
  \[
    M^{+a}_{n+1}\otimes_{R^{+a}/\varpi^{n+1}}^L R^{+a}/\varpi^{n}\cong M^{+a}_{n}
  \]
  for $n\geq 1$, i.e., we construct an object
  of the category
  \[
    \mathcal{D}ual(R^{+a})^{\rm comp}:=\varprojlim\limits_{n} \mathcal{D}ual(R^{+a}/\varpi^n).
  \]
Let $\mathfrak{m}:=\bigcup\limits_{n} \varpi^{1/p^n}R^+\subseteq R^+$, and consider the left adjoint
\[
  j_!(-)= - \otimes_{R^{+}} \mathfrak{m}\colon D(R^{+,a}) \to D(R^+)
\]
of the almostification functor $j^\ast$, which was introduced in \cite[2.2.19]{gabber_ramero_almost_ring_theory}.
Applying $j_!$ componentwise, then taking the inverse limit in $\mathcal{D}((R^+,R^+)_\bs)$, i.e., in the category defined via \cite[Theorem 3.28]{andreychev2021pseudocoherent}, and then inverting $\pi$ yields a natural functor
\[
  \alpha_R\colon \mathcal{D}ual(R^{+a})^{\rm comp}\to \mathcal{D}((R,R^+)_\bs).
\]
By \cite{lucas_mann_in_progress} (and the fact that $j_!(A\otimes_{R^{+a}/\varpi^n}^{\mathbb{L}} B)\cong j_!(A)\otimes^{\mathbb{L}}_{R^+/\varpi^n} j_!(B)$ for every $n\geq 0$) this functor is monoidal, i.e., the $\varpi$-completed tensor product agrees with the solid tensor product in this case. By \cite[Corollary 5.51.1]{andreychev2021pseudocoherent} the functor $\alpha_R$ has image in the full subcategory of perfect complexes over $R$ because by monoidality each object in the image is dualizable. By \cite{lucas_mann_in_progress}, i.e., compatibility of solid tensor products with the $\varpi$-completed tensor products appearing here, the functor $\alpha_R$ is compatible with pullbacks in $(R,R^+)$. From these considerations we can deduce that the perfect complex $\alpha_R(M^{+a})$ defines the desired descent of $N^{\bullet/Y}$. 
\end{proof}

\begin{lemma}
  \label{sec:desc-perf-compl-1-integral-models-commute-with-colimits}
  Let $(A,A^+)$ be a complete uniform Tate-Huber pair, and $M\in \mathcal{P}erf(A)$. Define the functor
  \[
   F_M\colon (B,B^+)\mapsto \{N^+\in \mathcal{P}erf(B^+) \textrm{ with an isomorphism } N^+\otimes_{B^+}^L B\cong M\otimes_A^L B\}
  \]
  on complete uniform Tate-Huber pairs over $(A,A^+)$ with values in spaces/anima. Then $F_M$ commutes with filtered colimits.
\end{lemma}
More formally, $F_M(B,B^+)$ is the fiber product
\[
  \mathcal{P}erf(B^+)^{\simeq }\times_{\mathcal{P}erf(B)^\simeq}\ast
\]
where $\ast\to \mathcal{P}erf(B)$ is the morphism determined by the object $M\otimes_A^L B\in\mathcal{P}erf(B)$, and $(-)^\simeq$ denotes the core of an $\infty$-category.
\begin{proof}
  Let $(B_i,B_i^+), i\in I,$ be a filtered system of complete uniform Tate-Huber $(A,A^+)$-algebras, let $(B,B^+)$ be their uncompleted colimit and $(C,C^+)$ be the (uniform) completion of $(B,B^+)$. By \cite[Proposition 5.6.(2)]{bhatt2014algebraization} the square
  \[
    \xymatrix{
      \mathcal{P}erf(B^+)\ar[r]\ar[d] & \mathcal{P}erf(B)\ar[d] \\
      \mathcal{P}erf(C^+)\ar[r] & \mathcal{P}erf(C)
    }
  \]
  is a cartesian diagram of $\infty$-categories. This implies that
  \[
    F_M(C,C^+)\cong \mathcal{P}erf(B^+)^\simeq\times_{\mathcal{P}erf(B)^\simeq}\ast
  \]
  with $\ast \to \mathcal{P}erf(B)$ induced by $M\otimes_A^L B$.
  But the right-hand side is equivalent to
  \[
    \varinjlim\limits_{i\in I} F_M(B_i,B_i^+)
  \]
  as desired because the functor $R\mapsto \mathcal{P}erf(R)$ on commutative rings commutes with filtered colimits.
\end{proof}

\subsection{Perfect complexes on relative Fargues-Fontaine curves}
\label{sec:v-descent-perfect-complexes-ff-curves}

The v-descent result of \Cref{subsection-v-descent-perfectoid-spaces} for perfect complexes on perfectoid spaces does not directly apply to relative Fargues-Fontaine curves, because these are not perfectoid spaces. However, we know that for any $T\in \Perf_{\mathbb{F}_q}$, the analytic adic spaces $X_T, Y_T$ are sousperfectoid, cf.\ \cite[Proposition 11.2.1]{scholze2020berkeley} and adapting the argument of \cite[Proposition 19.5.3]{scholze2020berkeley} we get the following result.

\begin{proposition}
  \label{sec:v-descent-perfect-descent-of-perfect-complexes-on-relative-ff-curve}
  Let $U_T\subseteq X_T$ or $U_T\subseteq Y_T$ be an open subset, and denote for any $T^\prime\to T$ by $U_{T^\prime}$ the pullback of $U_T$ along $X_{T^\prime}\to X_T$ or $Y_{T^\prime}\to Y_T$. Then the functor
  \[(T^\prime\to T)\mapsto \mathcal{P}erf(U_{T^\prime})
  \]
  on $\Perf_T$ satisfies v-descent.
\end{proposition}
\begin{proof}
Let $\widetilde{T}\to T$ be a v-cover in $\Perf$ with associated \v{C}ech nerve $\widetilde{T}^{\bullet/T}$, which is a simplicial object in $\Perf_T$. We want to show that the canonical morphism
\[
  \mathcal{P}erf(U_T)\to \varprojlim\limits_{\Delta} \mathcal{P}erf(U_{\widetilde{T}^{\bullet/T}})
\]
is an equivalence.
Commuting inverse limits we may assume that $U_T$ is affinoid.
Fix a uniformizer $\pi\in E$ and set
\[
  E_\infty:=\widehat{E(\pi^{1/p^{\infty}})},
\]
which is a perfectoid field. Set $S_\infty:=\Spa(E_\infty)$
and let $S_\infty^{\bullet/S}$ be the \v{C}ech nerve of $S_\infty\to S:=\Spa(E)$, a simplical adic space over $S$.
We obtain the bisimplicial adic space
\[
  U_{\widetilde{T}^{\bullet/T}}\times_S S_\infty^{\bullet/S}.
\]
over $S$.
Let us note that for $m\geq 1$ the space
\[
  S_\infty^{n/S}
\]
is not uniform (and thus in particular not sousperfectoid) because $S_\infty^{\bullet/S}$ was defined as the \v{C}ech nerve in all adic spaces over $S$ (and not just the uniform ones).
However, $S_\infty=S_\infty^{0/S}$ is a perfectoid space living over the totally disconnected space $\mathrm{Spa}(E_\infty)$ and
\[
  U_{\widetilde{T}^{\bullet/T}}\times_S S_\infty
\]
is a simplicial perfectoid space over $E$.
By \Cref{sec:desc-perf-compl-1-descent-of-perfect-complexes-on-perfectoid-spaces} we can deduce that
\[
      \mathcal{P}erf(U_T\times_S S^{m/S}_\infty)\to \varprojlim\limits_{\Delta} \mathcal{P}erf(U_{\widetilde{T}^{\bullet/T}}\times_S S^{m/S}_\infty)
    \]
    is an equivalence for $m=0$.

    We claim that it is fully faithful for $m\geq 1$. Indeed, write
    \[
      U_{\widetilde{T}^{n/T}}=\Spa(A_{\widetilde{T}^{n/T}},A_{\widetilde{T}^{n/T}}^+)
    \]
    and
    \[
      S^{m/S}_\infty=\Spa(E_{\infty,m/S},E^+_{\infty,m/S}),
    \]
    i.e., $E_{\infty,m/S}$ is the $m+1$-fold Banach space tensor product $E_\infty\hat{\otimes}_E \ldots \hat{\otimes}_E E_\infty$.
    By d\'evissage in perfect complexes it suffices to see that the natural complex
    \[
      0\to A_T\hat{\otimes}_E E_{\infty,m/S} \to A_{\widetilde{T}}\hat{\otimes}_E E_{\infty,m/S}\to (A_{\widetilde{T}}\hat{\otimes}_{A_T} A_{\widetilde{T}})\hat{\otimes}_E E_{\infty,m/S}\to \ldots 
    \]
    of $E$-Banach vector spaces is exact (here we use that the fiber product of affinoid adic spaces over $E$ is calculated via the Banach space tensor product). But this exactness follows from the case $m=0$ (where all components are perfectoid) and the fact that $-\hat{\otimes}_E M$ preserves exactness of complexes of $E$-Banach spaces for each Banach space $M$ over $E$. This finishes the proof of the desired fully faithfulness.

Furthermore, we claim that for each $n\in \N$ the canonical morphism
  \[
    \mathcal{P}erf(U_{\widetilde{T}^{n/T}})\to \varprojlim\limits_{\Delta} \mathcal{P}erf(U_{\widetilde{T}^{n/T}}\times_S S^{\bullet/S}_\infty)
  \]
  is an equivalence. To see it, let us consider the $\infty$-derived category $\mathcal{C}$ of solid $\underline{E}$-vector spaces endowed with the (derived) solid tensor product (cf. \cite{scholze_lectures_on_condensed_mathematics}). It forms a stable homotopy category in the sense of \cite[Definition 2.1]{mathew_the_galois_group_of_a_stable_homotopy_theory}. Since the map $E \to E_\infty$ admits an $E$-linear continuous splitting (by \cite[\S 2.7, Proposition 4]{bosch_guentzer_remmert_non_archimedean_analysis}, or simply since any Banach space over $E$ is orthonormalizable), the commutative algebra object $\underline{E_\infty}$ in $\mathcal{C}$ admits descent, cf. \cite[Proposition 3.20]{mathew_the_galois_group_of_a_stable_homotopy_theory} and the comment right after it. This gives descent for solid $\underline{E}$-vector spaces along the map $\underline{E} \to \underline{E_\infty}$, and thus also for perfect complexes, which are the dualizable objects in $\mathcal{C}$. Since the tensor product of $E$-Banach spaces agrees with the derived solid tensor product, this gives the desired claim. 
  
  Now we can finish the proof. By descent in the ``vertical direction'' which we proved above we can rewrite
\[
  \varprojlim\limits_{[n]\in \Delta}\mathcal{P}erf(U_{\widetilde{T}^{n/T}})
\]
as
\[
   \varprojlim\limits_{[n]\in \Delta}\varprojlim\limits_{[m]\in\Delta} \mathcal{P}erf(U_{\widetilde{T}^{n/T}}\times_S S_\infty^{m/S}),
 \]
 and thus as
 \begin{equation}
   \label{eq:1}
  \varprojlim\limits_{[m]\in \Delta}\varprojlim\limits_{[n]\in\Delta} \mathcal{P}erf(U_{\widetilde{T}^{n/T}}\times_S S_\infty^{m/S}),  
 \end{equation}
 by commuting the limits.

 Using \Cref{sec:v-descent-perfect-lemma-on-limits-of-inf-categories} below and the descent/resp.\ fully faithfulness in the ``horizontal direction'' discussed before we can therefore simplify 
 (\Cref{eq:1}) to
 \[
   \varprojlim\limits_{[m]\in \Delta} \mathcal{P}erf(U_{T}\times_S S_\infty^{m/S}),
 \]
 which is equivalent to $\mathcal{P}erf(U_T)$ by the same argument for $U_{\widetilde{T}}$ as above. This finishes the claim.
\end{proof}

We used the following lemma on limits of $\infty$-categories, cf.\ \cite[Lemma B.6]{bunke2019twisted}.

 \begin{lemma}
   \label{sec:v-descent-perfect-lemma-on-limits-of-inf-categories}
   Let $f_\bullet \colon C^\bullet \to D^\bullet$ be a morphism of two cosimplicial $\infty$-categories such that $f_0\colon C^0\to D^0$ is an equivalence and $f_i\colon C^i\to D^i$ is fully faithful for $i\geq 1$. Then
   \[
     \varprojlim\limits_{\Delta} C^\bullet\to \varprojlim\limits_{\Delta} D^\bullet 
   \]
   is an equivalence. 
 \end{lemma}

 We end this subsection with the observation that each perfect complex on a relative Fargues-Fontaine curve is, locally in the analytic topology on the base, ``strict''.

 \begin{proposition}
   \label{sec:perf-compl-relat-perfect-complex-strict}
   Let $T=\Spa(A,A^+)\in \Perf_{\mathbb{F}_q}$ be affinoid perfectoid. Then each perfect complex $K$ on $X_T$ is quasi-isomorphic to a bounded complex of vector bundles on $X_T$.
 \end{proposition}
 \begin{proof}
   Let $\varpi\in A$ be a pseudo-uniformizer. As in \cite[Theorem II.2.6]{fargues2021geometrization} this yields a radius function on the space $Y_T$ and (using similar notation as in \cite[Theorem II.2.6]{fargues2021geometrization}) thus we can write
   \[
     X_T\cong Y_{T,[1,q]}/\simeq
   \]
   With $\simeq$ coming from the identification $\varphi\colon Y_{T,[1,1]}\cong Y_{T,[q,q]}$. By analytic descent for perfect complexes on analytic adic spaces we obtain that
   \[
     \mathcal{P}erf(X_T)
   \]
   identifies with the $\infty$-categorical equalizer of the functors
   \[
     \mathcal{P}erf(Y_{T,[1,q]})\xrightarrow{\mathrm{res}}\mathcal{P}erf(Y_{T,[q,q]}),\ \mathcal{P}erf(Y_{T,[1,q]})\xrightarrow{\varphi\circ \mathrm{res}} \mathcal{P}erf(Y_{T,[q,q])}
   \]
   (here $\mathrm{res}$ denotes the respective restriction morphisms). The spaces
   \[
     Y_{T,[1,q]}, Y_{T,[1,1]}, Y_{T,[q,q]}
   \]
   are affinoid perfectoid, which implies that perfect complexes on them identify with perfect complexes over their respective coordinate rings. Thus, the proposition follows from the general observation \Cref{sec:perf-compl-relat-general-lemma-for-strictness} by setting $R=\mathcal{O}_{Y_{T,[1,q]}}, S=\mathcal{O}_{Y,[q,q]}$ and $f,g$ induced by restriction respectively restriction composed with Frobenius.
 \end{proof}

 Let $R,S$ be two (commutative) rings, and $f,g\colon R\to S$ two ring homomorphisms. Denote by $f^\ast, g^\ast$ the functors $S\otimes_{R,f}^{\mathbb{L}}-$, $S\otimes_{R,g}^{\mathbb{L}}-$ of (derived) base change along $f$ respectively $g$. Let
 \[
   \mathcal{C}
 \]
 be the $\infty$-categorical equalizer of the functors $\mathbb{L}f^\ast, \mathbb{L}g^\ast\colon \mathcal{D}(R)\to \mathcal{D}(S)$ between the $\infty$-derived categories of $R$- and $S$-modules. Thus, roughly, the objects of $\mathcal{C}$ are pairs $(K,\alpha_K)$ with $K\in\mathcal{D}(R)$ and $\alpha_K\colon \mathbb{L}f^\ast K\cong \mathbb{L}g^\ast K$ an isomorphism in $\mathcal{D}(S)$ while the morphisms in $\mathcal{C}$ are morphisms in $\mathcal{D}(R)$ together with homotopies in $\mathcal{D}(S)$ recording compatibility with the $\alpha_K$'s.

 Let us call an object $(K,\alpha_K) \in \mathcal{C}$ perfect if its image $K\in \mathcal{D}(R)$ is perfect, and let us call it strictly perfect if it can be written as a finite limit of objects $(M,\alpha_M)\in \mathcal{C}$ with $M$ a finite projective $R$-module. Note that $\mathcal{C}$ is a stable $\infty$-category, and the functor $\mathcal{C}\to \mathcal{D}(R)$ commutes with finite limits/colimits.

 Now we can prove the following general observation.
 
 \begin{proposition}
   \label{sec:perf-compl-relat-general-lemma-for-strictness}
   Assume $(K,\alpha_K)\in \mathcal{C}$ is perfect. Then $(K,\alpha_K)$ is strictly perfect.
 \end{proposition}
 The proof is a slight adaptation of \cite[Lemma 1.5.2]{kedlaya_liu_relative_p_adic_hodge_theory_foundations}.
 \begin{proof}
   We may assume that $H^i(K)=0$ for $i>0$. Then represent $K$ by a complex
   \[
     \ldots\to  K_{-2} \to K_{-1}\xrightarrow{d_{-1}} K_0\xrightarrow{d_0} K_1 \to \ldots 
   \]
   with $K_0, K_{-1}$ finite free $R$-modules, $K_{i}$ finite projective for all $i\in \Z$ and $K_{i}=0$ for $i\ll 0$ or $i>0$. The morphism $\alpha_K$ respectively $\alpha_K^{-1}$ is then represented by morphisms
   \[
     A_{i}\colon f^\ast K_i\to g^\ast K_i,\ \mathrm{resp.} B_i\colon g^\ast K_i\to f^\ast K_i
   \]
   (respecting the differential).
   As $\alpha_K\circ \alpha_K^{-1}=\Id_K$ (equality in the homotopy category), there exists
   \[
     h_i\colon g^\ast K_i\to g^\ast K_{i-1}
   \]
   such that
   \[
     A_i\circ B_i-\Id=g^\ast d_{i-1}\circ h_i+h_{i-1}\circ g^\ast d_i.
   \]
   In particular,
   \[
     A_0\circ B_0-\Id=g^\ast d_{-1}\circ h_0.
   \]
   Fix an isomorphism $R^n\cong K_0$. This yields identifications
   \[
     f^\ast K_0\cong S^n\cong g^\ast K_0
   \]
   and we let $\tilde{A}_0,\tilde{B}_0$ be the matrices representing $A_0, B_0$. Set $F:=K_0\oplus K_0$ and
   \[
     \alpha_F=
     \begin{pmatrix}
       \tilde{A}_0 & \tilde{A}_0\tilde{B}_0-\Id \\
       \Id & \tilde{B}_0
     \end{pmatrix}\colon f^\ast F\cong S^n\oplus S^n\to S^n\oplus S^n\cong g^\ast F.
   \]
   As
   \[
     \begin{pmatrix}
       \tilde{A}_0 & \tilde{A}_0\tilde{B}_0-\Id \\
       \Id & \tilde{B}_0
     \end{pmatrix}=\begin{pmatrix}
      \Id & \tilde{A}_0 \\
      0&  \Id 
     \end{pmatrix}
     \begin{pmatrix}
       0 & -\Id \\
       \Id & \tilde{B}_0
     \end{pmatrix}
   \]
   $\alpha_F$ is an isomorphism, and thus the pair $(F,\alpha_F)$ (with $F$ sitting in degree $0$) defines an object of $\mathcal{C}$. We now want to construct a morphism
   \[
     (F,\alpha_F)\to (K,\alpha_K).
   \]
   On underlying complexes we take the unique morphism of complexes which in degree $0$ is given by the first projection
   \[
     \varphi:=(\Id,0)\colon F=K_0\oplus K_0\to K_0.
   \]
   In order to upgrade $\varphi$ to a morphism in $\mathcal{C}$ we need to find a homotopy between the two compositions
   \[
     f^\ast F\to f^\ast K\xrightarrow{A_{\bullet}}g^\ast K_{\bullet}
   \]
   and
   \[
     f^\ast F\xrightarrow{\alpha_F} g^\ast F\to g^\ast F\to g^\ast K.
   \]
   Unravelling the definitions, we see that the composition
   \[
     f^\ast F\cong S^n\oplus S^n\xrightarrow{(\Id,0)} S^n\cong f^\ast K_0\xrightarrow{h_0} g^\ast K_{-1}
   \]
   works as
   \[
     A_0\circ B_0-\Id=g^\ast d_{-1}\circ h_0.
   \]
   Now the proposition follows by using induction on the amplitude of $K$.
 \end{proof}

 From the proof of \Cref{sec:perf-compl-relat-perfect-complex-strict} we can conclude that if $K\in \mathcal{P}erf(X_T)$ is a perfect complex of Tor-amplitude in $[a,b]$, then we can represent it by a complex if vector bundles, which is concentrated in degrees $[a,b]$.

\subsection{The moduli stack of flat coherent sheaves}
In this subsection, we define \textit{flat coherent sheaves} on relative Fargues-Fontaine curves, and show that they satisfy v-descent.

Let $S \in \mathrm{Perf}_{\mathbb{F}_q}$. An $\mathcal{O}_{X_S}$-linear map between two $\mathcal{O}_{X_S}$-modules is said to be \textit{fiberwise injective} if it is injective after base change along $X_{\mathrm{Spa}(C,C^+)} \to X_S$ for any geometric point $\mathrm{Spa}(C,C^+) \to S$ of $S$. Equivalently, it is an injective map which remains injective after base change along $X_{S'} \to X_S$ for any map $S' \to S$ in $\mathrm{Perf}_{\mathbb{F}_q}$. The same definition will be used when $X_S$ is replaced by an open subset of it.

\begin{definition}
\label{definition-flat-coherent-sheaf}
A \textit{flat coherent sheaf} on $X_S$ is an $\mathcal{O}_{X_S}$-module which can, locally for the analytic topology on $X_S$, be presented as the cokernel of a fiberwise injective map between two vector bundles on $X_S$. 
\end{definition}

By the remark after \Cref{sec:perf-compl-relat-general-lemma-for-strictness} we can globally represent any flat coherent sheaf on $X_T$ by a two-term complex  $\mathcal{E}_{-1}\xrightarrow{\alpha} \mathcal{E}_0$ of vector bundles on $X_T$. The morphism $\alpha$ is then automatically fiberwise injective.

\red{
\begin{remark}
\label{terminology-flat}
To justify the terminology we adopted, recall the following result in algebraic geometry\footnote{We thank David Hansen for drawing our attention to this point.}: if $S$ is a reduced Noetherian scheme, and $X \to S$ is a smooth projective relative curve, a coherent sheaf $\mathcal{F}$ on $X$ is $S$-flat (meaning that for each point $x \in X$ mapping to $s\in S$, the stalk $\mathcal{F}_x$ is flat over the ring $\mathcal{O}_{S,s}$) if and only if its fiberwise degree and generic rank functions are locally constant if and only if it is the cokernel of a fiberwise injective map between vector bundles. Indeed, the equivalence of the first two conditions is a particular case of a more general statement: if $S$ is a reduced Noetherian scheme and $f: X \to S$ a projective morphism, a coherent sheaf $\mathcal{F}$ on $X$ is $S$-flat if and only if the fiberwise Hilbert polynomial function is locally constant (this uses the characterization, \cite[Proposition 7.9.14]{egaiii2}, of $S$-flat coherent sheaves on $X$ as the coherent sheaves $\mathcal{F}$ on $X$ such that there exists $N\geq 0$ such that $f_\ast(\mathcal{F} \otimes \mathcal{L}^{\otimes n})$ is a locally free $\mathcal{O}_S$-module for all $n\geq N$, where $\mathcal{L}$ is a fixed ample line bundle for $f$ on $X$). The last condition clearly implies the second. Conversely, if $\mathcal{F}$ is a flat coherent sheaf on $X$, \cite[Proposition 2.1.10]{huybrechts2010geometry} provides a resolution
$$
0 \to \mathcal{E}_{-1} \to \mathcal{E}_0 \to \mathcal{F} \to 0
$$
of $\mathcal{F}$ by two vector bundles (using that $X \to S$ is smooth projective of relative dimension $1$), and the map $\mathcal{E}_{-1} \to \mathcal{E}_0$ has to be fiberwise injective, by flatness of $\mathcal{F}$ over $S$.   
\end{remark}
}

\begin{example}
\label{example-flat-coherent-sheaf}
Any vector bundle on $X_S$ is a flat coherent sheaf on $X_S$. If $S^\sharp$ is an untilt of $S$ over $E$, defining a Cartier divisor $i_{S^\sharp} : S^\sharp \to X_S$ on the Fargues-Fontaine curve (\cite[Proposition II.1.18, Proposition II.2.3]{fargues2021geometrization}), the ``skyscraper sheaf" $i_{S^\sharp,\ast} \mathcal{O}_{S^\sharp}$ is a flat coherent sheaf on $X_S$. 
\end{example}

The next result is an easy corollary of the work done in the previous subsections.

\begin{theorem}
\label{flat-coh-sheaves-form-a-v-stack}
The functor $\mathrm{Coh}^{\rm fl}$, sending $S\in \mathrm{Perf}_{\mathbb{F}_q}$ to the groupoid of flat coherent sheaves on $X_S$, is a small v-stack.
\end{theorem}

\begin{proof}
The smallness of $\Coh^{\mathrm{fl}}$ follows as in \cite[Proposition III.1.3]{fargues2021geometrization}.

Note now that for any $S\in \Perf_{\mathbb{F}_q}$, $\mathrm{Coh}^{\rm fl}(S)$ is a full subcategory of $\mathcal{P}erf^{[-1,0]}(X_S)$.

Let $T \to S$ be a v-cover in $\Perf_{\mathbb{F}_q}$ with associated \v{C}ech nerve
\[
  T^{\bullet/S},
\]
which is a simplicial object in $\Perf_S$.

We want to show that the canonical morphism
\[
  \mathrm{Coh}^{\rm fl}(S) \to \varprojlim\limits_{\Delta} \mathrm{Coh}^{\rm fl}(T^{\bullet/S})
\]
is an equivalence. We already know fully faithfulness and only need to prove essential surjectivity. Since we know by \Cref{sec:v-descent-perfect-descent-of-perfect-complexes-on-relative-ff-curve} that perfect complexes with Tor-amplitude in $[-1,0]$ satisfy v-descent, all we need to check is that fiberwise injectivity can be checked v-locally, which is clear.
\end{proof}    

We can extend the ampleness result \cite[Theorem II.2.6]{fargues2021geometrization} to flat coherent sheaves.

\begin{lemma}
  \label{sec:moduli-stack-flat-ampleness-for-flat-coherent-sheaves}
  Let $T\in \Perf_{\F_q}$ be affinoid perfectoid and let $\mathcal{F}\in \Coh^{\rm fl}(T)$ be a flat coherent sheaf on $X_T$. Then there exists an $n_0\geq  0$ such that for every $n\geq n_0$ there exists a presentation
  \[
    0\to \mathcal{E}_{-1}\to\mathcal{O}_{X_T}(-n)^m\to \mathcal{F}\to 0
  \]
  for some $m\geq 0$ with $\mathcal{E}_{-1}$ a vector bundle, and $H^1(X_T,\mathcal{F}(n))=0$.
\end{lemma}
\begin{proof}
  Fix a presentation $0\to \mathcal{F}_{-1}\to \mathcal{F}_0\to \mathcal{F}\to 0$ with $\mathcal{F}_{-1},\mathcal{F}_0$ two vector bundles. By choosing $n_0$ large enough we can assure that $H^1(X_T,\mathcal{F}_{-1}(n))=H^1(X_T,\mathcal{F}_0(n))=0$ and that $\mathcal{F}_0(n)$ is globlly generated for $n\geq 0$ (\cite[Theorem II.2.6]{fargues2021geometrization}.
  Then $H^1(X_T,\mathcal{F}(n))=0$ and we claim that for any surjection $\mathcal{O}_X(-n)^m\to \mathcal{F}_{0}$ the kernel $\mathcal{E}_{-1}$ of the composition $\mathcal{O}_X(-n)^m\to \mathcal{F}_{0}\to \mathcal{F}$ is a vector bundle. For this it suffices to show that $\mathcal{E}_{-1}$ is a perfect complex of Tor-amplitude $0$. Perfectness is clear, and assertion on Tor-amplitude follows because the dual of $\mathcal{F}$ is contained in $\Perf^{[0,1]}(X_T)$. This finishes the proof. 
\end{proof}

Since the rank and degree functions for vector bundles are additive, we can as well define the rank and degree functions (which are respectively functions from $|S|$ to $\mathbb{\Z}_{\geq 0}$ and $\Z$) for a perfect complex and thus in particular for a flat coherent sheaf. These functions are locally constant and for each pair of integers $(i,d) \in \Z_{\geq 0} \times \Z$, we will denote by 
$$
\Coh_{i,d}^{\rm fl}
$$
the open substack of $\Coh^{\rm fl}$ formed by flat coherent sheaves having (generic) rank $i$ and degree $d$.

\subsection{Further results on flat coherent sheaves}
\label{sec:further-results-flat-coherent-sheaves}
A flat coherent sheaf can by definition be presented as the cokernel of a fiberwise injective morphism between vector bundles. We will now prove a refinement of this, under assumptions on the slopes.

\begin{definition}
Let $S \in \Perf_{\mathbb{F}_q}$ and $\mathcal{F} \in \Coh^{\rm fl}(S)$. We say that $\mathcal{F}$ has \textit{non-negative slopes}, resp. \textit{positive slopes}, if its pullback along $X_{\mathrm{Spa}(C,C^+)} \to X_S$ for any geometric point $\mathrm{Spa}(C,C^+) \to S$ has only non-negative slopes, resp. only positive slopes (by convention, a torsion coherent sheaf on $X_{\mathrm{Spa}(C,C^+)}$ has slope $+\infty$). 
\end{definition}

For $i\geq 0$, $d\in \Z$, we will denote by
$$
\Coh_{i,d}^{\rm fl,\geq 0}, ~~ \mathrm{resp}. ~~ \Coh_{i,d}^{\rm fl,> 0},
$$
the substack of $\Coh_{i,d}^{\rm fl}$ formed by flat coherent sheaves having non-negative, resp. positive slopes.

\begin{proposition}
\label{presentation-flat-coh-sheaf-non-negative-slopes}
Let $S\in \Perf_{\mathbb{F}_q}$ and let $\mathcal{F} \in \Coh_{i,d}^{\rm fl,\geq 0}(S)$. There exists, v-locally on $S$, a short exact sequence
$$
0 \to \mathcal{O}_{X_S}(-1)^d \to \mathcal{O}_{X_S}^{i+d} \to \mathcal{F} \to 0.
$$ 
\end{proposition}
\begin{proof}
Let $Z_{i,d}$ be the moduli stack of fiberwise injective maps 
$$
\mathcal{O}_{X_S}(-1)^d \to \mathcal{E}
$$ 
with $\mathcal{E}$ a slope $0$ semi-stable vector bundle of rank $i+d$. There is a natural map
$$
f: Z_{i,d} \to \Coh_{i,d}^{\rm fl}
$$
sending a fiberwise injective map $\mathcal{O}_{X_S}(-1)^d \to \mathcal{E}$ to its cokernel $\mathcal{F}$. It suffices to show that $f$ is a v-cover.

We claim that $f$ is representable in smooth Artin v-stacks. Indeed, let $T\in \Perf_{\mathbb{F}_q}$ with a map $T \to \Coh_{i,d}^{\rm fl}$ corresponding to a flat coherent sheaf $\mathcal{F}$ on $X_T$ of generic rank $i$ and degree $d$, and let
$$
T^\prime := T \times_{\Coh_{i,d}^{\rm fl},f} Z_{i,d}
$$
be the fiber product. Then $T^\prime$ is a substack of the stack $\mathcal{E}xt_{X_T}^1(\mathcal{F},\mathcal{O}(-1)^d)$ sending $U \in \Perf_T$ to the groupoid of extensions
$$
0\to \mathcal{O}(-1)^d \to \mathcal{E} \to \mathcal{F} \to 0
$$
(which is indeed a v-stack, since $\Coh^{\rm fl}$ is). The condition that the extension $\mathcal{E}$ is a vector bundle is an open condition, since $\Bun$ is an open substack of $\Coh^{\rm fl}$. (To see it, one can argue as follows: the locus in $|X_S|$ where a flat coherent sheaf is locally free is open as it is a union of loci defined by the condition that a suitable minor of some matrix is invertible. Its complement is therefore closed, and since the map $|X_S| \to |S|$ is closed (\cite[Proposition II.1.21]{fargues2021geometrization}), its image in $|S|$ is also closed. Hence its complement is open.) Moreover, the condition that the vector bundle $\mathcal{E}$ is semi-stable of slope $0$ is an open condition, by upper semi-continuity of the Harder-Narasimhan polygon (\cite[Theorem II.2.19]{fargues2021geometrization}). Hence, $T^\prime$ is an open substack of $\mathcal{E}xt_{X_T}^1(\mathcal{F},\mathcal{O}(-1)^d)$. Therefore, to establish the claimed cohomological smoothness, it suffices to see that $\mathcal{E}xt_{X_T}^1(\mathcal{F},\mathcal{O}(-1)^d)$ is a smooth Artin v-stack over $T$. To do so, pick a presentation 
$$
0 \to \mathcal{E}_{-1} \to \mathcal{E}_0 \to \mathcal{F} \to 0
$$
with $\mathcal{E}_0 \cong \mathcal{O}_{X_T}(-n)^m$, with $n \gg 0$, $m\gg 0$, and $\mathcal{E}_{-1}$ a vector bundle (necessarily with slopes $\leq -n$), as we can by \Cref{sec:moduli-stack-flat-ampleness-for-flat-coherent-sheaves}. We deduce, for $U \in \Perf_T$, a long exact sequence
$$
0 \to \mathcal{H}om_{X_U}(\mathcal{F},\mathcal{O}(-1)^d) \to \mathcal{H}om_{X_U}(\mathcal{E}_0,\mathcal{O}(-1)^d) \to \mathcal{H}om_{X_U}(\mathcal{E}_{-1},\mathcal{O}(-1)^d) 
$$
$$
\to \mathcal{E}xt_{X_U}^1(\mathcal{F},\mathcal{O}(-1)^d) \to \mathcal{E}xt_{X_U}^1(\mathcal{E}_0,\mathcal{O}(-1)^d). 
$$
For $n \gg 0$, the last term is zero, while $\mathcal{H}om_{X_U}(\mathcal{E}_{i},\mathcal{O}(-1)^d)$ for $i=-1,0$ is the Banach-Colmez space attached to a vector bundle with positive slopes, hence is cohomologically smooth by \cite[Proposition III.3.5]{fargues2021geometrization}. Thus, the above sequence expresses the stack $\mathcal{E}xt_{X_U}^1(\mathcal{F},\mathcal{O}(-1)^d)$ as the quotient of a cohomologically smooth Banach-Colmez space by the action of another cohomologically smooth Banach-Colmez space.

As $f$ is cohomologically smooth, it is open (\cite[Proposition 23.11]{scholze_etale_cohomology_of_diamonds}), and thus a v-cover onto its image. The image is determined on underlying topological spaces and hence we may reduce to the case that $U=\Spa(C,\mathcal{O}_C)$ is a geometric point. By looking at the slope of the last quotient of the Harder-Narasimhan filtration it is clear that the image of $f$ is contained in $\Coh^{\mathrm{fl},\geq 0}_{i,d}$. For the converse, assume that $\mathcal{F}$ is a coherent sheaf with non-negative slopes on the Fargues-Fontaine curve $X_C$ attached to the complete algebraically closed field $C$ of characteristic $p$. It suffices to treat the case of a vector bundle with non-negative slopes or $\mathcal{F}=i_{x,\ast} B_\dR^+(C_x)/t_x^n$, for some classical point $x$ of $X_C$ corresponding to some untilt $C_x$ of $C$ and some $n>0$. Here, $t_x\in H^0(X_C,\mathcal{O}_{X_C}(1))$ is some section with vanishing locus $x$. The first case follows now from \cite[Theorem II.3.1]{fargues2021geometrization}. In the second case, choose another section $t\in H^0(X_C,\mathcal{O}_{X_C}(1))$ such that $t_x,t$ generate $\mathcal{O}_{X_C}(1)$. The pair $(t_x,t)$ define a morphism $\alpha\colon X_C\to \mathbb{P}^1_E$ of locally ringed spaces mapping $x$ to the $E$-rational point $y:=[0:1]\in \mathbb{P}^1_E$. Now, let $\mathcal{G}$ be a skyscraper sheaf on $\mathbb{P}^1_E$ supported at $y$, and consider the canonical surjection
  \[
    \mathcal{O}_{\mathbb{P}^1_{E}}\otimes_E H^0(\mathbb{P}^1_E,\mathcal{G})\to \mathcal{G}.
  \]
  Its kernel is a vector bundle of rank the degree of $\mathcal{G}$, with only non-positive Harder-Narasimhan slopes and without global sections, and thus isomorphic to $\mathcal{O}_{\mathbb{P}^1_E}(-1)$. Pulling back such a sequence (for the obvious choice of $\mathcal{G}$) along $\alpha$ yields the desired presentation of $i_{x,\ast} B_\dR^+(C_x)/t_x^n$. This finishes the proof. 
\end{proof}

\begin{corollary}
\label{presentation-flat-coh-sheaf-positive-slopes}
Let $S\in \Perf_{\mathbb{F}_q}$ and let $\mathcal{F} \in \Coh^{\rm fl,> 0}(S)$. There exists, v-locally on $S$, a short exact sequence
$$
0 \to \mathcal{O}_{X_S}^d \to \mathcal{E} \to \mathcal{F} \to 0,
$$
where $\mathcal{E}$ is a semi-stable vector bundle of positive slope. 
\end{corollary}
\begin{proof}
The statement being local on $S$, we can assume that $S$ is affinoid perfectoid and that $\mathcal{F}$ has constant rank $r$ and degree on $S$. We argue as in the proof of \cite[Corollary II.3.3]{fargues2021geometrization} to deduce the statement from \Cref{presentation-flat-coh-sheaf-non-negative-slopes}. Since $\mathcal{F}$ has constant rank $r$ and only positive slopes, we see that all slopes of $\mathcal{F}$ at all geometric points of $S$ are $\geq 1/r$. Let $\pi_r$ be the natural map from the Fargues-Fontaine curve $X_{S,r}$ attached to $S$ and the degree $r$ unramified extension $E_r$ of $E$ to $X_S$ (the Fargues-Fontaine curve for $S$ and $E$). We apply \Cref{presentation-flat-coh-sheaf-non-negative-slopes} to $\pi_r^\ast \mathcal{F}(-1)$. Locally on $S$, we get a short exact sequence of $\mathcal{O}_{X_{S,r}}$-modules
$$
0 \to \mathcal{O}_{X_{S,r}}^{d'} \to \mathcal{E}^\prime \to \pi_r^\ast \mathcal{F} \to 0,
$$
with $\mathcal{E}^\prime$ semi-stable of slope $1$. Applying $\pi_{r,\ast}$, we get a short exact sequence
$$
0 \to \mathcal{O}_{X_S}^{rd'} \to \pi_{r,\ast} \mathcal{E}^\prime \to \pi_{r,\ast} \pi_r^\ast \mathcal{F} \to 0. 
$$
Since $\mathcal{F}$ is a direct summand of $\pi_{r,\ast} \pi_r^\ast \mathcal{F}$, we get the desired exact sequence by pullback.
\end{proof}

\begin{remark}
  \label{the-stack-coh-positive-slopes-is-smooth}
  
For $d\geq 1$, let $W_d$ be the moduli stack of fiberwise injective maps 
$$
\mathcal{O}_{X_S}^d \to \mathcal{E}
$$ 
with $\mathcal{E}$ semi-stable vector bundle of positive slope. There is a natural map
$$
g_d: W_d \to \Coh^{\rm fl}
$$
sending a fiberwise injective map $\mathcal{O}_{X_S}^d \to \mathcal{E}$ to its cokernel $\mathcal{F}$, which is shown to be cohomologically smooth as in the proof of \Cref{presentation-flat-coh-sheaf-non-negative-slopes}. Hence the map
$$
g : W =\sqcup_{d\geq 1} W_d \to \Coh^{\rm fl, >0},
$$
which is $g_d$ in restriction to $W_d$ for each $d\geq 1$, is cohomologically smooth, too. It is also surjective by the last proposition. Moreover, for each $d\geq 1$, one has a map
$$
W_d \to \Bun^{\rm ss, >0}
$$
to the moduli stack of semi-stable vector bundles with positive slope, sending $\mathcal{O}_{X_S}^d \to \mathcal{E}$ to $\mathcal{E}$. This map is cohomologically smooth (since its fibers are open\footnote{To check openness, it is enough, by taking $d$th exterior powers, to argue when $d=1$ and then the claim is easy.} in positive slope Banach-Colmez spaces) and has a cohomologically smooth target. Therefore, $W_d$ for each $d\geq 1$, and hence also $W$, is cohomologically smooth. 

We have exhibited a cohomologically smooth cover of the small v-stack $\Coh^{\rm fl,>0}$ by a cohomologically smooth v-stack. Hence, if we can prove that the diagonal of the small v-stack $\Coh^{\rm fl,>0}$ is representable in locally spatial diamonds, we will have shown that $\Coh^{\rm fl, >0}$ is a cohomologically smooth Artin v-stack, in the sense of \cite[\S IV]{fargues2021geometrization}.

\red{It remains to prove the assertion about the diagonal of $\Coh^{\rm fl,>0}$. Let $S \in \Perf$ and let $\mathcal{F}, \mathcal{F}^\prime$ be two flat coherent sheaves on $X_S$. The sheaf of isomorphisms $\mathcal{I}som(\mathcal{F}, \mathcal{F}^\prime)$ is relatively representable by an open subspace of the sheaf $\mathcal{H}om(\mathcal{F},\mathcal{F}^\prime)$ (use that the map $|X_T| \to |T|$ is closed for every $T\in \Perf$), so it is enough to prove that the latter is a locally spatial diamond. Applying \Cref{sec:moduli-stack-flat-ampleness-for-flat-coherent-sheaves} to $\mathcal{F}$, we see that it is even enough to prove that $\mathcal{H}om(\mathcal{F},\mathcal{F}^\prime)$ is a locally spatial diamond when $\mathcal{F}$ is a vector bundle (as fiber products exist in the category of locally spatial diamonds). Hence, up to replacing $\mathcal{F}^\prime$ by $\mathcal{F}^\prime \otimes \mathcal{F}^{-1}$, it suffices to show that the functor sending $T\in \Perf_S$ to $H^0(X_T, \mathcal{F}^\prime)$ is a locally spatial diamond if $\mathcal{F}^\prime$ is a flat coherent sheaf on $X_S$. Applying \Cref{sec:moduli-stack-flat-ampleness-for-flat-coherent-sheaves} again, the statement is a special case of \cite[Proposition II.3.5 (i)]{fargues2021geometrization}. }
\end{remark}

\section{The Fourier transform}
\label{sec:a-fourier-transform-for-bc-spaces}

  In this section, we define and study \textit{very nice stacks in $E$-vector spaces} and their Fourier transform. Pro-\'etale $\underline{E}$-local systems are examples of very nice stacks in $E$-vector spaces, as are Banach-Colmez spaces attached to flat coherent sheaves with only positive slopes. Proving this last fact requires some preliminary Ext-groups computations, which we address first.

\subsection{Some Ext-group computations}
\label{sec:some-ext-groups-computation}

The main result of this subsection is \Cref{computations-of-the-desired-extensions} below, which describes some (local) Ext's in the category of v-sheaves of $\underline{E}$-vector spaces. This generalizes and improves on the results of \cite{le_bras_espaces_de_banach_colmez_et_faisceaux_coherents_sur_la_courbe_de_fargues_fontaine}. Contrary to \cite{le_bras_espaces_de_banach_colmez_et_faisceaux_coherents_sur_la_courbe_de_fargues_fontaine}, which used a truncated version of the Breen-Deligne resolution and was limited to low degrees and the case $E=\Q_p$, the main idea here is to prove a statement about self-extensions of the sheaf $\mathbb{A}_{\rm inf}$ (recalled below), which implies the desired results but has the advantage of being reducible to old deep results of Breen (\cite{breen_extensions_du_groupe_additif_sur_le_site_parfait}). 

When $\mathcal{S}$ is a site and $\Xi$ a sheaf of rings on $\mathcal{S}$, we will use the notation $$R\mathcal{H}om_{\mathcal{S},\Xi}(-,-)$$ to denote the derived internal Hom in the category of sheaves of $\Xi$-modules on $\mathcal{S}$\footnote{When $\mathcal{S}$ is the v-site of a small v-stack and $\Xi$ a topological ring, we will even write $R\mathcal{H}om_{\mathcal{S},\Xi}(-,-)$ instead of $R\mathcal{H}om_{\mathcal{S},\underline{\Xi}}(-,-)$, to keep the notation light.}.
\\

Before starting the computation, let us recall the following general result. 

\begin{theorem}
\label{definition-and-existence-of-the-maclane-complex}
Let $(\mathcal{T},\Xi)$ be a ringed topos, and let $P$ be a $\Xi$-module in $\mathcal{T}$. There exists a complex $M_{(\mathcal{T},\Xi)}(P)$ of $\Xi$-modules in $\mathcal{T}$ with an augmentation to $P$, functorial in $P$ and called \textit{the MacLane complex of $P$}, with the following two properties:
\begin{itemize}
\item The augmentation $\epsilon: M_{(\mathcal{T},\Xi)}(P) \to P$ is a quasi-isomorphism, i.e. $M_{(\mathcal{T},\Xi)}(P)$ is a resolution of $P$.
\item Each component $M_{(\mathcal{T},\Xi)}(P)_i$ of the complex is of the form $\Xi[\Xi^s \times P^t]/\Xi[0]$, where $s, t$ are integers depending on the integer $i$.
\end{itemize}
\end{theorem}
\begin{proof}
The construction is explained in \cite[\S 3]{breen_extensions_du_groupe_additif_sur_le_site_parfait}, and is essentially due to MacLane.
\end{proof}

If $R$ is a perfect characteristic $p$ ring, we let $R[F^{\pm 1}]$
be the ring of \textit{non-commutative} polynomials in one variable $F$ over $R$, with multiplication given by
$$
Fa = \varphi(a) F
$$
for $a\in R$, where $\varphi$ denotes the $q$-Frobenius on $R$ (recall that $q$ was fixed once and for all, as the cardinality of the residue field of $E$). We denote by $\mathrm{Spec}(\mathbb{F}_q)_{\rm perf}$ the \textit{algebraic} perfect v-site of $\mathrm{Spec}(\mathbb{F}_q)$.

\begin{theorem}[Breen]
\label{main-theorem-of-breen-perfect-site}
The natural map, sending $F$ to the $q$-Frobenius $\varphi$ on $\mathbb{G}_a$ and $\mathbb{F}_q$ to its action on the right factor,
$$
\mathbb{F}_q[F^{\pm 1}] \to R\mathrm{Hom}_{\mathrm{Spec}(\mathbb{F}_q)_{\rm perf},\mathbb{F}_q}(\mathbb{G}_a,\mathbb{G}_a)
$$
is an isomorphism.
\end{theorem}
\begin{proof}
In \cite[Th\'eor\`eme 0.1, \S 1.4]{breen_extensions_du_groupe_additif_sur_le_site_parfait}, Breen has computed self-extensions of $\mathbb{G}_a$ seen as a sheaf of $\mathbb{F}_p$-vector spaces (not $\mathbb{F}_q$-vector spaces!) on the site of all perfect schemes over $\mathrm{Spec}(\mathbb{F}_q)$ endowed with the \'etale topology. The MacLane resolution from \Cref{definition-and-existence-of-the-maclane-complex}, together with the fact that $\mathbb{G}_a$ is represented by the perfect affine line and \cite[Theorem 4.1]{bhatt_scholze_projectivity_of_the_witt_vector_affine_grassmannian}, show that these extension groups are the same when computed on the site considered by Breen and on $\mathrm{Spec}(\mathbb{F}_q)_{\rm perf}$. Hence, we deduce from Breen's result that the natural map
$$
\mathbb{F}_q[F_p^{\pm 1}] \to R\mathrm{Hom}_{\mathrm{Spec}(\mathbb{F}_q)_{\rm perf},\mathbb{F}_p}(\mathbb{G}_a,\mathbb{G}_a)
$$
is an isomorphism, where on the left $\mathbb{F}_q[F_p^{\pm 1}]$ is the ring of non-commutative polynomials in one variable $F_p$ over $\mathbb{F}_q$, with multiplication given by
$$
F_pa = a^p F_p
$$
for $a\in \mathbb{F}_q$.
Adjunction gives an isomorphism
$$
 R\mathrm{Hom}_{\mathrm{Spec}(\mathbb{F}_q)_{\rm perf},\mathbb{F}_q}(\mathbb{G}_a \otimes_{\mathbb{F}_p} \mathbb{F}_q,\mathbb{G}_a)  \cong R\mathrm{Hom}_{\mathrm{Spec}(\mathbb{F}_q)_{\rm perf},\mathbb{F}_p}(\mathbb{G}_a,\mathbb{G}_a).
 $$
The tensor product $\mathbb{G}_a \otimes_{\mathbb{F}_p} \mathbb{F}_q$ decomposes as a direct sum
$$
\mathbb{G}_a \otimes_{\mathbb{F}_p} \mathbb{F}_q = \bigoplus_{i=0}^{f-1} \mathbb{G}_a^{(i)},
$$
where $q=p^f$ and $\mathbb{G}_a^{(i)}$, $0\leq i <f$, is the sheaf sending a perfect $\mathbb{F}_q$-algebra $R$ to the $\mathbb{F}_q$-vector obtained by twisting the $\mathbb{F}_q$-action on the $\mathbb{F}_q$-vector space $R$ by the $i$-th power of the Frobenius on $\mathbb{F}_q$. We have $\mathbb{G}_a^{(0)}=\mathbb{G}_a$. In particular, we already see that necessarily
$$
 \mathrm{Ext}_{\mathrm{Spec}(\mathbb{F}_q)_{\rm perf},\mathbb{F}_q}^k(\mathbb{G}_a,\mathbb{G}_a) =0
$$
for all $k>0$. Moreover, for each $i=0,\dots,f-1$, right composition by $F_p^{f-i}$ induces an isomorphism between $\mathrm{Hom}_{\mathrm{Spec}(\mathbb{F}_q)_{\rm perf},\mathbb{F}_q}(\mathbb{G}_a^{(0)},\mathbb{G}_a)$ and $\mathrm{Hom}_{\mathrm{Spec}(\mathbb{F}_q)_{\rm perf},\mathbb{F}_q}(\mathbb{G}_a^{(i)},\mathbb{G}_a)$. Hence, the decomposition
$$
\mathrm{Hom}_{\mathrm{Spec}(\mathbb{F}_q)_{\rm perf},\mathbb{F}_p}(\mathbb{G}_a,\mathbb{G}_a) = \bigoplus\limits_{i=0}^f \mathrm{Hom}_{\mathrm{Spec}(\mathbb{F}_q)_{\rm perf},\mathbb{F}_q}( \mathbb{G}_a^{(i)},\mathbb{G}_a)  
$$
from above corresponds to the decomposition
$$
\mathbb{F}_q[F_p^{\pm 1}] = \bigoplus\limits_{i=0}^f \mathbb{F}_q[(F_p^f)^{\pm 1}]. F_p^i.
$$
Since $\mathbb{F}_q[(F_p^f)^{\pm 1}]=\mathbb{F}_q[F^{\pm 1}]$, this concludes the proof.
\end{proof}

We denote by $\mathcal{O}^+ \langle F^{\pm 1}\rangle$ the v-sheaf sending $\mathrm{Spa}(R,R^+) \in \Perf_{\mathbb{F}_q}$ to the $\varpi$-adic completion of $R^+[F^{\pm 1}]$, where $\varpi$ is any pseudo-uniformizer of $R$. There is a natural map
$$
\mathcal{O}^+ \langle F^{\pm 1} \rangle \to \mathcal{H}om_{S_{v},\mathbb{F}_q}(\mathcal{O}^+,\mathcal{O}^+)
$$
which sends $F$ to $\varphi_{R^+}$ and $\mathcal{O}^+$ to its action on the \textit{right} factor of $\mathcal{O}^+$.

\begin{proposition}
\label{self-extensions-o-plus}
The natural map
$$
\mathcal{O}^+ \langle F^{\pm 1} \rangle \to R\mathcal{H}om_{\mathrm{Spa}(\mathbb{F}_q)_{v},\mathbb{F}_q}(\mathcal{O}^+,\mathcal{O}^+)
$$
is an almost isomorphism (for the left action of $\mathcal{O}^+$ on $\mathcal{O}^+\langle F^{\pm 1}\rangle$ and the left action of $\mathcal{O}^+$ on $R\mathcal{H}om_{\mathrm{Spa}(\mathbb{F}_q)_{v},\mathbb{F}_q}(\mathcal{O}^+,\mathcal{O}^+)$ via multiplication on the second factor).
\end{proposition}
\begin{proof}
Fix an affinoid perfectoid space $S=\mathrm{Spa}(R,R^+)$ in $\Perf_{\mathbb{F}_q}$, and a pseudo-uniformizer $\varpi$ of $R$. Recall the morphism of sites
$$
f: \mathrm{Spa}(\mathbb{F}_q)_{v} \to \mathrm{Spec}(\mathbb{F}_q)_{\rm perf},
$$
from the perfectoid v-site to the algebraic v-site. The pullback of the sheaf $\mathbb{G}_a$ on $\mathrm{Spec}(\mathbb{F}_q)_{\rm perf}$ (representend by the perfect affine line over $\mathbb{F}_q$) along $f$ is the sheaf $\mathcal{O}^+$. We can further pullback along the natural morphism
$$
S_{v} \to \mathrm{Spa}(\mathbb{F}_q)_{v} .
$$
Since all these operations are exact, we deduce a map
$$
R\mathrm{Hom}_{\mathrm{Spec}(\mathbb{F}_q)_{\rm perf}, \mathbb{F}_q}(\mathbb{G}_a,\mathbb{G}_a) \to R\mathrm{Hom}_{S_{v},\mathbb{F}_q}(\mathcal{O}^+,\mathcal{O}^+).
$$
Since the right-hand side is an $R^+$-module (via the second factor) and is derived $\varpi$-adically complete, this extends to a map
$$
(R\mathrm{Hom}_{\mathrm{Spec}(\mathbb{F}_q)_{\rm perf}, \mathbb{F}_q}(\mathbb{G}_a,\mathbb{G}_a) \otimes_{\mathbb{F}_q} R^+)^{\wedge_\varpi} \to R\mathrm{Hom}_{S_{v},\mathbb{F}_q}(\mathcal{O}^+,\mathcal{O}^+).
$$
We claim that this map is an almost isomorphism. To compute both sides, we can use the MacLane resolution from \Cref{definition-and-existence-of-the-maclane-complex}, once for the topos $\widetilde{\mathrm{Spec}(\mathbb{F}_q)_{\rm perf}}$ of sheaves on $\mathrm{Spec}(\mathbb{F}_q)_{\rm perf}$ and the constant sheaf of rings $\mathbb{F}_q$ and once for the topos $\widetilde{S_v}$ of sheaves on $S_v$ and the constant sheaf of rings $\mathbb{F}_q$. From the explicit shape of the MacLane complexes and the fact (\cite[Proposition 8.8]{scholze_etale_cohomology_of_diamonds}) for all $m\geq 0, n\geq 1, i>0$,
$$
H_v^i(\mathrm{Spec}(\mathbb{F}_q[T_1^{1/p^{\infty}},\dots,T_n^{1/p^{\infty}}]),\mathbb{G}_a) = 0 ~~ , ~~ H_v^i(\mathbb{B}_{S}^n,\mathcal{O}^+) \overset{a} =  0,
$$
we deduce that the above map, seen as a map between objects of the derived category of almost $R^+$-modules, can be rewritten as
$$
(\mathrm{Hom}_{\mathrm{Spec}(\mathbb{F}_q)_{\rm perf}, \mathbb{F}_q}(M_{(\widetilde{\mathrm{Spec}(\mathbb{F}_q)_{\rm perf}},\mathbb{F}_q)}(\mathbb{G}_a),\mathbb{G}_a) \otimes_{\mathbb{F}_q} R^+)^{\wedge_\varpi} \to \mathrm{Hom}_{S_{v},\mathbb{F}_q}(M_{(\widetilde{S_v},\mathbb{F}_q)}(\mathcal{O}^+),\mathcal{O}^+).
$$
Hence, the claim follows if for each $n\geq 1$, the natural map
$$
(H_v^0(\mathrm{Spec}(\mathbb{F}_q[T_1^{1/p^{\infty}},\dots,T_n^{1/p^{\infty}}]),\mathbb{G}_a) \otimes_{\mathbb{F}_q} R^+)^{\wedge_\varpi} \to H_v^0(\mathbb{B}_{S}^n,\mathcal{O}^+)
$$
is an isomorphism. But this map is just the natural map
$$
(\mathbb{F}_q[T_1^{1/p^{\infty}},\dots,T_n^{1/p^{\infty}}] \otimes_{\mathbb{F}_q} R^+)^{\wedge_\varpi} \to R^+\langle T_1^{1/p^{\infty}},\dots,T_n^{1/p^{\infty}} \rangle
$$
so the assertion is true. Therefore, the proposition follows from \Cref{main-theorem-of-breen-perfect-site}.
\end{proof}

  We can get the following ``honest'' version. Here, we set
  \[
    \mathcal{O}\langle F^{\pm 1}\rangle :=\mathcal{O}\otimes_{\mathcal{O}^+}\mathcal{O}^+\langle F^{\pm 1}\rangle.
  \]

\begin{corollary}
\label{extensions-o-plus-and-o}
The natural map
$$
\mathcal{O} \langle F^{\pm 1} \rangle \to R\mathcal{H}om_{\mathrm{Spa}(\mathbb{F}_q)_{v},\mathbb{F}_q}(\mathcal{O}^+,\mathcal{O})
$$
is an isomorphism.
\end{corollary}
\begin{proof}
By \Cref{self-extensions-o-plus}, we only need to justify that for each affinoid perfectoid space $S=\Spa(R,R^+)$ and pseudo-uniformizer $\varpi\in R$, the natural map
$$
 \underset{\times \varpi} \varinjlim ~ R\mathcal{H}om_{S_v,\mathbb{F}_q}(\mathcal{O}^+,\mathcal{O}^+) \to R\mathcal{H}om_{S_v,\mathbb{F}_q}(\mathcal{O}^+,\mathcal{O})
$$
is an isomorphism. This can be checked on cohomology groups and thus it suffices to show that $\mathcal{O}^+$ is pseudo-coherent as a v-sheaf of $\mathbb{F}_q$-vector spaces. However, for each $i\geq 0$, the functors $\mathcal{E}xt_{S_v,\mathbb{F}_q}^i(\mathcal{O}^+,-)$ can be computed using the MacLane complex (\Cref{definition-and-existence-of-the-maclane-complex}) $M_{(\widetilde{S_v},\mathbb{F}_q)}(\mathcal{O}^+)$ and because of the description of the terms of this complex, to check that $\mathcal{E}xt_{S_v,\mathbb{F}_q}^i(\mathcal{O}^+,-)$ commutes with filtered colimits for all $i$, it suffices to prove that the functors $H^j((\mathcal{O}^+)^s \times \mathbb{F}_q^r,-)$ commute with filtered colimits for all $j,r,s\geq 0$: this is true, since $(\mathcal{O}^+)^s \times \mathbb{F}_q^r$ is represented by a \textit{qcqs} perfectoid space.
\end{proof}

If $(R,R^+)$ is a perfectoid Tate algebra over $\mathbb{F}_q$, we will denote
$$
A_{\rm inf}(R,R^+) = W_{\mathcal{O}_E}(R^+), \quad A(R,R^+)=W_{\mathcal{O}_E}(R).
$$
By definition of Witt vectors, there are bijections
$$
A_{\rm inf}(R,R^+) \cong (R^+)^{\mathbb{N}}, \quad A(R,R^+) \cong R^{\mathbb{N}}
$$
and we define the topology on $A_{\rm inf}(R,R^+)$, resp. on $A(R,R^+)$, as the product topology of the natural topology on $R^+$, resp. $R$. Fix a pseudo-uniformizer $\varpi$ of $R$. Then a basis of neighborhoods of $0$ in $A_{\rm inf}(R,R^+)$, resp. on $A(R,R^+)$, are the $\pi^r A_{\rm inf}(R,R^+) + [\varpi^s] A_{\rm inf}(R,R^+)$, resp. the $\pi^r A(R,R^+) + [\varpi^s] A_{\rm inf}(R,R^+)$, $r,s \geq 0$ (on the ring $A_{\rm inf}(R,R^+)$, we are just considering the $(\pi,[\varpi])$-adic topology).

\begin{definition}
\label{def-the-sheaf-a-inf}
We define the sheaves 
$$
\mathbb{A}_{\rm inf} = W_{\mathcal{O}_E}(\mathcal{O}^{+}), \quad \mathbb{B}_{\rm inf} = \mathbb{A}_{\rm inf}[1/\pi],
$$
and  
$$
\mathbb{A} = W_{\mathcal{O}_E}(\mathcal{O}), \quad \mathbb{B} = \mathbb{A}[1/\pi],
$$
on $\mathrm{Spa}(\mathbb{F}_q)_v$. We note that for each perfectoid pair $(R,R^+)$ over $\mathbb{F}_q$, there are canonical isomorphisms
$$
\mathbb{A}_{\rm inf}(\mathrm{Spa}(R,R^+)) \cong A_{\rm inf}(R,R^+), \quad \mathbb{A}(\mathrm{Spa}(R,R^+)) \cong A(R,R^+),
$$
cf. \cite[Theorem 6.5]{scholze_p_adic_hodge_theory_for_rigid_analytic_varieties}. We also define the sheaf 
$$
\mathbb{A}_{\rm inf}\langle F^{\pm 1} \rangle
$$
by sending $S=\mathrm{Spa}(R,R^+)$ affinoid perfectoid in $\Perf_{\mathbb{F}_q}$ to the ring $A_{\rm inf}(R,R^+)\langle F^{\pm 1} \rangle$
which is by definition the completion of the ring $A_{\rm inf}(R,R^+)[F^{\pm 1}]$
of non-commutative polynomials in one variable $F$ over $A_{\rm inf}(R^+)$, with multiplication given by
$$
Fa = \varphi(a) F
$$
for $a\in A_{\rm inf}(R^+)$, where $\varphi$ denotes the Frobenius on $A_{\rm inf}(R^+)$, with respect to the $(\pi, [\varpi])$-adic topology. 
\end{definition}

There is a natural map of left $\mathbb{A}_{\rm inf}$-modules
$$
\mathbb{A}_{\rm inf}\langle F^{\pm 1} \rangle \to \mathcal{H}om_{\mathrm{Spa}(\mathbb{F}_q)_v,\mathcal{O}_E}(\mathbb{A}_{\rm inf},\mathbb{A}_{\rm inf})
$$
which sends $F$ to $\varphi_{\mathbb{A}_{\rm inf}}$ and $\mathbb{A}_{\rm inf}$ to its action on the right factor, and a natural map of left $\mathbb{A}$-modules
$$
\mathbb{A}[ F^{\pm 1}] \to \mathcal{H}om_{\mathrm{Spa}(\mathbb{F}_q)_v,\mathcal{O}_E}(\mathbb{A}_{\rm inf},\mathbb{A})
$$
(with the similar definition of the non-commutative ring $A[F^{\pm 1}]$)
which sends $F$ to $\varphi_{\mathbb{A}}$, and $\mathbb{A}$ to its action on the right factor.

\begin{proposition}
\label{self-ext-of-a-inf}
The natural map
$$
\mathbb{A}_{\rm inf}\langle F^{\pm 1} \rangle \to R\mathcal{H}om_{\mathrm{Spa}(\mathbb{F}_q)_v,\mathcal{O}_E}(\mathbb{A}_{\rm inf},\mathbb{A}_{\rm inf})
$$
is an almost\footnote{Recall that an $\mathbb{A}_{\rm inf}$-module is almost zero if it is killed $[\varpi]$, for any local choice of a pseudo-uniformizer $\varpi$.} isomorphism of left $\mathbb{A}_{\rm inf}$-modules.
 \end{proposition}
\begin{proof}
  Both sides being (derived) $\pi$-adically complete, it suffices to check the assertion modulo $\pi$. The left hand side becomes $\mathcal{O}^+\langle F^{\pm 1} \rangle$. For the right hand side, recall that $\mathbb{A}_{\rm inf}/\pi=\mathcal{O}^+$. We then have
    \begin{align*}
      R\mathcal{H}om_{\mathrm{Spa}(\mathbb{F}_q)_v,\mathcal{O}_E}(\mathbb{A}_{\rm inf},\mathbb{A}_{\rm inf}) \otimes_{A_{\rm inf}}^L A_{\rm inf}/\pi & \cong  R\mathcal{H}om_{\mathrm{Spa}(\mathbb{F}_q)_v,\mathcal{O}_E}(\mathbb{A}_{\rm inf},\mathbb{A}_{\rm inf}/\pi) \\
     & \cong  R\mathcal{H}om_{\mathrm{Spa}(\mathbb{F}_q)_v,\mathbb{F}_q}(\mathbb{A}_{\rm inf} \otimes_{\mathcal{O}_E}^L \mathbb{F}_q,\mathcal{O}^+) \\
      & \cong  R\mathcal{H}om_{\mathrm{Spa}(\mathbb{F}_q)_v,\mathbb{F}_q}(\mathcal{O}^+,\mathcal{O}^+),
    \end{align*}
  where the second isomorphism comes from adjunction. Hence the claim follows from \Cref{self-extensions-o-plus}. 
\end{proof}

\begin{corollary}
\label{self-ext-of-a-inf-1-over-pi}
The natural map
$$
\mathbb{A}_{\rm inf}\langle F^{\pm 1} \rangle [1/\pi] \to R\mathcal{H}om_{\mathrm{Spa}(\mathbb{F}_q)_v,E}(\mathbb{B}_{\rm inf},\mathbb{B}_{\rm inf})
$$
is an almost isomorphism.
\end{corollary}
\begin{proof}
First, note that by adjunction 
$$
R\mathcal{H}om_{\mathrm{Spa}(\mathbb{F}_q)_v,E}(\mathbb{B}_{\rm inf},\mathbb{B}_{\rm inf}) \cong R\mathcal{H}om_{\mathrm{Spa}(\mathbb{F}_q)_v,\mathcal{O}_E}(\mathbb{A}_{\rm inf},\mathbb{B}_{\rm inf}).
$$
Hence, by \Cref{self-ext-of-a-inf}, we only need to justify that the natural map
$$
 \underset{\times \pi} \varinjlim ~ R\mathcal{H}om_{\mathrm{Spa}(\mathbb{F}_q)_v,\mathcal{O}_E}(\mathbb{A}_{\rm inf},\mathbb{A}_{\rm inf}) \to R\mathcal{H}om_{\mathrm{Spa}(\mathbb{F}_q)_v,\mathcal{O}_E}(\mathbb{A}_{\rm inf},\mathbb{B}_{\rm inf})
$$
is an isomorphism. This can be checked on cohomology groups, and for this it is enough to show that $\mathbb{A}_{\rm inf}$ is pseudo-coherent as a v-sheaf of $\underline{\mathcal{O}_E}$-modules. However, for each $i\geq 0$, the functors $\mathcal{E}xt_{\mathrm{Spa}(\mathbb{F}_q)_v,\mathcal{O}_E}^i(\mathbb{A}_{\rm inf},-)$ can be computed using the MacLane complex (\Cref{definition-and-existence-of-the-maclane-complex}) $M_{(\widetilde{\mathrm{Spa}(\mathbb{F}_q)_v},\underline{\mathcal{O}_E})}(\mathbb{A}_{\rm inf})$ and because of the description of the terms of this complex, to check that $\mathcal{E}xt_{\mathrm{Spa}(\mathbb{F}_q)_v,\mathcal{O}_E}^i(\mathbb{A}_{\rm inf},-)$ commutes with filtered colimits for all $i$, it suffices to prove that the functors $H_{\mathrm{Spa}(\mathbb{F}_q)_v}^j(\mathbb{A}_{\rm inf}^s \times \underline{\mathcal{O}_E}^r,-)$ commute with filtered colimits for all $j,r,s\geq 0$: this is true, since $\mathbb{A}_{\rm inf}^s \times \underline{\mathcal{O}_E}^r$ is represented by a qcqs perfectoid space. 
\end{proof}

Let $S \in \Perf_{\mathbb{F}_q}$ be an affinoid perfectoid space. Let $S^\sharp$ be an untilt over $E$ of $S$, corresponding to a primitive element $\xi \in \mathbb{A}_{\rm inf}(S)$. We will denote by $\mathcal{O}^\sharp$ (resp. $\mathcal{O}^{+\sharp}$) the sheaf of $\underline{E}$-vector spaces on $S_v$ sending $T \in \Perf_S$ to $\mathcal{O}(T^\sharp)$ (resp. $\mathcal{O}^+(T^\sharp)$), where $T^\sharp \in \Perf_{S^\sharp}$ is the perfectoid space corresponding to $T$ by the tilting equivalence $\Perf_{S^\sharp} \cong \Perf_S$.

\begin{theorem}
\label{computations-of-the-desired-extensions}
Let $S$ be as above. One has canonical identifications
$$
R\mathcal{H}om_{S_v,E}(\underline{E},\underline{E}) \cong \underline{E}, \quad R\mathcal{H}om_{S_v,E}(\underline{E},\mathcal{O}^\sharp)\cong \mathcal{O}^{\sharp},
$$
$$
R\mathcal{H}om_{S_v,E}(\mathcal{O}^\sharp,\mathcal{O}^\sharp) \cong \mathcal{O}^\sharp \oplus \mathcal{O}^\sharp[-1]
$$ 
and 
$$
R\mathcal{H}om_{S_v,E}(\mathcal{O}^\sharp,\underline{E}) \cong \mathcal{O}^\sharp (-1)[-1]
$$
(on the right, the symbol $(-1)$ denotes a Tate twist).
\end{theorem}
\begin{proof}
The first two isomorphisms are obvious, so we only need to prove the last two. There are exact sequences of v-sheaves of $\underline{E}$-vector spaces:
$$
0 \to \underline{E} \to \mathbb{B}  \overset{\varphi_{\mathbb{B}}-1} \longrightarrow \mathbb{B}  \to 0
$$
and
$$
0 \to \mathbb{B}_{\rm inf} \overset{\times \xi} \longrightarrow \mathbb{B}_{\rm inf} \overset{\theta} \to \mathcal{O}^\sharp \to 0,
$$
to which we will refer in the rest of this proof simply as the ``first" and ``second" exact sequences.

We have a distinguished triangle 
$$
R\mathcal{H}om_{S_{v},E}(\mathbb{B}_{\rm inf},\mathbb{B}_{\rm inf}) \to R\mathcal{H}om_{S_{v},E}(\mathbb{B}_{\rm inf},\mathbb{B}_{\rm inf}) \to R\mathcal{H}om_{S_{v},E}(\mathbb{B}_{\rm inf},\mathcal{O}^\sharp) 
$$
induced by the second exact sequence. Through the isomorphism of \Cref{self-ext-of-a-inf-1-over-pi}, it can be rewritten in the almost category as
$$
\mathbb{A}_{\rm inf}\langle F^{\pm 1} \rangle [1/\pi] \overset{\xi \times} \longrightarrow \mathbb{A}_{\rm inf}\langle F^{\pm 1} \rangle [1/\pi]  \to R\mathcal{H}om_{S_{v},E}(\mathbb{B}_{\rm inf},\mathcal{O}^\sharp),
$$
where $\xi\times$ denotes multiplication on the left. Write
$$
\mathbb{A}_{\rm inf} \langle F^{\pm 1} \rangle [1/\pi] = ( \oplus_{n\in \Z} \mathbb{A}_{\rm inf}. F^n )^{\wedge \pi}[1/\pi].
$$
We get an isomorphism
$$
R\mathcal{H}om_{S_{v},E}(\mathbb{B}_{\rm inf},\mathcal{O}^\sharp)  \cong \mathcal{O}^{+ \sharp} \langle F^{\pm 1} \rangle [1/\pi],
$$
where in the right hand side $\mathcal{O}^{+ \sharp} \langle F^{\pm 1} \rangle$ denotes the left $\mathcal{O}^{+\sharp}$-module
$$
\mathcal{O}^{+ \sharp} \langle F^{\pm 1} \rangle = ( \oplus_{n\in \Z} \mathcal{O}^{+ \sharp}. F^n )^{\wedge \pi}.
$$
Using again the second exact sequence, we get an exact triangle
$$
R\mathcal{H}om_{S_{v},E}(\mathcal{O}^\sharp,\mathcal{O}^\sharp) \to R\mathcal{H}om_{S_{v},E}(\mathbb{B}_{\rm inf},\mathcal{O}^\sharp) \to R\mathcal{H}om_{S_{v},E}(\mathbb{B}_{\rm inf},\mathcal{O}^\sharp)
$$
with the second map given by multiplication by $\xi$ on $\mathbb{B}_{\rm inf}$.
Write
$$
\mathcal{O}^{+ \sharp} \langle F^{\pm 1} \rangle [1/\pi] = ( \oplus_{n\in \Z} \mathcal{O}^{+ \sharp}. F^n )^{\wedge \pi}[1/\pi].
$$
Through the above isomorphism, the right map identifies with (the extension to the completion of) the map which is multiplication by the reduction modulo $\xi$ of $\varphi^n(\xi)$ in degree $n$ (as follows from the relation $F^n a=\varphi^n(a).F$, $a \in \mathbb{A}_{\rm inf}$, in $\mathbb{A}_{\rm inf}[F^{\pm 1}]$). For each $n\neq 0$, the reduction modulo $\xi$ of $\varphi^n(\xi)$ is a unit in $\mathcal{O}^\sharp$ and hence the triangle is quasi-isomorphic to  
$$
R\mathcal{H}om_{S_{v},E}(\mathcal{O}^\sharp,\mathcal{O}^\sharp) \to \mathcal{O}^\sharp \overset{0} \to \mathcal{O}^\sharp.
$$
In other words,
$$
R\mathcal{H}om_{S_{v},E}(\mathcal{O}^\sharp,\mathcal{O}^\sharp) \cong \mathcal{O}^\sharp \oplus \mathcal{O}^\sharp[-1],
$$ 
as desired.
\\

  To prove the isomorphism
  \begin{equation}
    \label{eq:3}
    R\mathcal{H}om_{S_{v},E}(\mathbb{B}_{\rm inf},\underline{E})\cong \mathcal{O}^\sharp (-1)  
  \end{equation}
  we will establish that if $S=\Spa(R,R^+)$ is perfectoid with untilt $(R^\sharp, R^{+,\sharp})$ over $E$, then there exists a natural isomorphism
  \[
    R\Hom_{S_v,E}(\mathbb{B}_{\rm inf}, \underline{E})\cong {A(R,R^+)}/{W_{\mathcal{O}_E}(R^{++})}[1/\pi][-1],
  \]
  where $R^{++}\subseteq R$ denotes the $R^+$-submodule of topologically nilpotent elements, such that the multiplication by $\xi$ on $\mathbb{B}_{\rm inf}$ is converted to the left multiplication by $\xi$. Granting this claim the assertion follows from the second exact sequence above applied in the first variable. Indeed, we get that
  $$
  R\Hom_{S_v,E}(\mathcal{O}^\sharp, \underline{E})
  $$
  is isomorphic to the kernel of multiplication by $\xi$ on ${A(R,R^+)}/{W_{\mathcal{O}_E}(R^{++})}[1/\pi]$ placed in degree $1$. But this kernel is
  \[
    1/\xi W_{\mathcal{O}_E}(R^{++})/W_{\mathcal{O}_E}(R^{++})[1/\pi]\cong R^{\sharp}(-1),
  \]
  the isomorphism being induced by Fontaine's map $\theta$.
  
  Thus, we are left with proving (\Cref{eq:3}). We have
  \[
    R\Hom_{S_v,E}(\mathbb{B}_{\mathrm{inf}},\underline{E})\cong R\Hom_{S_v,\mathcal{O}_E}(\mathbb{A}_{\rm inf},\mathcal{O}_E)[1/\pi]
  \]
  and
  \[
    R\Hom_{S_v,\mathcal{O}_E}(\mathbb{A}_{\rm inf},\mathcal{O}_E)\cong \mathrm{fib}(R\Hom_{S_v,\mathcal{O}_E}(\mathbb{A}_{\rm inf},\mathbb{A})\xrightarrow{F-1} R\Hom_{S_v,\mathcal{O}_E}(\mathbb{A}_{\rm inf},\mathbb{A}))
  \]
  using the integral version of the first exact sequence.
  The object $R\Hom_{S_v,\mathcal{O}_E}(\mathbb{A}_{\rm inf},\mathbb{A})$ is $\pi$-adically complete and concentrated in degree $0$ (as follows from \Cref{extensions-o-plus-and-o} via reduction mod $\pi$). In particular, we can (as $\Hom_{S_v}(\mathbb{A}_{\rm inf}, \mathcal{O}_E)=0$ by fiberwise connectedness of $\mathbb{A}_{\rm inf}$) conclude that $R\Hom_{S_v,\mathcal{O}_E}(\mathbb{A}_{\rm inf},\mathcal{O}_E)$ is concentrated in degree $1$. Acting with $A(R,R^+)$ on the natural morphism $\mathbb{A}_{\rm inf}\to \mathbb{A}$ yields a natural map
  \[
    A(R,R^+)\to R\Hom_{S_v,\mathcal{O}_E}(\mathbb{A}_{\rm inf},\mathbb{A}), 
  \]
  and thus by composition with the connecting morphism of the above exact triangle a natural morphism
  \[
   \Phi\colon A(R,R^+)\to R\Hom_{S_v,\mathcal{O}_E}(\mathbb{A}_{\rm inf},\mathcal{O}_E)[1].
  \]
  We claim that $\Phi$ factors over $A(R,R^+)/W_{\mathcal{O}_E}(R^{++})$. Because $\Phi_{|W_{\mathcal{O}_E}(R^{++})}$ factors over the morphism
  \[
    A_\inf(R,R^{+})\to A_\inf(R,R^+)\langle F^{\pm 1}\rangle \to R\Hom_{S_v,\mathcal{O}_E}(\mathbb{A}_\inf,\mathbb{A}_\inf) \to R\Hom_{S_v,\mathcal{O}_E}(\mathbb{A}_{\rm inf},\mathcal{O}_E)
  \]
  this follows from~\Cref{sec:some-ext-groups-critical-analytic-lemma-for-calculating-the-exts} below.
  Having constructed the natural morphism
  \[
    A(R,R^+)/W_{\mathcal{O}_E}(R^{++})\to R\Hom_{S_v,\mathcal{O}_E}(\mathbb{A}_{\rm inf},\mathcal{O}_E)[1]
  \]
  we can check that it is an isomorphism modulo $\pi$ as both sides are derived $\pi$-adically complete.
  Modulo $\pi$ (in the derived sense) the left-hand side becomes
  \[
    R/R^{++},
  \]
  while the right-hand side becomes
  \[
    R\Hom_{S_v,\mathcal{O}_E}(\mathbb{A}_{\rm inf},\F_q)[1]\cong R\Hom_{S_v,\F_q}(\mathcal{O}^+,\F_q)[1].
  \]
  Using the exact sequence
  \[
0\to \F_q\to \mathcal{O}\xrightarrow{F-1}\mathcal{O}\to 0
\]
(with $F$ the $q$-Frobenius) and \Cref{extensions-o-plus-and-o} the last term simplifies to
\[
  \mathrm{coker}(\mathcal{O}\langle F^{\pm 1}\rangle \xrightarrow{F-1} \mathcal{O}\langle F^{\pm 1}\rangle)
\]
(one can easily see that $F-1$ is injective on $\mathcal{O}\langle F^{\pm 1}\rangle$).
Therefore we conclude the proof of the theorem by \Cref{sec:some-ext-groups-critical-analytic-lemma-for-calculating-the-exts}.
\end{proof}

\begin{lemma}
  \label{sec:some-ext-groups-critical-analytic-lemma-for-calculating-the-exts}
  Assume $S=\Spa(R,R^+)$ is an affinoid perfectoid space over $\F_q$. 
 Then
    \[
      \mathrm{coker}(R\langle F^{\pm 1}\rangle \xrightarrow{F-1} R\langle F^{\pm 1}\rangle)\cong R/R^{++}
    \]
    and
    \[
      \mathrm{coker}(A_\inf(R,R^+)\langle F^{\pm 1}\rangle \xrightarrow{F-1} A_\inf(R,R^+)\langle F^{\pm 1}\rangle)\cong A_\inf(R,R^{++})/W_{\mathcal{O}_E}(R^{++}).
    \]
     
  \end{lemma}
  \begin{proof}
    We first prove the first assertion. Let $\varpi\in R$ be a pseudo-uniformizer and assume that
    \[
      a=\sum\limits_{i\in \Z} a_iF^i\in R\langle F^{\pm 1}\rangle
    \]
    is an element such that $|a_i|\leq |\varpi|$ for all $i\in \Z$. We claim that $a=(F-1)(b)$ for some $b\in R\langle F^{\pm 1}\rangle$. Let $\varphi\colon R\to R$ be the $q$-Frobenius. For $i\in \Z$, set
$$
b_i = - \sum_{k\geq 0} \varphi^{k}(a_{i-k}).
$$
The sum is converging. We claim that
$$
b_i \underset{|i| \to +\infty} \to 0
$$
for the natural topology on $R$. It is clear when $i \to -\infty$, since $a_j \underset{j\to -\infty} \to 0$. Fix $N>0$. Let $M>0$ such that if
$$
 \quad \forall j>M, \quad |a_j|\leq |\varpi|^{q^N}.
$$
Then, if $i>N+M$, we have
$$
|b_i| \leq \max(|\sum_{0\leq k <N} \varphi^{k}(a_{i-k})|, |\sum_{k \geq N} \varphi^k(a_{i-k})|). 
$$
In the first sum on the right, all terms have absolute value less than $|\varpi|^{q^N}$, by definition of $M$. In the second too, by our starting assumption on the sequence $(a_i)$. Therefore, if $i>N+M$, 
$$
|b_i| \leq |\varpi|^{q^N}
$$
and so $b_i \underset{i\to +\infty} \to 0$, as desired. The sequence $(b_i)$ therefore gives rise to an element $\sum\limits_{i\in \Z} b_iF^i\in R\langle F^{\pm 1}\rangle$.

For $i\in \Z$, we can compute
$$
\varphi(b_{i-1})-b_i =  - \sum_{k\geq 0} \varphi^{k+1}(a_{i-1-k}) + \sum_{k\geq 0} \varphi^k(a_{i-k}) = a_i
$$
since $a_j \underset{j\to - \infty} \to 0$. Hence $a=(F-1)b$, and the claim is proved.

    In particular, we see that the map
    \[
      R[F^{\pm 1}]\to \mathrm{coker}(R\langle F^{\pm 1}\rangle \xrightarrow{F-1} R\langle F^{\pm 1}\rangle)
    \]
    is surjective, i.e., each element in the cokernel can be represented by a finite sum of ${F^i}$'s with coefficients in $R$. As $R$ is perfect, we can further assume that in this sum all coefficients in front of $F^i, i\neq 0,$ are zero. It is easy to see that an element $r\in R\subseteq R\langle F^{\pm 1} \rangle$ is of the form $(F-1)(b)$ for some $b\in R\langle F^{\pm 1}\rangle$ if and only if $r\in R^{++}$. This finishes the proof of the first assertion.

    To prove the second we note that the above argument for $a\in \varpi R^+\langle F^{\pm 1}\rangle $ works mutatis mutandis for $a\in [\varpi]A_{\inf}(R,R^+)$ and shows that the natural morphism
    \[
      A_\inf(R,R^+)\to C:=\mathrm{coker}(A_\inf(R,R^+)\langle F^{\pm 1}\rangle \xrightarrow{F-1} A_\inf(R,R^+)\langle F^{\pm 1}\rangle)
    \]
    factors over
    \[
      B:=A_\inf(R,R^+)/(\bigcup\limits_{\varpi\in R \textrm{ pseudo-uniformizer}} [\varpi] A_{\inf}(R,R^+)).
    \]
    Let us note that $C$ is derived $\pi$-adically complete, and that the derived $\pi$-adic completion of $B$ is $A_{\inf}(R,R^+)/W_{\mathcal{O}_E}(R^{++})$. Hence, by derived modding out $\pi$ we can deduce that $B^\vee_\pi\cong C$ by the first part. 
  \end{proof}

Recall the morphism of sites 
$$
\tau: (X_S)_v \to S_v. 
$$
If $\mathcal{E} \in \Bun(S)$ with only non-negative slopes, resp. only negative slopes, we write
$$
\BC(\mathcal{E}) = \tau_\ast \mathcal{E} ~~ , ~~ \mathrm{resp}. ~ \BC(\mathcal{E})=R^1\tau_\ast \mathcal{E}.
$$

\begin{corollary}
\label{fully-faithfulness-R-tau-ast}
Let $S$ be a small v-stack. The functor 
$$
R\tau_\ast : \mathcal{P}erf(X_S) \to D(S_v, \underline{E}) 
$$
is fully faithful. In particular, if $\mathcal{E}$ is a vector bundle on $X_S$ with either only positive slopes or only negative slopes,
$$
R\mathcal{H}om_{S_v,E}(\BC(\mathcal{E}), \underline{E}[1]) \cong \BC(\mathcal{E}^\vee).
$$
\end{corollary}

\begin{proof}
The question is local on $S$, so we can assume that $S$ is affinoid perfectoid. Let $K, L \in \mathcal{P}erf(X_S)$. By \Cref{sec:perf-compl-relat-perfect-complex-strict}, both $K$ and $L$ are strictly perfect. Considering stupid truncations, we can assume that they are both vector bundles on $X_S$. If $\mathcal{E}$ is a vector bundle on $X_S$, we can always, up to localizing further on $S$, find a presentation of $\mathcal{E}$ of the form
$$
0 \to \mathcal{O}_{X_S}(\lambda)^n \to \mathcal{O}_{X_S}(\lambda+1)^m \to \mathcal{E} \to 0
$$
with $\lambda\in \Z$ (apply \cite[Proposition II.3.1]{fargues2021geometrization} to a suitable twist of $\mathcal{E}$, and \cite[Corollary II.2.20]{fargues2021geometrization}). Hence we can assume that $K$ and $L$ are both sums of line bundles, and this case in turn is reduced to the case $K, L \in \{\mathcal{O}_{X_S},\mathcal{O}_{X_S}(1)\}$ by using analogs of twists of the Euler sequence on $\mathbb{P}^1_{E}$. Choose an untilt $S^\sharp$ of $S$ over $E$, corresponding to a Cartier divisor (denoted by the same letter) on $X_S$. One has a short exact sequence
$$
0 \to \mathcal{O}_{X_S} \to \mathcal{O}_{X_S}(1) \to \mathcal{O}_{S^\sharp} \to 0
$$
and so we can as well assume that $K, L \in \{\mathcal{O}_{X_S},\mathcal{O}_{S^\sharp}\}$. We have
$$
R\tau_\ast \mathcal{O}_{X_S}= \underline{E},
$$
cf. \cite[Proposition II.2.5 (ii)]{fargues2021geometrization}, and, with the notations of \Cref{computations-of-the-desired-extensions},
$$
R\tau_\ast \mathcal{O}_{S^\sharp} = \mathcal{O}^\sharp
$$
as, on v-sites, $\mathcal{O}_{S^\sharp}$ is the pushforward of $\mathcal{O}^\sharp$ along a section of $\tau$.

We note that, after passing to a v-cover of $S$ to trivialize the Tate twist,
  \begin{align*}
    R\tau_\ast R\mathcal{H}om_{X_S}(\mathcal{O}_{X_S},\mathcal{O}_{X_S})\cong \underline{E}, \\
    R\tau_\ast R\mathcal{H}om_{X_S}(\mathcal{O}_{X_S},\mathcal{O}_{S^\sharp})\cong \mathcal{O}^\sharp, \\
    R\tau_\ast R\mathcal{H}om_{X_S}(\mathcal{O}_{S^\sharp},\mathcal{O}_{S^\sharp})\cong \mathcal{O}^\sharp\oplus \mathcal{O}^\sharp[-1], \\
    R\tau_\ast R\mathcal{H}om_{X_S}(\mathcal{O}_{S^\sharp},\mathcal{O}_{X_S})\cong \mathcal{O}^\sharp[-1],\\
  \end{align*}
which matches with \Cref{computations-of-the-desired-extensions}, and we have to check that certain maps $\mathcal{O}^\sharp\to \mathcal{O}^\sharp$ on $S_v$ are isomorphisms. But this can be checked after base change from $S$ to geometric points. If $S$ is a geometric point, it suffices to see that the respective maps are non-zero, which reduces to the assertion that the pushforwards of the exact sequences
\[
  0\to \mathcal{O}_{X_S}\to \mathcal{O}_{X_S}(1)\to \mathcal{O}_{S^\sharp}\to 0
\]
resp.\
\[
  0\to \mathcal{O}_{S^\sharp}\xrightarrow{\xi} B_{\dR,S^\sharp}^+/\xi^2 B^{+}_{\dR,S^\sharp}\to \mathcal{O}_{S^\sharp}\to 0 
\]
via $\tau$ are nonsplit. For the pushforward of the first sequence this follows because $\mathrm{BC}(\mathcal{O}_{X_S}(1))$ is connected. For the pushforward of the second, note that a splitting would in particular, using \cite[Corollary 6.6]{scholze_p_adic_hodge_theory_for_rigid_analytic_varieties}, give a map of condensed $\underline{E}$-vector spaces
$$
\underline{C^\sharp}  \to \underline{B_{\rm dR}^+(C^\sharp)/\xi^2}
$$
(where $S^\sharp=\mathrm{Spa}(C^\sharp, C^{\sharp +})$) and hence, since $C^\sharp$ and $B_{\rm dR}^+(C^\sharp)/\xi^2$ are compactly generated as topological spaces, a continuous $E$-linear map
$$
C^\sharp \to B_{\rm dR}^+(C^\sharp)/\xi^2
$$  
being a section of the map $\theta$. However, such a map is known not to exist.

Therefore, the first part of the corollary follows (and we further identified generating classes in $\mathcal{E}xt^1_{S_v,E}(\mathcal{O}^\sharp, \underline{E})$ and $\mathcal{E}xt^1_{S_v,E}(\mathcal{O}^\sharp, \mathcal{O}^\sharp)$). 

We turn to the last assertion. Assume first that $\mathcal{E}$ has only positive slopes (so that $\mathcal{E}^\vee$ has only negative slopes). By \cite[Proposition II.3.4 (iii)]{fargues2021geometrization}, there is an \'etale cover $S' \to S$ such that for any $T$ affinoid perfectoid over $S'$,
$$
H^1(X_T,\mathcal{E})=0
$$
and thus
$$
H^i(X_T,\mathcal{E})=0
$$
for all $i>0$. In particular, $R\tau_\ast \mathcal{E}=\BC(\mathcal{E})$. Hence, by the above fully faithfulness (applied to $T$),
$$
R\mathrm{Hom}_{T_v,E}(\BC(\mathcal{E}), \underline{E}[1]) \cong R\mathrm{Hom}_{\mathcal{O}_{X_T}}(\mathcal{E},\mathcal{O}_{X_T}[1]) \cong R\Gamma(X_T,\mathcal{E}^\vee)[1].
$$
But by \cite[Proposition II.3.4 (i)]{fargues2021geometrization}, the latter is nothing but
$$
H^1(X_T,\mathcal{E}^\vee)
$$
placed in degree $0$.
This being true for all $T$ affinoid perfectoid over $S'$, we deduce the desired claim. The proof is similar when $\mathcal{E}$ has negative slopes.
\end{proof}

\begin{remark}
\label{serre-duality}
Let $S$ be a small v-stack. From the corollary (applied to any $T$ over $S$), one deduces for each $K \in \mathcal{P}erf(X_S)$ a natural isomorphism
 \[
     R\tau_{\ast}(R\mathcal{H}om_{\mathcal{O}_{X_S}}(K,\mathcal{O}))\cong R\mathcal{H}om_{S_v,E}(R\tau_{\ast}K,E),
  \]
which can be thought of as ``relative Serre duality" on the Fargues-Fontaine curve.
\end{remark}

\begin{remark}
\label{dualizability-bc-spaces-geometric point}
The proof of \Cref{fully-faithfulness-R-tau-ast} shows that the essential image of $\mathcal{P}erf(X_S)$ in $D(S_v,\underline{E})$ can be described as the full subcategory $\mathcal{C}$ of objects which v-locally on $S$ lie in the smallest idempotent complete triangulated subcategory spanned by $\underline{E}$ and $\mathbb{A}^1_{S^\sharp}$ for some untilt $S^\sharp$ of $S$ over $\underline{E}$ and that the functor
$$
R\mathcal{H}om_{S_v,E}(-, \underline{E}[1])
$$
defines an autoduality of $\mathcal{C}$.

Milne, \cite{milne1976duality}, has used a similar autoduality statement for unipotent perfect group schemes (due to Breen) to study Poincar\'e duality for flat cohomology of smooth proper surfaces in characteristic $p$. Analogously, one may expect the above fact to be useful to the study of Poincar\'e duality for pro-\'etale or syntomic cohomology of smooth proper rigid-analytic varieties over $S^\sharp$.
\end{remark}

\subsection{Stacks in $E$-vector spaces}
\label{sec:gener-fram-case}
Let $\mathcal{S}$ be a site. A Picard groupoid, as defined in \cite[Section 1.4]{deligne1973formule}, cf. also \cite[Definition 5.1.1]{patel_de_rham_epsilon_factors}, is a symmetric monoidal category such that all the morphisms are invertible and the semigroup of isomorphism classes of the objects is a group. Picard groupoids form an $\infty$-category equivalent to the category of spectra whose only non-zero homotopy groups are in degrees $0,1$ (for a proof at the level of homotopy categories, cf. \cite[Theorem 5.1.3]{patel_de_rham_epsilon_factors}). The truncation of the tensor product of spectra defines a symmetric monoidal structure on the category of spectra whose only non-zero homotopy groups are in degrees $0,1$, and therefore, via this equivalence, on the category of Picard groupoids. More generally, sheaves of Picard groupoids\footnote{We do not want to call them \textit{Picard stacks}, because Deligne uses this terminology for a more restricted class of objects.} on $\mathcal{S}$ form an $\infty$-category equivalent to the category sheaves of spectra on $\mathcal{S}$, whose sheaves of homotopy groups vanish for $i \neq 0, 1$, and therefore has a symmetric monoidal structure.    

Let $\mathcal{A}$ be a sheaf of discrete (i.e. concentrated in degree $0$) rings on $\mathcal{S}$. Then $\mathcal{A}$ is a commutative algebra object in the symmetric monoidal category of sheaves of Picard groupoids on $\mathcal{S}$. One can consider the category of $\mathcal{A}$-module objects in the category of sheaves of Picard groupoids on $\mathcal{S}$. Since complexes of sheaves of $\mathcal{A}$-modules on $\mathcal{S}$ are equivalent to $\mathcal{A}$-modules in the category of sheaves of spectra on $\mathcal{S}$ (\cite[Corollary 2.1.2.3]{lurie_spectral_algebraic_geometry}), this new category is equivalent to the category of sheaves of $\mathcal{A}$-modules on $\mathcal{S}$ having non-zero cohomology sheaves only in degrees $-1, 0$. 

Taking $\mathcal{A}=\mathbb{Z}$ the constant sheaf (on any site $\mathcal{S}$), one recovers the category of \textit{strictly commutative Picard stacks} studied by Deligne in \cite[Section 1.4]{deligne1973formule} and most often called by him \textit{Picard stacks}.

\begin{definition}
\label{definition-stack-in-E-vs}
Let $S$ be small v-stack. We define a \textit{stack in $E$-vector spaces (over $S$)} to be an object in the category of $\underline{E}$-module objects in the category of sheaves of Picard groupoids on the site $S_v$.
\end{definition}

By the considerations above, the category of stacks in $E$-vector spaces (over $S$) is equivalent to the category of complexes of v-sheaves of $\underline{E}$-vector spaces on $S$, having non-zero cohomology sheaves only in degrees $-1,0$. The stack in $E$-vector spaces attached to such a complex $K$ will be denoted
$$
\mathbb{V}(K).
$$
Conversely, the complex attached to a stack in $E$-vector spaces $\mathcal{G}$ is denoted
$$\mathcal{G}_\bullet.$$

For example,
$$
\mathbb{V}(\underline{E}[0])=\underline{E}, \quad \mathbb{V}(\underline{E}[1])=[S/\underline{E}].
$$
A sequence 
$$
0 \to \mathcal{G}^\prime \to \mathcal{G} \to \mathcal{G}^{\prime\prime} \to 0
$$
of stacks in $E$-vector spaces will be said to be \textit{exact} when $\mathcal{G}\to \mathcal{G}^{\prime\prime}$ is a surjection of stacks and $\mathcal{G}^{\prime}$ is equipped with an equivalence to the homotopy fiber of $0\in \mathcal{G}^{\prime\prime}$. Equivalently, there exists a morphism $(\mathcal{G}^{\prime\prime})_\bullet \to (\mathcal{G}^{\prime})_\bullet[1]$ in $D(S_v,\underline{E})$ such that the triangle 
$$
(\mathcal{G}^{\prime})_\bullet \to \mathcal{G}_\bullet \to (\mathcal{G}^{\prime\prime})_\bullet \to (\mathcal{G}^{\prime})_\bullet[1] 
$$
is an exact triangle. For example, the sequence
  \[
    0\to \underline{E}\to \ast \to [\ast/\underline{E}]\to 0
  \]
  is an exact sequence of stacks in $E$-vector spaces.

\begin{definition}
\label{definition-dual-commutative-group-stack}
Let $S$ be a small v-stack. Let $\mathcal{G}$ be a stack in $E$-vector spaces over $S$. We define its \textit{dual} $\mathcal{G}^{\vee}$ to be the stack of homomorphisms of stacks in $E$-vector spaces from $\mathcal{G}$ to $[S/E]$:
$$
\mathcal{G}^{\vee}:= \mathcal{H}om(\mathcal{G}, [S/\underline{E}]).
$$
One says that $\mathcal{G}$ is \textit{dualizable} if the natural morphism
$$
\mathcal{G} \to (\mathcal{G}^\vee)^\vee
$$
is an isomorphism. 
\end{definition}

\begin{remark}
\label{dual-in-terms-of-deligne-equivalence}
Let $\mathcal{G}$ be a stack in $E$-vector spaces over a small v-stack $S$. Then
$$
(\mathcal{G}^{\vee})_\bullet = \tau_{\leq 0} R\mathcal{H}om(\mathcal{G}_\bullet, \underline{E}[1]).
$$
\end{remark}

For the next examples, fix a small v-stack $S$. 

\begin{example}
\label{example-dual-of-pro-etale-local-systems}
Let $\mathbb{L}$ be a pro-\'etale $\underline{E}$-local system on $S_{v}$. Then
$$
\mathbb{L}^\vee \cong [S/\mathbb{L}^\ast],
$$
where $\mathbb{L}^\ast$ is the dual of $\mathbb{L}$ as a pro-\'etale $\underline{E}$-local system, and
$$
[S/\mathbb{L}]^\vee \cong \mathbb{L}^\ast.
$$
In particular, $\mathbb{L}$ and $[S/\mathbb{L}]$ are dualizable.
\end{example}

If $\mathcal{E} \in \Bun(S)$ with only non-negative slopes, resp.\ only negative slopes, we write
$$
\BC(\mathcal{E}) = \tau_\ast \mathcal{E} ~~ , ~~ \mathrm{resp}. ~ \BC(\mathcal{E})=R^1\tau_\ast \mathcal{E}.
$$

\begin{example}
\label{example-dual-of-ss-bc-space}
Let $\mathcal{E}$ be a vector bundle on $X_S$ having either only positive slopes or only negative slopes. Then 
$$
\BC(\mathcal{E})^\vee = \BC(\mathcal{E}^\vee). 
$$
This is an immediate consequence of \Cref{fully-faithfulness-R-tau-ast} and \Cref{dual-in-terms-of-deligne-equivalence}. In particular, these stacks in $E$-vector spaces are dualizable.
\end{example}

\begin{example}
\label{example-dual-of-flat-coh-sheaf-with-positive-slopes}
Let $\mathcal{F}\in \Coh^{\rm fl, >0}(S)$. Then we claim that
$$
\BC(\mathcal{F}):= \tau_\ast \mathcal{F}
$$
and $\BC(\mathcal{F})^\vee$ are cohomologically smooth and dualizable over $S$. Since this assertion is v-local on $S$, we can assume that $S$ is affinoid perfectoid. Choose a presentation 
$$
0 \to \mathcal{O}_{X_S}^m \to \mathcal{E} \to \mathcal{F} \to 0
$$
with $\mathcal{E}$ a positive slope semi-stable vector bundle on $X_S$, as in \Cref{presentation-flat-coh-sheaf-positive-slopes}. Since $R^1\tau_\ast \mathcal{O}_{X_S}=0$, we get a short exact sequence of pro-\'etale sheaves
$$
0 \to \underline{E}^m \to \BC(\mathcal{E}) \to \BC(\mathcal{F}) \to 0
$$
which shows that $\BC(\mathcal{F})$ is the quotient of a cohomologically smooth v-sheaf by a locally pro-$p$ equivalence relation, hence is cohomologically smooth.

By \Cref{presentation-flat-coh-sheaf-non-negative-slopes}, we can also find a short exact sequence
$$
0 \to \mathcal{O}_{X_S}(-1)^d \to \mathcal{E}^\prime \to \mathcal{F} \to 0, 
$$
with $\mathcal{E}^\prime$ semi-stable of slope 0. This gives a short exact sequence
$$
0  \to \BC(\mathcal{E}^\prime) \to \BC(\mathcal{F}) \to \BC(\mathcal{O}_{X_S}(-1)^d) \to 0
$$
By \Cref{fully-faithfulness-R-tau-ast}, we know that
$$
\mathcal{E}xt_{S_{v},E}^2(\BC(\mathcal{O}_{X_S}(-1)^d),\underline{E})=0.
$$
Hence, applying \Cref{dual-exact-sequence-stacks-in-E-vs} below, we obtain a short exact sequence of stacks in $E$-vector spaces
$$
0 \to \BC(\mathcal{O}_{X_S}(-1)^d)^\vee \to \BC(\mathcal{F})^\vee \to \BC(\mathcal{E}^\prime)^\vee \to 0.
$$
By the above two examples, the left and right terms are cohomologically smooth, hence so is $\BC(\mathcal{F})^\vee$. Applying again \Cref{dual-exact-sequence-stacks-in-E-vs}, we also get a short exact sequence
$$
0 \to (\BC(\mathcal{E}^\prime)^\vee)^\vee \to (\BC(\mathcal{F})^\vee)^\vee \to (\BC(\mathcal{O}_{X_S}(-1)^d)^\vee)^\vee \to 0.
$$
Hence, using the above two examples, we also conclude that $\BC(\mathcal{F})$ (and hence also $\BC(\mathcal{F})^\vee$) is dualizable.
\end{example}

We already used the following lemma.

\begin{lemma}
\label{dual-exact-sequence-stacks-in-E-vs}
Let 
$$
0 \to \mathcal{G}^\prime \to \mathcal{G} \to \mathcal{G}^{\prime\prime} \to 0
$$
be a short exact sequence of stacks in $E$-vector spaces. We have an exact sequence
$$
0 \to (\mathcal{G}^{\prime \prime})^\vee \to \mathcal{G}^\vee \to (\mathcal{G}^{\prime})^\vee. 
$$
It is exact on the right if
$$
R^2 \mathcal{H}om((\mathcal{G}^{\prime \prime})_\bullet, \underline{E})=0.
$$
\end{lemma}
\begin{proof}
The proof is easy.
\end{proof}

\subsection{The Fourier transform for stacks in $E$-vector spaces and its properties}
\label{sec:four-transf-stacks}
To define the Fourier transform for a stack in $E$-vector spaces $f\colon \mathcal{G} \to S$, we will have to consider (derived) pushforwards with proper supports (satisfying some of the usual properties, cf.\ \cite{scholze_etale_cohomology_of_diamonds}, \cite[Theorem 4.1.2, Theorem 4.1.3]{kaletha_weinstein_on_the_kottwitz_conjecture_for_local_shimura_varieties}), and unfortunately this means that we have to demand some condition on $f$.
According to \cite[Theorem 4.2.2]{kaletha_weinstein_on_the_kottwitz_conjecture_for_local_shimura_varieties}, the most general available assumption under which derived functors $g_!, g^!$ are defined is that $g$ is \textit{smooth-locally nice}. Here, a morphism $g\colon X\to Y$ between v-stacks is called smooth-locally nice if there exists a surjective, cohomologically smooth morphism $h\colon U\to X$ from a locally spatial diamond $U$ such that $g\circ h$ is compactifiable, representable in locally spatial diamonds of finite geometric transcendence degree, cf.\ \cite[Definition 4.2.1, Definition 4.1.1]{kaletha_weinstein_on_the_kottwitz_conjecture_for_local_shimura_varieties}. Let us note that the property of being smooth-locally nice is stable under base change along morphisms which are representable in Artin v-stacks. This motivates the following definition of a \textit{nice} stack in $E$-vector spaces. 

\begin{definition}
  \label{sec:stacks-bc-type-definition-nice-stack-in-e-v-s}
  Let $S$ be a small v-stack. We call the stack in $E$-vector spaces $f\colon \mathcal{G}\to S$ \textit{nice} if $f$ and $f^\vee\colon \mathcal{G}^\vee\to S$ are representable in Artin v-stacks and smooth-locally nice.
\end{definition}

In the following, we fix a nice stack in $E$-vector spaces $f\colon \mathcal{G}\to S$ with dual $f^\vee\colon \mathcal{G}^\vee\to S$, and consider the diagram
\[
  \xymatrix{
    & \mathcal{G}^{\vee}\times_S \mathcal{G} \ar[r]^\alpha\ar[ld]_{\pi^\vee}\ar[rd]^\pi& [S/E] \\
    \mathcal{G}^{\vee} & & \mathcal{G}
  }
\]
of stacks over $S$, where $\alpha$ is the evaluation map. Note that by our assumption that $\mathcal{G}$ is nice, the morphisms $\pi, \pi^\vee$ are smooth-locally nice. In particular, we can now define the (unnormalized) Fourier transform for $\mathcal{G}$. For any (non-trivial) character $\psi\colon E\to \Lambda^\times$, we let
\[
  \mathcal{L}_\psi\in D_{\et}([S/\underline{E}],\Lambda)
\]
be the associated rank $1$ local system.

\begin{definition}
  \label{sec:stacks-bc-type-definition-fourier-transform-for-nice-stacks}
  We define
  \[
    \mathcal{F}^{\mathrm{un}}_{\psi}:=\mathcal{F}^{\mathrm{un}}_{\psi,!,\mathcal{G}\to \mathcal{G}^\vee}\colon D_{\et}(\mathcal{G},\Lambda)\to D_{\et}(\mathcal{G}^\vee,\Lambda),
  \]
  the \textit{(unnormalized) Fourier transform associated with $\mathcal{G}$ and $\psi$},
  as the functor
  \[
    A\mapsto \pi_!^\vee(\pi^\ast(A)\otimes \alpha^\ast \mathcal{L}_\psi).
  \]
\end{definition}

It is formal that $\mathcal{F}^{\mathrm{un}}_{\psi,!,\mathcal{G}\to \mathcal{G}^\vee}$ is left adjoint to the functor
\[
  \mathcal{F}^{\mathrm{un}}_{\psi^{-1},\ast}=\mathcal{F}^{\mathrm{un}}_{\psi^{-1},\ast, \mathcal{G}^\vee\to \mathcal{G}}\colon D_{\et}(\mathcal{G}^\vee,\Lambda)\to D_{\et}(\mathcal{G},\Lambda)
\]
defined as
\[
  A\mapsto \pi_\ast(\pi^{\vee,!}(A)\otimes \alpha^\ast \mathcal{L}_{\psi^{-1}}),
\]
and that
\[
  \mathbb{D}_{f^\vee}\circ \mathcal{F}^{\mathrm{un}}_{\psi,!}\cong \mathcal{F}^{\mathrm{un}}_{\psi^{-1},\ast}\circ \mathbb{D}_f,
\]
where $\mathbb{D}_{f^\vee},\mathbb{D}_f$ denotes the relative Verdier duality functor of $f,f^\vee$.

We can also reverse directions and consider
\[
  \mathcal{F}^{\mathrm{un}}_{\psi,!,\mathcal{G}^\vee\to \mathcal{G}}:=\pi_!(\pi^{\vee,\ast}(-)\otimes \alpha^{\ast}\mathcal{L}_\psi)\colon D_{\et}(\mathcal{G}^\vee,\Lambda)\to D_{\et}(\mathcal{G},\Lambda)
\]
and
\[
  \mathcal{F}^{\mathrm{un}}_{\psi,\ast,\mathcal{G}\to \mathcal{G}^\vee}:=\pi_\ast^\vee(\pi^{!}(-)\otimes \alpha^{\ast}\mathcal{L}_\psi)\colon D_{\et}(\mathcal{G}^\vee,\Lambda)\to D_{\et}(\mathcal{G},\Lambda)
\]

In the following we denote by
\[
  a\colon \mathcal{G}\to (\mathcal{G}^\vee)^\vee
\]
the \textit{negative} of the canonical biduality morphism.
If $\mathcal{G}^\vee$ is again nice, it is formal (using the proper base change theorem) to see that
\[
  a^\ast\circ \mathcal{F}^{\mathrm{un}}_{\psi,!,\mathcal{G}^\vee\to (\mathcal{G}^{\vee})^\vee}\cong \mathcal{F}^{\mathrm{un}}_{\psi^{-1},!,\mathcal{G}^\vee\to \mathcal{G}}.
\]

\begin{remark}
\label{compatibility-of-ft-with-base-change}
By the proper base change theorem, the formation of the Fourier transform $\mathcal{F}^{\mathrm{un}}_{\psi,!,\mathcal{G}\to \mathcal{G}^\vee}(A)$ of $A\in D_\et(\mathcal{G},\Lambda)$ commutes with any base change $S' \to S$.
\end{remark}

The following definition is motivated by \cite[Definition A.4.5]{chen_zhu_geometric_langlands_in_prime_characteristic}.

\begin{definition}
  \label{sec:stacks-bc-type-definition-very-nice}
  We call a nice stack $\mathcal{G}$ in $E$-vector spaces over $S$ \textit{very nice} if the following conditions are satisfied:
  \begin{enumerate}
  \item $\mathcal{G}$ is dualizable,
  \item the Fourier transform $\mathcal{F}^{\mathrm{un}}_{\psi,!,\mathcal{G}\to \mathcal{G}^\vee}$ is an equivalence (for any choice of non-trivial character $\psi$),
  \item after decomposing $S$ into clopen substacks, the inverse of $\mathcal{F}^{\mathrm{un}}_{\psi,!,\mathcal{G}\to \mathcal{G}^\vee}$ is isomorphic to $\mathcal{F}^{\mathrm{un}}_{\psi^{-1},!,\mathcal{G}^\vee\to \mathcal{G}}$, up to a shift.
  \end{enumerate}
\end{definition}

In \cite[Definition A.4.5]{chen_zhu_geometric_langlands_in_prime_characteristic} the last requirement is not taken as part of the definition, but is satisfied in all the examples. We will give several examples of very nice stacks in \Cref{sec:exampl-picard-stacks}.

For now, let us list some abstract properties of very nice stacks in $E$-vector spaces and their Fourier transform.

\begin{lemma}
  \label{sec:stacks-bc-type-properties-of-very-nice-stacks-plus-fourier-transform} Let $S$ be a small v-stack. Let $f\colon \mathcal{G}\to S$ be a very nice stack in $E$-vector spaces.
  \begin{enumerate}
  \item The dual $f^\vee\colon \mathcal{G}^\vee \to S$ is a very nice stack in $E$-vector spaces.
    \item We have, on a clopen decomposition of $S$,
      \[
        \mathcal{F}^{\mathrm{un}}_{\psi,!,\mathcal{G}\to \mathcal{G}^\vee}\cong \mathcal{F}^{\mathrm{un}}_{\psi,\ast,\mathcal{G}\to \mathcal{G}^\vee}[d].
      \]
      for some $d\in \Z$.
    \item We have
      \[
        \mathbb{D}_{f^\vee}\circ \mathcal{F}_{\psi,!,\mathcal{G}\to \mathcal{G}^\vee}^{\mathrm{un}}\cong \mathcal{F}^{\mathrm{un}}_{\psi^{-1},!,\mathcal{G}\to \mathcal{G}^\vee}\circ \mathbb{D}_f.
      \]
  \end{enumerate}
\end{lemma}
\begin{proof}
  The first statement follows from the definition of a very nice stack in $E$-vector spaces and the proper base change theorem. For the second statement, note that
  \[
    \mathcal{F}_{\psi,!,\mathcal{G}\to \mathcal{G}^\vee}^{\mathrm{un}}
  \]
  is inverse to $\mathcal{F}_{\psi^{-1},!,\mathcal{G}^\vee\to \mathcal{G}}$ (locally on a clopen decomposition of $S$ up to a shift) because $\mathcal{G}^\vee$ is very nice, and thus in particular right adjoint to it.
  But 
  \[
    \mathcal{F}^{\mathrm{un}}_{\psi,\ast,\mathcal{G}\to \mathcal{G}^\vee}
  \]
  is formally right adjoint to $\mathcal{F}_{\psi^{-1},!,\mathcal{G}^\vee\to \mathcal{G}}^{\mathrm{un}}$ (up to a shift on a clopen decomposition of $S$) and thus both functors are isomorphic (on a clopen decomposition of $S$ up to some shift). 
  The last statement follows from $(2)$ and the formula
  \[
    \mathbb{D}_{f^\vee}\circ \mathcal{F}_{\psi,!,\mathcal{G}\to \mathcal{G}^\vee}^{\mathrm{un}}\cong \mathcal{F}^{\mathrm{un},\prime}_{\psi^{-1},\ast ,\mathcal{G}^\vee\to \mathcal{G}}\circ \mathbb{D}_f
  \]
  mentioned after \Cref{sec:stacks-bc-type-definition-fourier-transform-for-nice-stacks}.
\end{proof}

\begin{remark}
\label{commutation-with-duality-implies-perversity}
The Fourier transform studied by Deligne and Laumon in algebraic geometry also commutes with Verdier duality. In this setting, it has the following important consequence: the Fourier transform preserves perversity. Indeed, the commutation with Verdier duality reduces to verifying half of the inequalities defining perversity, and they are easily checked using that the morphism from the total space of a vector bundle to the base is affine. We do not know how to make sense of of the statement that ``the Fourier transform preserves perversity" and prove it in our context.
\end{remark}

\begin{lemma}
  \label{sec:stacks-bc-type-properties-of-very-nice-stacks-plus-fourier-transform-bis} Let $S$ be a small v-stack. Let $f\colon \mathcal{G}\to S$ be a very nice stack in $E$-vector spaces. Assume moreover that the implicit natural isomorphism between $\mathcal{F}^{\mathrm{un}}_{\psi,!,\mathcal{G}\to \mathcal{G}^\vee}\circ \mathcal{F}^{\mathrm{un}}_{\psi^{-1},!,\mathcal{G}^\vee\to \mathcal{G}}$, $\mathcal{F}^{\mathrm{un}}_{\psi^{-1},!,\mathcal{G}^\vee\to \mathcal{G}} \circ \mathcal{F}^{\mathrm{un}}_{\psi,!,\mathcal{G}\to \mathcal{G}^\vee}$ and some shift functors are induced by morphisms in $D_\et(\mathcal{G}^\vee\times_S \mathcal{G}^\vee,\Lambda)$ resp.\ $D_\et(\mathcal{G}\times_S \mathcal{G},\Lambda)$ between the kernels of the two compositions. Then,
  \begin{enumerate}
    \item If $S^\prime\to S$ is a morphism, which is representable in Artin v-stacks, then $\mathcal{G}\times_S S^\prime\to S^\prime$ is a very nice stack in $E$-vector spaces. 
    \item If $f^\prime\colon \mathcal{G}^\prime\to S$ is another very nice stack in $E$-vector spaces, then $\mathcal{G}\times \mathcal{G}^\prime\to S$ is a very nice stack in $E$-vector spaces. 
     \end{enumerate}
\end{lemma}
If $f,f^\vee$ are cohomologically smooth or nice, then $(2)$ holds without the assumption that $S^\prime\to S$ is representable in Artin v-stacks.
\begin{proof}
  For the first statement, note that the representability in Artin v-stacks is assumed in order to ensure that $\mathcal{G}\times_S S^\prime\to S^\prime$ is again smooth-locally nice. For the second statement, note that the property that $\mathcal{F}_{\psi,\mathcal{G}}^{\mathrm{un}},\mathcal{F}_{\psi, \mathcal{G}^\prime}^{\mathrm{un}}$ are equivalences (with inverses up to a shift given by $\mathcal{F}_{\psi,\mathcal{G}^{\vee}}^{\mathrm{un}},\mathcal{F}_{\psi, \mathcal{G}^{\prime,\vee}}^{\mathrm{un}}$ with respective natural transformations witnessing these inverses induced by morphisms between the kernels) passes to $\mathcal{F}_{\psi, \mathcal{G}\times \mathcal{G}^\prime}^{\mathrm{un}}$ as this functor is the functor with kernel the exterior product the kernels of $\mathcal{F}_{\psi,\mathcal{G}}^{\mathrm{un}},\mathcal{F}_{\psi, \mathcal{G}^\prime}^{\mathrm{un}}$.
\end{proof}

The following criterion is an abstract version of the usual argument for proving involutivity of the Fourier transform, cf.\ \cite[Th\'eor\`eme 1.2.2.1]{laumon_transformation_de_fourier}, which uses only proper base change and the projection formula.

\begin{lemma}
  \label{sec:stacks-e-vector-criterion-to-get-involutivity}
  Let $S$ be a small v-stack. Let $f\colon \mathcal{G}\to S$ be a nice dualizable stack in $E$-vector spaces. Assume that the unit section $e: S \to \mathcal{G}$ is smooth-locally nice and that after passing to some clopen decomposition of $S$, $\pi_!(\alpha^\ast \mathcal{L}_\psi)\cong e_!\Lambda[d]$ for some $d\in \Z$. Then the composition
  $$
  \mathcal{F}^{\mathrm{un}}_{\psi^{-1},!,\mathcal{G}^\vee\to \mathcal{G}} \circ \mathcal{F}^{\mathrm{un}}_{\psi,!,\mathcal{G}\to \mathcal{G}^\vee}: D_\et(\mathcal{G},\Lambda) \to D_\et(\mathcal{G},\Lambda)
  $$
  is, after passing to some clopen decomposition of $S$, isomorphic to the shift functor $(-)[d]$.
  
  In particular, if the unit sections $e: S \to \mathcal{G}, e^\vee: S \to \mathcal{G}^\vee$ are smooth-locally nice and if after passing to some clopen decomposition of $S$,
  \[
    \pi_!(\alpha^\ast \mathcal{L}_\psi)\cong e_!\Lambda[d_1], \pi_!^\vee(\alpha^\ast \mathcal{L}_\psi)\cong e^\vee_!\Lambda[d_2]
  \]
  for some $d_1,d_2\in \Z$, then $\mathcal{G}$ is very nice. 
\end{lemma}
\begin{proof}
  The argument is exactly as in the proof of \cite[Th\'eo\`eme 1.2.2.1]{laumon_transformation_de_fourier}.
\end{proof}

\begin{definition}
\label{def-reflexive-ula-compact}
Let $S$ be a small v-stack. Let $f\colon \mathcal{G}\to S$ be a nice stack in $E$-vector spaces and let $A\in D_\et(\mathcal{G},\Lambda)$.
\begin{itemize}
\item We say that $A$ is \textit{reflexive} if the natural map $A \to \mathbb{D}_f(\mathbb{D}_f(A))$ is an isomorphism.
\item We say that $A$ is \textit{ULA} if it is universally locally acyclic with respect to $f$ in the sense of \cite[Definition IV.2.1, Definition IV.2.22]{fargues2021geometrization}.
\item We say that $A$ is \textit{compact} if it is so as an object of the category $D_\et(\mathcal{G},\Lambda)$, i.e. $\mathrm{Hom}_{D_\et(\mathcal{G},\Lambda)}(A,-)$ commutes with all direct sums.
\end{itemize}
We denote by
$$
D_\et^{\rm ref}(\mathcal{G},\Lambda), \quad  \mathrm{resp}. ~ D_\et^{\rm ULA}(\mathcal{G},\Lambda), \quad \mathrm{resp}. ~ D_\et^{\omega}(\mathcal{G},\Lambda)
$$
the full subcategory of $D_\et(\mathcal{G},\Lambda)$ formed by reflexive, resp. ULA, resp. compact, objects.
\end{definition} 

Any ULA-object on $\mathcal{G}$ is reflexive, cf. \cite[Corollary IV.2.25]{fargues2021geometrization}\footnote{The statement given there is for morphisms representable in locally spatial diamonds, but extends to this setting, see \cite[\S IV.2.4]{fargues2021geometrization}.}.

\begin{proposition}
\label{fourier-transform-preserves-compact-objects}
Let $S$ be a small v-stack, and let $f: \mathcal{G} \to S$ be a very nice stack in $E$-vector spaces. Then the Fourier transform $\mathcal{F}^{\mathrm{un}}_{\psi}$ induces equivalences
  \[
    \mathcal{F}^{\mathrm{un}}_{\psi}\colon D_{\et}^{\mathrm{ref}}(\mathcal{G},\Lambda)\cong D_\et^{\mathrm{ref}}(\mathcal{G}^\vee,\Lambda), \quad  \mathcal{F}^{\mathrm{un}}_{\psi}\colon D_{\et}^{\omega}(\mathcal{G},\Lambda)\cong D_\et^{\omega}(\mathcal{G}^\vee,\Lambda)
  \]
  between reflexive objects and compact objects.
\end{proposition}
\begin{proof}
It suffices to prove that $\mathcal{F}^{\mathrm{un}}_{\psi}$ preserves reflexive and compact objects. The first statement is implied by the commutation with relative Verdier duality proved in \Cref{sec:stacks-bc-type-properties-of-very-nice-stacks-plus-fourier-transform}. The second statement is formal ($\mathcal{F}^{\mathrm{un}}_{\psi}$ is an equivalence, and thus has both a left and a right adjoint).
\end{proof}

For universally locally acyclic objects, we get the following statement.

\begin{lemma}
  \label{sec:stacks-e-vector-fourier-transform-preserves-ula-objects}
  Assume that $f\colon G\to S$ is a very nice stack in $E$-vector spaces, with dual $f^\vee\colon \mathcal{G}^\vee\to S$. Assume that $f, f^\vee$ are nice, i.e., compactifiable, representable in locally spatial diamonds and of finite geometric transcendence degree. Then the Fourier transform $\mathcal{F}^{\mathrm{un}}_{\psi}$ induces an equivalence
  \[
    \mathcal{F}^{\mathrm{un}}_{\psi}\colon D_{\et}^{\mathrm{ULA}}(\mathcal{G},\Lambda)\cong D_\et^{\mathrm{ULA}}(\mathcal{G}^\vee,\Lambda)
  \]
  between ULA-objects.
\end{lemma}
\begin{proof}
  By our assumption that $f,f^\vee$ are nice, we see that the Fourier transform for $\mathcal{G}$ lifts to an isomorphism in the $2$-category $\mathcal{C}_S$ from \cite[Section IV.2.3.3]{fargues2021geometrization}. In particular, it preserves left adjoints. By \cite[Theorem IV.2.23]{fargues2021geometrization} these are equivalent to ULA-objects. This finishes the proof.  
\end{proof}

The following statements are proved exactly as in \cite[Th\'eor\`eme 1.2.2.4, Corollaire 1.2.25, Proposition 1.2.2.7, Proposition 1.2.2.8]{laumon_transformation_de_fourier}.

\begin{proposition}
\label{fourier-transform-and-push-forward}
Let $f\colon \mathcal{G} \to \mathcal{G}^\prime$ be a morphism of very nice stacks in $\underline{E}$-vector spaces, which is smooth-locally nice. Let $f^\vee : (\mathcal{G}^\prime)^\vee \to \mathcal{G}^\vee$ be the transpose of $f$. Denote by $\mathcal{F}_\psi^{\mathrm{un}}$, resp. $\mathcal{F}_\psi^{\mathrm{un},\prime}$, the Fourier transform functor for $\mathcal{G}$, resp. $\mathcal{G}^\prime$.

One has, on a clopen decomposition of $S$, for each $A \in D_\et(\mathcal{G},\Lambda)$, a functorial isomorphism
$$
\mathcal{F}_\psi^{\mathrm{un},\prime}(Rf_! A) \cong (f^\vee)^\ast \mathcal{F}_\psi^{\mathrm{un}}(A)[d]
$$
for some $d\in \Z$.
\end{proposition}

\begin{corollary}
\label{push-forward-of-the-fourier-transform}
Let $\mathcal{G}$ be a very nice stack in $\underline{E}$-vector spaces, and let $A \in D_\et(\mathcal{G},\Lambda)$. On a clopen decomposition of $S$, one has a functorial isomorphism
$$
f_!^\vee \mathcal{F}_\psi^{\rm un}(A) \cong e^\ast A [d],
$$
where $e: S \to \mathcal{G}$ is the unit section, for some $d\in \Z$.
\end{corollary}

If $f: \mathcal{G}\to S$ is a nice stack in $E$-vector spaces, with its addition law
$$
m : \mathcal{G} \times \mathcal{G} \to \mathcal{G}
$$
(note that via the morphism $\mathcal{G} \times \mathcal{G} \to \mathcal{G} \times \mathcal{G}, (g,h)\mapsto (g,g+h)$, $m$ is isomorphic to the second projection, which is the base change along a representable morphism in Artin v-stacks of a smooth-locally nice map), one defines the \textit{convolution product} 
$$
\ast: D_\et(\mathcal{G},\Lambda) \times D_\et(\mathcal{G},\Lambda) \to D_\et(\mathcal{G}, \Lambda)
$$
by the formula
$$
A \ast B := m_!(A \boxtimes B).
$$

\begin{proposition}
\label{convolution-and-plancherel-formula}
Let $f: \mathcal{G}\to S$ be a very nice stack in $E$-vector spaces. For all $A, B \in D_\et(\mathcal{G},\Lambda)$, one has a functorial isomorphism
$$
\mathcal{F}_\psi^{\rm un}(A \ast B) \cong \mathcal{F}_\psi^{\rm un}(A) \otimes \mathcal{F}_\psi^{\rm un}(B)
$$
and, on a clopen decomposition of $S$, a functorial isomorphism 
$$
f_!^\vee (\mathcal{F}_\psi^{\rm un}(A) \otimes \mathcal{F}_\psi^{\rm un}(B)) \cong f_! (A \otimes [-1]^\ast B)[d],
$$
(``Plancherel formula") where $[-1]$ denotes the ``multiplication by $-1$" map on $\mathcal{G}$, for some $d\in \Z$. 
\end{proposition}

\subsection{Examples of very nice stacks in $E$-vector spaces}
\label{sec:exampl-picard-stacks}
in this subsection, we present two large classes of examples of very nice stacks in $E$-vector spaces. Let as before $S$ be a small v-stack.

\begin{proposition}
\label{description-push-forward-l-psi}
Let $\mathcal{G}=\mathrm{BC}(\mathcal{F})$ be the stack in $E$-vector spaces associated with some $\mathcal{F}\in \Coh^{\mathrm{fl},>0}(S)$. Then $\mathcal{G}$ is very nice. In particular, $\mathcal{G}^\vee=\mathrm{BC}(\mathcal{F})^\vee$ is also very nice.
\end{proposition}
\begin{proof}
  The last assertion follows from \Cref{sec:stacks-bc-type-properties-of-very-nice-stacks-plus-fourier-transform}, so it suffices to show that $\mathcal{G}$ is very nice. We will use the criterion of \Cref{sec:stacks-e-vector-criterion-to-get-involutivity}. Note that $\mathcal{G}, \mathcal{G}^\vee$ are dualizable and cohomologically smooth (and representable in locally spatial diamonds) over $S$, by \Cref{example-dual-of-flat-coh-sheaf-with-positive-slopes}, and the unit sections $e,e^\vee$ are proper. Let $d$ be the dimension of $\mathcal{G}$ (this makes sense by \cite[IV.1.17]{fargues2021geometrization}). First, we will construct a natural map
$$
\pi_! \alpha^\ast(\mathcal{L}_\psi)  \to e_{!} \Lambda[-2d].
$$ 
By proper base change, we have 
$$
e^* (\pi_! \alpha^\ast(\mathcal{L}_\psi)) \cong f_!^\vee \Lambda,
$$
with $f^\vee\colon \mathcal{G}^\vee\to S$ the projection map to the base. Then we have an adjunction morphism
$$
f_!^\vee \Lambda \cong f_!^\vee f^{\vee !} \Lambda[-2d] \to \Lambda[-2d]
$$
(using cohomological smoothness of $\mathcal{G}^\vee$), i.e., a natural map
$$
e^* (\pi_! \alpha^\ast(\mathcal{L}_\psi)) \to \Lambda[-2d],
$$
whence by adjunction (and properness of $e$) a natural map
$$
\pi_! \alpha^\ast(\mathcal{L}_\psi) \to e_! \Lambda[-2d].
$$
To prove that it is an isomorphism, it suffices, since its formation is compatible with base change in $S$, to do so when $S=\mathrm{Spa}(C,C^+)$ is a geometric point. We can even assume that $S=\mathrm{Spa}(C,\mathcal{O}_C)$. In this situation, we can assume that $\mathcal{G}$ is either of the form $\BC(\mathcal{E})$, where $\mathcal{E}$ is a vector bundle having either only positive slopes or only negative slopes, and with $|\deg(\mathcal{E})|=d$, or of the form $\mathbb{B}_{\rm dR, C^\sharp}^+/\mathrm{Fil}^d$ for some untilt $C^\sharp$ of $C$. \red{By excision, it suffices to prove that 
$$
e^* (\pi_! \alpha^\ast(\mathcal{L}_\psi))\cong  \Lambda[-2d] , \quad 
\pi_! \alpha^\ast(\mathcal{L}_\psi)_{|_{\mathcal{G} \backslash \{e\}}}=0.
$$
By proper base change, we have
$$
e^* (\pi_! \alpha^\ast(\mathcal{L}_\psi))= f_!^\vee \Lambda,
$$
so for the first part, we want to show that
$$
f_!^\vee \Lambda \cong \Lambda[-2d].
$$}
In the vector bundle case, see the proof of \cite[Proposition V.2.1]{fargues2021geometrization}. In the torsion case, induction on $d$ reduces to the case $d=1$ and the computation of the compactly supported cohomology of $\mathbb{A}_{C^\sharp}^1$. \red{For the second part, i.e. to show that $\pi_! \alpha^\ast(\mathcal{L}_\psi)_{|_{\mathcal{G} \backslash \{e\}}}$ vanishes, by checking on fibers and using proper base change again, we see that we need to show that
$$
R\Gamma(E, R\Gamma_c(\BC(\mathcal{F}^\prime), \Lambda) \otimes \psi) = 0
$$
if $\mathcal{F}^\prime$ is a \textit{non-split} extension of $\mathcal{F}$ by $\mathcal{O}$. In both cases, $\mathcal{F}^\prime$ is the direct sum of a vector bundle with positive slopes and a torsion sheaf, so the previous computation shows its compactly supported cohomology has a trivial $E$-action (indeed, this action extends to an action of $\BC(\mathcal{F}^\prime)$, which must be trivial since the latter is connected and $\Lambda$ is totally disconnected). Hence these cohomology groups vanish.}

The argument for $\pi_!^\vee \alpha^\ast\mathcal{L}_\psi\cong e_!^\vee \Lambda[-2d]$ is entirely similar.
\end{proof}  

The other class of examples we consider are associated with pro-\'etale local systems.

\begin{proposition}
  \label{sec:examples-very-nice-local-systems-are-very-nice}
  Let $S$ be a small v-stack. Let $\mathcal{G}=\mathbb{L}$ be a pro-\'etale $\underline{E}$-local system on $S$. Then $\mathcal{G}, \mathcal{G}^\vee$ are very nice stacks in $E$-vector spaces over $S$. 
\end{proposition}
\begin{proof}
By \Cref{sec:stacks-bc-type-properties-of-very-nice-stacks-plus-fourier-transform}, it suffices to show that $\mathcal{G}$ is very nice. Dualizability has been proved in \Cref{example-dual-of-pro-etale-local-systems}. We need to show that $\mathcal{F}^{\mathrm{un}}_{\psi,!,\mathcal{G}\to \mathcal{G}^\vee}$, $\mathcal{F}^{\mathrm{un}}_{\psi^{-1},!,\mathcal{G}^\vee\to \mathcal{G}}$ are equivalences. The functors $U \mapsto D_\et(\mathcal{G} \times_S U,\Lambda), U \mapsto D_\et(\mathcal{G}^\vee \times_S U,\Lambda)$ are sheaves of $\infty$-categories on the v-site of $S$ and (\Cref{compatibility-of-ft-with-base-change}) $\mathcal{F}^{\mathrm{un}}_{\psi,!,\mathcal{G}\to \mathcal{G}^\vee}$, $\mathcal{F}^{\mathrm{un}}_{\psi^{-1},!,\mathcal{G}^\vee \to \mathcal{G}}$ define morphisms of sheaves. It is therefore enough to show that they are equivalences under the additional assumption that $\mathbb{L}$ is trivial. By base change, we may then reduce to the case that $S=\mathrm{Spa}(C,C^+)$ is a geometric point.
If $S$ is a geometric point (and thus $\mathbb{L}$ trivial), $\mathcal{F}^{\mathrm{un}}_{\psi,!,\mathcal{G}\to \mathcal{G}^\vee}$ and $\mathcal{F}^{\mathrm{un}}_{\psi^{-1},!,\mathcal{G}^\vee \to \mathcal{G}}$ can be made entirely explicit and are seen to be equivalences: we postpone the discussion of this to \Cref{unraveling-fourier-transform-for-finite-dimensional-vector-spaces}, \Cref{end-of-proof-push-forward-of-l-psi} below.
 
 It remains to check point (3) of \Cref{sec:stacks-bc-type-definition-very-nice}. Since we already know that the two Fourier transforms are equivalences, it is enough to show that 
 $$
 \mathcal{F}^{\mathrm{un}}_{\psi^{-1},!,\mathcal{G}^\vee \to \mathcal{G}} \circ \mathcal{F}^{\mathrm{un}}_{\psi,!,\mathcal{G}\to \mathcal{G}^\vee}
 $$
 is isomorphic (up to shift) to the identity functor. We apply the criterion of \Cref{sec:stacks-e-vector-criterion-to-get-involutivity}. Note that $f^\vee: \mathcal{G}^\vee \to S$ is cohomologically smooth (of dimension $0$) and that $e: S \to \mathcal{G}$ is proper. Therefore we get as in \Cref{description-push-forward-l-psi} a morphism
\[
  \pi_!(\alpha^\ast \mathcal{L}_\psi)\to e_!\Lambda
\]
and the arguments there show that it is an isomorphism if this is true in the case that $S=\Spa(C,C^+)$ is a geometric point. In this case, $\mathbb{L}$ is trivial and the assertion is easy.
\end{proof}

\begin{remark}
\label{stacks-of-bc-type-and-problem}
Let $S$ be a small v-stack. To bring together the last two classes of examples, one could more generally consider stacks in $E$-vector spaces $\mathcal{G}$ admitting v-locally on $S$ a $2$-step filtration $W_\bullet \mathcal{G}$ 
$$
W_{-1}=0 \subset W_0 \subset W_1 \subset W_2=\mathcal{G}
$$
such that $\mathrm{Gr}_0^W \cong [S/\mathbb{L}]$ is the classifying stack of a pro-\'etale $\underline{E}$-local system, $\mathrm{Gr}_1^W$ is either $\BC(\mathcal{F})$ or $\BC(\mathcal{F})^\vee$, with $\mathcal{F} \in \Coh^{\rm fl, >0}(S)$ and $\mathrm{Gr}_2^W=\mathbb{L}^\prime$ is a pro-\'etale $\underline{E}$-local system.

This definition would be reminiscent of the definition of Beilinson's $1$-motives, see e.g. \cite[Appendix A]{chen_zhu_geometric_langlands_in_prime_characteristic} (cf. also Laumon's generalized $1$-motives of \cite{laumon_transformation_de_fourier_generalisee}), with Banach-Colmez spaces of flat coherent sheaves having positive slopes and their duals being an analogue in this setting of abelian schemes (they behave formally similarly with respect to duality).

One can check, using \Cref{fully-faithfulness-R-tau-ast}, that for such a $\mathcal{G}$, there is locally for the v-topology on $S$, a decomposition
$$
\mathcal{G} \cong \mathrm{Gr}_0^W \mathcal{G} \times \mathrm{Gr}_1^W \mathcal{G} \times \mathrm{Gr}_2^W \mathcal{G}.
$$
However, we do not know if the property of being smooth-locally nice is v-local on target, and therefore we do not know if any such $\mathcal{G}$ is very nice.
\end{remark}

\subsection{An inductive construction}
\label{sec-an-inductive-construction}

Let $i\geq 0$ an integer. 

\begin{definition}
\label{coh-prime-and-coh-prime-ex}
Let $\mathcal{C}_i$ be the open substack 
$$
\mathcal{C}_i = \Coh_i^{\rm fl,>0} \cup \Bun_i^{\mathbf{1}} \hookrightarrow \Coh_i^{\rm fl}.
$$
Let
$$
\mathcal{C}_{i, \rm ex}^\prime, \quad \mathrm{resp.} ~ \mathcal{C}_i^\prime,
$$
be the stack sending $S \in \mathrm{Perf}_{\mathbb{F}_q}$ to the groupoid of pairs $$\left(\mathcal{F} \in \mathcal{C}_i(S),s: \mathcal{O}_{X_S} \to \mathcal{F}\right),$$ where $s$ is an $\mathcal{O}_{X_S}$-linear map, resp. a fiberwise injective $\mathcal{O}_{X_S}$-linear map. We denote by $j_i$ the open embedding of $\mathcal{C}_i^\prime$ in $\mathcal{C}_{i, \rm ex}^\prime$ and by $\pi_i: \mathcal{C}_{i,\rm ex}^\prime \to \mathcal{C}_i$ the forgetful map sending a pair $(\mathcal{F},s)$ to $\mathcal{F}$.
\end{definition}

Using \Cref{description-push-forward-l-psi} and \Cref{sec:examples-very-nice-local-systems-are-very-nice}, we see that the morphism $\pi_i$ makes $\mathcal{C}_{i,\rm ex}^\prime$ a very nice stack in $E$-vector spaces over the small v-stack $\mathcal{C}_i$. The fiber of $\pi_i$ over a flat coherent sheaf $\mathcal{F}$ in $\mathcal{C}_i$ is $\BC(\mathcal{F})$. Let
$$
\pi_i^\vee: \mathcal{C}_{i,\rm ex}^{\prime \vee} \to \mathcal{C}_i
$$
be the dual stack in $E$-vector spaces. By construction and \Cref{fully-faithfulness-R-tau-ast}, $\mathcal{C}_{i, \rm ex}^{\prime \vee}$ sends $S \in \mathrm{Perf}_{\mathbb{F}_q}$ to the groupoid of extensions
$$
0 \to \mathcal{O}_{X_S} \to \mathcal{F}^\prime \to \mathcal{F} \to 0.
$$

Write $\mathcal{C}_{i, \rm ex,0}^\prime= \mathcal{C}_{i, \rm ex}^\prime \times_{\mathcal{C}_i} \Bun_i^{\mathbf{1}}$ and $\mathcal{C}_{i, \rm ex,d}^\prime= \mathcal{C}_{i, \rm ex}^\prime \times_{\mathcal{C}_i} \Coh_{i,d}^{\rm fl, >0}$ for $d>0$. We have a direct product decomposition
$$
D_\et(\mathcal{C}_{i, \rm ex}^\prime,\Lambda) \cong \prod_{d\geq 0} D_\et(\mathcal{C}_{i, \rm ex, d}^\prime,\Lambda).
$$
We will denote by 
$$
\mathcal{F}_{\mathcal{C}_{i,\rm ex}^\prime,\psi}: D_\et(\mathcal{C}_{i, \rm ex}^\prime,\Lambda)  \to D_\et(\mathcal{C}_{i, \rm ex}^{\prime \vee},\Lambda) 
$$
defined to be 
$$
\left(\mathcal{F}_{\psi,!,\mathcal{C}_{i,\rm ex}^\prime \to \mathcal{C}_{i, \rm ex}^{\prime \vee}}^{\rm un} \right)_{|_{D_\et(\mathcal{C}_{i, \rm ex, d}^\prime,\Lambda)}}[d]
$$
in restriction to $D_\et(\mathcal{C}_{i, \rm ex, d}^\prime,\Lambda)$, $d\geq 0$. 

Let
$$
j_i^{\vee}: \mathcal{C}_i^{\prime \vee} \to \mathcal{C}_{i,\rm ex}^{\prime \vee}
$$
be the open substack defined by the condition that $\mathcal{F}^\prime \in \mathcal{C}_{i+1}$. The morphism sending $(s:\mathcal{O}_{X_S} \hookrightarrow \mathcal{F}^\prime) \in \mathcal{C}_{i+1}^\prime(S)$ to
$$
(0 \to \mathcal{O}_{X_S} \to \mathcal{F}^\prime \to \mathcal{F} \to 0) \in \mathcal{C}_i^{\prime \vee}(S),
$$
where $\mathcal{F}$ is defined to be the cokernel of $s$ then (tautologically) defines an isomorphism
$$
\alpha_i: \mathcal{C}_{i+1}^\prime \cong \mathcal{C}_i^{\prime \vee}.
$$ 

\begin{definition}
\label{def-the-functor-phi-plus}
Let $i\geq 0$. Using the above notations, we define a functor
$$
\Phi_{\mathcal{C}_i}^+:  D_\et(\mathcal{C}_i^\prime,\Lambda) \to D_\et(\mathcal{C}_{i+1}^\prime, \Lambda)
$$
by the formula
$$
\Phi_{\mathcal{C}_i}^+ =\alpha_i^{-1} \circ (j_i^\vee)^\ast \circ \mathcal{F}_{\mathcal{C}_{i,\rm ex}^\prime,\psi} \circ j_{i,!}.
$$
\end{definition}

We can now iterate this construction, increasing by one the generic rank at each step. For $n\geq 1$, denote by
$$
\Bun_n^\prime
$$
the fiber product of $\Bun_n$ with $\mathcal{C}_n^\prime$ over $\mathcal{C}_n$ and by $\iota_n$ the open embedding
$$
\iota_n: \Bun_n^\prime \hookrightarrow \mathcal{C}_n^\prime.
$$

\begin{definition}
\label{def-the-laumon-construction}
Let $n\geq 1$. We define a functor
$$
\mathcal{A}_{n,\psi}: D_\et(\mathcal{C}_0,\Lambda) \to D_\et(\Bun_n^\prime,\Lambda)
$$
by the formula
$$
\mathcal{A}_{n,\psi} = \iota_n^{-1} \circ  \Phi_{\mathcal{C}_{n-1}}^+ \circ \dots \circ  \Phi_{\mathcal{C}_1}^+ \circ  \alpha_0^{-1} \circ (\pi_0 \circ j_0^\vee)^\ast.
$$
\end{definition}
 
The definition of $\mathcal{A}_{n,\psi}$ (with different notations) is due to Laumon, \cite{laumon_duke}, \cite{laumon_premiere_construction_de_drinfeld}, inspired by constructions of Drinfeld \cite{drinfeld_two_dimensional_l_adic_representations} in the case $n=2$, in the classical setting of a smooth projective curve over a finite field, who applied it to the geometric Langlands program, and is our main motivation for introducing the Fourier transform studied in this paper. Analogously to what Drinfeld, Laumon and Frenkel-Gaitsgory-Vilonen did in the function field setting, we hope that it is possible, starting from a continuous irreducible rank $n$ representation $\mathbb{L}$ of $W_E$ over $\Lambda$ (assumed to be a field), to produce an object
$$
\mathcal{L}_{\mathbb{L}} \in D_\et(\mathcal{C}_0,\Lambda)
$$
such that
$$
\mathcal{A}_{n,\psi}(\mathcal{L}_{\mathbb{L}}) \in  D_\et(\Bun_n^\prime,\Lambda)
$$
descends\footnote{This is to be understood here as saying that for each $d\geq 0$, the pullback to $\Bun_n^\prime$ of the restriction of $\mathrm{Aut}_{\mathbb{L}}$ to the degree $d$ connected component $\Bun_n^d$ of $\Bun_n$ shifted by $d$ is isomorphic to the restriction in degree $d$ of $\mathcal{A}_{n,\psi}(\mathcal{L}_{\mathbb{L}})$.} along the natural map $\Bun_n^\prime \to \Bun_n$ to an object $\mathrm{Aut}_{\mathbb{L}}$ of $D_\et(\Bun_n,\Lambda)$ satisfying the requirements of Fargues' geometrization conjecture.

\section{Examples}
\label{sec:examples}

In this section, we would like to discuss various concrete examples of the Fourier transform introduced in the first section. 

\subsection{The case of finite dimensional $E$-vector spaces}
\label{sec:the-case-of-E-vector-spaces}
Before starting, let us recall the following lemma.

\begin{lemma}
  \label{sec:relat-with-bernst-d-et-for-quotients-of-locally-profinite-sets-by-locally-profinite-groups}
  Let $X$ be a locally profinite set (seen as a small v-sheaf over $\mathrm{Spa}(\mathbb{F}_q)$), and let $H$ be a locally pro-$p$ group acting on $X$. Then
  \[
    D_\et([X/H],\Lambda)
  \]
  is equivalent to the derived category of smooth\footnote{A $C_c^\infty(X,\Lambda)$-module $M$ is said to be \textit{smooth} if the natural map $\varinjlim (\mathbf{1}_U.M) \to M$ is an isomorphism, where $U$ runs over compact open subsets of $X$ and the transition maps are given by the natural inclusions.} $C_c^\infty(X,\Lambda)$-modules ($C_c^\infty(X,\Lambda)$ seen as a ring via multiplication of functions) with a semilinear, smooth $H$-action.
\end{lemma}
\begin{proof}
See \cite[Appendix B.2]{kaletha_weinstein_on_the_kottwitz_conjecture_for_local_shimura_varieties}.
\end{proof}

In the next proposition, we take $S=\mathrm{Spa}(\mathbb{F}_q)$ (we could also take $S=\mathrm{Spa}(C)$, with $C$ complete algebraically closed, if we want $S$ to be a perfectoid space).

\begin{proposition}
\label{unraveling-fourier-transform-for-finite-dimensional-vector-spaces}
Let $V$ be a finite dimensional $E$-vector space. Fix a Haar measure $d\check{v}$ on $V^\vee$. Let $\mathcal{G}=\underline{V}$ be the associated stack in $E$-vector spaces, and let
$$
\mathcal{F}_\psi := \mathcal{F}_{\psi, !, \mathcal{G} \to \mathcal{G}^\vee}^{\rm un}
$$
 be the corresponding Fourier transform. The functor 
$$
\mathcal{F}_\psi: D_\et(\mathcal{G},\Lambda) \to D_\et(\mathcal{G}^\vee, \Lambda)
$$
is induced, via the identification of $D_\et(\mathcal{G},\Lambda)$, resp. of $D_\et(\mathcal{G}^\vee, \Lambda)$, with the derived category of smooth $(C_c^\infty(V,\Lambda), \times)$-modules, resp. of smooth $(C_c^\infty(V^\vee,\Lambda), \ast)$-modules, coming from \Cref{sec:relat-with-bernst-d-et-for-quotients-of-locally-profinite-sets-by-locally-profinite-groups}, resp. from the choice of $d\check{v}$ and \cite[Theorem V.1.1]{fargues2021geometrization}, by the isomorphism
$$
(C_c^\infty(V^\vee, \Lambda), \ast) \cong (C_c^\infty(V,\Lambda), \times), 
$$
given by the ``naive" Fourier transform, sending $f \in C_c^\infty(V^\vee, \Lambda)$ to the element $\widehat{f} \in C_c^\infty(V, \Lambda)$, such that
$$ \forall v\in V, ~ \widehat{f}(v)= \int_{V^\vee} f(\check{v}) \psi(\langle \check{v}, v \rangle) d\check{v}.
$$
\end{proposition}
\begin{proof}
Let $M$ be a smooth $C_c^\infty(V,\Lambda)$-module. We would like to describe the smooth module over the convolution algebra for $V^\vee$ defined by $\mathcal{F}_\psi(M)$ explicitly. Recall that if $\pi$ is a smooth representation of $V^\vee$, it corresponds to a smooth module over the convolution algebra, with module structure described the formulas
$$
\forall f \in C_c^\infty(V^\vee, \Lambda), \forall w \in \pi, ~ f \cdot w= \int_{V^\vee} f(\check{v}) \pi(\check{v})\cdot w d\check{v}.
$$
Hence, it suffices to prove that the $V^\vee$-action on the smooth representation of $V^\vee$ corresponding to $\mathcal{F}_\psi(M)$ is given by $M$, endowed with the action of $V^\vee$ defined by
$$
\forall \check{v} \in V^\vee, \forall m\in M, ~ \check{v} \cdot m= \psi(\langle \check{v}, - \rangle)\cdot m.
$$
To do so, it suffices to prove that the tensor product of the pullback of $M$, seen as an object of $D_\et(\mathcal{G},\Lambda)$, along the morphism 
$$
\pi : \mathcal{G}^\vee \times \mathcal{G} \to \mathcal{G}
$$
with $\alpha^\ast \mathcal{L}_\psi$ corresponds, via \Cref{sec:relat-with-bernst-d-et-for-quotients-of-locally-profinite-sets-by-locally-profinite-groups}, to the smooth $C_c^\infty(V,\Lambda)$-module $M$, together with the semilinar action of $V^\vee$ just described: indeed, the functor $\pi_!^\vee$, with
$$
\pi^\vee : \mathcal{G}^\vee \times \mathcal{G} \to \mathcal{G}^\vee
$$
is simply given by forgetting the $C_c^\infty(V,\Lambda)$-module structure. 

But the map
$$
\alpha: \mathcal{G}^\vee \times \mathcal{G} \to [\ast/\underline{E}]
$$
is induced by the natural pairing $\langle -,- \rangle: V^\vee \times V \to E$. Therefore, $\alpha^\ast \mathcal{L}_\psi$ corresponds to the smooth $C_c^\infty(V,\Lambda)$-module $C_c^\infty(V,\Lambda)$, together with the action of $V^\vee$ written above, and this gives the description we want.
\end{proof}

\begin{remark}
\label{end-of-proof-push-forward-of-l-psi}
On the one hand, by \Cref{unraveling-fourier-transform-for-finite-dimensional-vector-spaces} and its proof, we see that for $\mathcal{G}$ as above, $\mathcal{F}_\psi(\Lambda)$ corresponds to the $\Lambda$-module $C_c^\infty(V,\Lambda)$ endowed with the smooth action of $V^\vee$ defined by
$$
\forall \check{v} \in V^\vee, \forall f\in C_c^\infty(V,\Lambda), ~ \check{v} \cdot f= \psi(\langle \check{v}, - \rangle). f.
$$
On the other hand, the $!$-pushforward along the unit section
$$
e: \ast \to \mathcal{G}^\vee \cong [\ast/V^\vee]
$$
of the constant sheaf $\Lambda$ corresponds to the $\Lambda$-module $C_c^\infty(V^\vee,\Lambda)$ with its $V^\vee$-action by right translations. The ``naive" Fourier transform provides a functorial identification of one with the other. We have therefore finished the proof of \Cref{sec:examples-very-nice-local-systems-are-very-nice}.
\end{remark}

In fact, it is interesting to take into account automorphisms in the above. More precisely, fix $V$ a finite dimensional $E$-vector space. Let 
$$
S= [\ast/\underline{\GL(V)}]
$$
and consider the very nice stack in $E$-vector spaces
$$
\mathcal{G} =  [\underline{V}/\underline{\GL(V)}] 
$$
over $S$ (with $GL(V)$ acting on $V$ via its natural action), with its associated Fourier transform $\mathcal{F}_\psi$. The diagram defining $\mathcal{F}_\psi$ is
\[
  \xymatrix{
    & [\underline{V}/\underline{P(V)}] \ar[r]^\alpha\ar[ld]_{\pi^\vee}\ar[rd]^\pi& [\ast/\underline{E}] \\
    [\ast/\underline{P(V)}] & & [\underline{V}/\underline{\GL(V)}]
  }
\]
where
$$
P(V) = \GL(V) \ltimes V^\vee 
$$
is the \textit{mirabolic group} (of size $\dim(V)+1$), with its natural action on $V$, and $\alpha$ is the composition:
$$
[\underline{V}/\underline{P(V)}]=[\underline{V}/\underline{\GL(V)} \ltimes \underline{V^\vee}] \to [\ast/\underline{\GL(V)} \times \underline{E}] \to [\ast/\underline{E}]
$$
induced by the natural pairing
$$
V \times [\ast/V^\vee] \to [\ast/E].
$$

\begin{remark}
\label{relation-with-inductive-construction}
Take $V=V_i:=E^i$. The stack $\mathcal{G}$, resp. the above diagram, is nothing but the pullback of the stack in $E$-vector spaces $\mathcal{C}_i^\prime$ introduced in \Cref{sec-an-inductive-construction}, resp. of the diagram defining the Fourier transform for $\mathcal{C}_i^\prime$, along the open embedding
$$
\Bun_i^{\bf 1} \hookrightarrow \mathcal{C}_i.
$$
\end{remark}

As above, we can also make $\mathcal{F}_\psi$ and $\mathcal{F}_\psi^\vee$ explicit, using \Cref{sec:relat-with-bernst-d-et-for-quotients-of-locally-profinite-sets-by-locally-profinite-groups}. In particular, $\alpha^* \mathcal{L}_{\psi} \in D_\et(V/P(V),\Lambda)$ corresponds to the smooth $C_c^\infty(V,\Lambda)$-module $L_\psi=C_c^\infty(V,\Lambda)$ itself, endowed with the semilinear smooth action of $P(V)$ given by
$$
(g,\check{v}).f := (v \mapsto \psi(\check{v}(g^{-1}.v)) f(g^{-1}.v)).
$$
An \'etale sheaf
$$
A \in D_\et([\underline{V}/\underline{\GL(V)}],\Lambda)
$$
 is the same thing as a smooth $C_c^\infty(V,\Lambda)$-module $M$, with a semilinear smooth $\GL(V)$-action, and an \'etale sheaf
 $$
 B \in D_\et([\ast/\underline{P(V)}],\Lambda)
 $$
 is the same thing as a smooth $P(V)$-representation $N$ on a $\Lambda$-module. It is easy to see that in these terms, $\mathcal{F}_{\psi}(A) \in D_\et([\ast/\underline{P(V)}],\Lambda)$ corresponds to the $P(V)$-representation
$$
M \otimes_{C_c^\infty(V,\Lambda)} L_\psi,
$$
endowed with its diagonal $P(V)$-action (acting through its quotient $GL(V)$ on $M$). Conversely, $\mathcal{F}_{\psi}^\vee(B) \in D_\et([\underline{V}/\underline{GL(V)}],\Lambda)$ corresponds to 
$$
(N \otimes_{\Lambda} L_\psi)_{V^\vee}.
$$
Fix $i\geq 0$, and take $V=V_i:=E^i$. To remember the index $i$, the Fourier transform and inverse Fourier transform will be denoted by $\mathcal{F}_{i,\psi}$ and $\mathcal{F}_{i,\psi}^\vee$ respectively. Let $j_i$ be the open immersion
$$
j_i: [\ast/\underline{P(V_{i-1})}]\cong [(\underline{V_i \backslash \{0\}})/\underline{\GL(V_i)}] \to [\underline{V_i}/\underline{\GL(V_i)}]
$$
and let $\iota_i$ be its closed complement
$$
\iota_i: [\ast/\underline{\GL(V_i)}] \to [\underline{V_i}/\underline{\GL(V_i)}].
$$
Define functors
$$ 
\Phi_i^+:=\mathcal{F}_{i,\psi} \circ j_{i,!} : D_\et([\ast/\underline{P(V_{i-1})}],\Lambda) \to D_\et([\ast/\underline{P(V_i)}],\Lambda), 
$$
$$ 
\Phi_i^-:= j_i^* \circ \mathcal{F}_{i,\psi}^\vee : D_\et([\ast/\underline{P(V_i)}],\Lambda) \to D_\et([\ast/\underline{P(V_{i-1})}],\Lambda),
$$
$$
\Psi_i^+:= \mathcal{F}_{i,\psi} \circ \iota_{i,\ast} : D_\et([\ast/\underline{\GL(V_i)}],\Lambda) \to D_\et([\ast/\underline{P(V_i)}],\Lambda), 
$$
$$
\Psi_i^-:=\iota_i^* \circ \mathcal{F}_{i,\psi}^\vee : D_\et([\ast/\underline{P(V_i)}],\Lambda) \to D_\et([\ast/\underline{\GL(V_i)}],\Lambda).
$$
It is easy to see, using the above explicit formulas, that these functors coincide with the Bernstein-Zelevinsky functors from \cite{bernstein_zelevinsky_representations_of_the_group_glnf}. Moreover, the properties of Bernstein-Zelevinsky functors are easily recovered using this reformulation.
For example, the exact sequence of functors on $D_\et([\ast/\underline{P(V_i)}],\Lambda)$
$$
0 \to \Phi_i^+ \circ \Phi_i^- \to \mathrm{Id} \to \Psi_i^+ \circ \Psi_i^- \to 0, 
$$
is deduced from the short exact sequence of functors on $D_\et([\underline{V_i}/\underline{\GL(V_i)}],\Lambda)$
$$
0 \to j_{i,!} \circ j_{i}^* \to \mathrm{Id} \to \iota_{i,*} \circ \iota_{i}^* \to 0.
$$
Similarly, the relations
$$
\Phi_i^- \circ \Psi_i^+ = 0, ~~ \Psi_i^- \circ \Phi_i^+=0,
$$
follow from
$$ 
j_{i}^{*} \circ \iota_{i,*}= 0 ~~ , ~~ \iota_{i}^* \circ j_{i,!}=0,
$$
and
$$
\Phi_i^- \circ \Phi_i^+ = \mathrm{Id}, ~~ \Psi_i^- \circ \Psi_i^+=\mathrm{Id}. 
$$
since 
$$
j_{i}^* \circ j_{i,!}= \mathrm{Id} ~~ , ~~ \iota_{i}^* \circ \iota_{i,*}= \mathrm{Id}.
$$

\begin{remark}
Continuing \Cref{relation-with-inductive-construction}, we see that for any $A \in D_\et(\Coh_0,\Lambda)$ whose pullback to $\Bun_0^{\bf 1}=\{0\}$ is isomorphic to $\Lambda$, the pullback of $\mathcal{A}_{n,\psi}(A)$ (notations of \Cref{sec-an-inductive-construction}) to $\Bun_n^{\prime, \bf 1}\cong [\ast/P(V_{n-1})]$ corresponds to the smooth representation of the mirabolic of size $n$
$$
\Phi_{n-1}^+ \circ \dots \Phi_1^+ \circ \Psi_0^+(1),
$$
where $1$ stands for the trivial representation. This is Gelfand-Kazhdan's description of the Kirillov model of an irreducible supercuspidal representation of $\GL_n(E)$.
\end{remark}

\subsection{The case of the affine line in characteristic $0$}
\label{the-case-of-the-affine-line-in-characteristic-zero}
Let $S=\mathrm{Spa}(C,\mathcal{O}_C)$, with $C$ an algebraically closed perfectoid field of characteristic $p$. Fix an untilt $C^\sharp$ of $C$ over $E$, given by $t \in B_C^{\varphi=\pi}$ and corresponding to a rigid point $x \in X_C$. Let $\mathcal{G}$ be the very nice stack in $E$-vector spaces attached to the coherent sheaf $i_{x,\ast} C^\sharp$. Then
$$
\mathcal{G} \cong \mathbb{A}_{C^\sharp}^{1,\diamond}
$$
and also 
$$
\mathcal{G}^\vee \cong \mathbb{A}_{C^\sharp}^{1,\diamond}.
$$
Therefore, the Fourier transform
$$
\mathcal{F}_\psi := \mathcal{F}_{\psi,!,\mathcal{G} \to \mathcal{G}^\vee}^{\rm un}[1]
$$
defines an equivalence of categories
$$
\mathcal{F}_\psi : D_\et(\mathbb{A}_{C^\sharp}^{1,\diamond},\Lambda) \overset{\sim}\longrightarrow D_\et(\mathbb{A}_{C^\sharp}^{1,\diamond},\Lambda).
$$

\begin{proposition}
\label{relation-with-ramero}
The Fourier transform $\mathcal{F}_\psi$ coincides with the functor induced on left-completions by the Fourier transform introduced and studied by Ramero, \cite{ramero1998class}, \cite{ramero2007hasse}.
\end{proposition}
Recall, cf. \cite[Proposition 14.15]{scholze_etale_cohomology_of_diamonds}), that if $X$ is any locally spatial diamond, the category $D_\et(X,\Lambda)$ is the left-completion of the category $D(X_\et,\Lambda)$. Moreover, if $X=Y^\diamond$, for an analytic adic space $Y$ over $\Z_p$, $D(X_\et, \Lambda) \cong D(Y_\et, \Lambda)$ (\cite[Lemma 15.6]{scholze_etale_cohomology_of_diamonds}). This explains the statement of the proposition.
\begin{proof}
Let $\mathrm{LT}$ be the Lubin-Tate formal group law for $E$ (the unique up to isomorphism $1$-dimensional formal group over $\mathcal{O}_{\breve{E}}$ with action of $\mathcal{O}_E$, such that the two actions of $\mathcal{O}_E$ induced on the Lie algebra coincide). The kernel of Ramero's Fourier transform is defined via the logarithm exact sequence
$$
0 \to \underline{E} \to \widetilde{\mathrm{LT}}_{C^\sharp} \overset{\rm log} \longrightarrow \mathbb{A}_{C^\sharp}^1 \to 0.
$$
It agrees with the kernel appearing in the definition of $\mathcal{F}_\psi$, which is defined via the exact sequence
$$
0 \to \underline{E} \overset{\times t} \longrightarrow \BC(\mathcal{O}(1)) \to \mathbb{A}_{C^\sharp}^{1,\diamond} \to 0,
$$
since the above two exact sequences are known to identify, when seen as sequences of sheaves on $\mathrm{Perf}_C$ (cf. e.g. \cite[Proposition II.2.2, Proposition II.2.3]{fargues2021geometrization}).
\end{proof}

\begin{remark}
\label{ramero-preservation-of-perversity}
In this specific situation, Ramero, \cite[Theorem 2.3.30]{ramero2007hasse}, was able to show that the Fourier transform preserves ``perverse sheaves with bounded ramification", without really defining what a perverse sheaf is in this setting though. Recently, Bhatt-Hansen, \cite{bhatt_hansen_the_six_functors_for_zariski_constructible_sheaves_in_rigid_geometry}, defined a category of $\Lambda$-perverse Zariski-constructible sheaves on any rigid space over a non-archimedean field of characteristic $0$, but for torsion coefficients, their theory is limited to the case where $\Lambda$ is \textit{finite}, and therefore does not seem to be directly applicable to the study of properties of the Fourier transform.
\end{remark}

\subsection{The case $\mathcal{G}=\BC(\mathcal{O}(1))$}
\label{sec:case-from-bcmathc}
In this subsection, we take
$$
\mathcal{G}=\BC(\mathcal{O}(1)),
$$
so that $\mathcal{G}^\vee=\BC(\mathcal{O}(-1))$. We set
$$
\mathcal{F}_\psi := \mathcal{F}_{\psi,!,\mathcal{G} \to \mathcal{G}^\vee}[1].
$$
We will also set
$$
j: \mathcal{G}^\circ :=\mathcal{G}\backslash \{0\} \hookrightarrow \mathcal{G}, ~ i: \{0\} \hookrightarrow \mathcal{G}
$$
and
$$
j^\vee: \mathcal{G}^{\vee,\circ} :=\mathcal{G}^\vee\backslash \{0\} \hookrightarrow \mathcal{G}^\vee, ~ i^\vee: \{0\} \hookrightarrow \mathcal{G}^\vee.
$$

The goal of this subsection is the proof of \Cref{computation-of-the-rank-fourier-transform-explicit-formula} below, using Huber's adic version of the Grothendieck-Ogg-Shafarevich formula. We start by establishing a finiteness result, \Cref{sec:case-from-bcmathc-2-ft-preserves-finiteness-of-stalks}, which uses crucially the fact that the small v-sheaves (actually diamonds) $\mathcal{G}^\circ, \mathcal{G}^{\vee,\circ}$ are qcqs (but recall that the morphisms $\mathcal{G}^\circ\to \ast, \mathcal{G}^{\vee,\circ}\to \ast$ are not quasicompact, because the v-sheaf $\ast$ is not quasi-separated!).

\begin{proposition}
  \label{sec:case-from-bcmathc-4-compact-for-v}
  An object in $D_{\et}(\mathcal{G},\Lambda)$ is compact if and only if each stalk over a geometric point of $\mathcal{G}$ is a perfect complex.
\end{proposition}

Note that this proposition is not covered by \cite[Proposition 20.17]{scholze_etale_cohomology_of_diamonds} as $\mathcal{G}$ is not a spatial diamond.

\begin{proof}
We can choose an isomorphism
  \[
    \mathcal{G}\cong \mathrm{Spd}(\Fpbar[[t^{1/p^\infty}]]),
  \]
  cf.\ \cite[Proposition II.2.2]{fargues2021geometrization}.
  Via pullback along the morphism of sites
  \[
    \Spd(\Fpbar[[t^{1/p^\infty}]])_{\et}\to \Spec(\Fpbar[[t]])_{\et},
  \]
  one derives an equivalence
  \[
    D_\et(\mathcal{G},\Lambda)\cong D(\Spec(\Fpbar[[t]])_\et,\Lambda),
  \]
  compatible with $i_\ast, i^!, j_\ast, j_!$. Indeed, compatibility of the $i_\ast, i^!, j_\ast, j_!$ follows from compatility with $i^\ast, j^\ast$ and adjunctions. Arguing on the strata $\{0\},\mathcal{G}^\circ$ therefore yields the equivalence. Then one can apply \cite[Proposition 6.4.8]{bhatt_scholze_the_pro_etale_topology_for_schemes} for the right-hand side.
\end{proof}

Let us prove the analog for $\mathcal{G}^\vee$. Here the situation is more complicated as 
\[
  \mathcal{G}^{\vee,\circ}\cong \Spd(\Fpbar((s^{1/p^\infty})))/H
\]
with
\[
  H:=\ker(\mathcal{O}_D^\times\xrightarrow{\mathrm{Nrd}} \Z_p^\times),
\]
is more complicated than $\mathcal{G}^\circ\cong \Spd(\Fpbar((t^{1/p^\infty})))$.

\begin{proposition}
  \label{sec:case-from-bcmathc-4-compact-for-w}
  An object in $D_{\et}(\mathcal{G}^\vee,\Lambda)$ is compact if and only if each stalk over a geometric point of $\mathcal{G}^\vee$ is a perfect complex.
\end{proposition}

We note that $|\mathcal{G}^\vee|$ has exactly two points (similarly to $|\mathcal{G}|$).

\begin{proof}
  We first note that
  \[
    \mathcal{G}^{\vee,\circ}\cong \Spd(\Fpbar((s^{1/p^\infty})))/H
  \]
  has finite $\ell$-cohomological dimension. Indeed, we can write
  \[
    \mathcal{G}^{\vee,\circ}\cong \Spd(\widehat{\overline{\Fpbar((s^{1/p^\infty}))}})/G
  \]
  where $G$ is a profinite group, which is an extension of $H$ by the Galois group of $\Fpbar((s^{1/p^\infty}))$.
  Then we can apply (the proof of) \cite[Proposition 21.16]{scholze_etale_cohomology_of_diamonds}, and conclude\footnote{Note that $H$ contains elements of finite order, and thus the extension defining $G$ must be non-split.}
  \[
    \mathrm{cd}_\ell G\leq 1=\mathrm{tr.c}(\widehat{\overline{\Fpbar((s^{1/p^\infty}))}}/\Fpbar).
    \]
    Thus by \cite[Proposition 20.17]{scholze_etale_cohomology_of_diamonds} we can conclude that an object $B\in D_\et(\mathcal{G}^{\vee,\circ},\Lambda)$ is compact if and only if its stalk is perfect. In particular, the cohomology of $\mathcal{G}^{\vee,\circ}$, i.e.,
    \[
      R\Gamma(\mathcal{G}^{\vee,\circ},-),
    \]
    commutes with direct sums.
    Next we claim that the functor $j_{\ast}^\vee$ (recall that we follow the derived convention, i.e., we are considering $R j_\ast^\vee$ here) commutes with direct sums.
    For this, it suffices to see that the functor
    \[
      i^{\vee\ast}\circ j_{\ast}^\vee \cong D_\et(\mathcal{G}^{\vee,\circ}, \Lambda)\to D(\Lambda)
    \]
    commutes with direct sums. We claim that
    \[
      i^{\vee \ast}\circ j_{\ast}^\vee \cong R\Gamma(\mathcal{G}^{\vee,\circ},-),
    \]
    which implies the claim. For this it suffices to see that the functor
    \[
      R\Gamma(\mathcal{G}^\vee,j_{!}^\vee(-))
    \]
    is identically $0$. This is implied by \cite[Theorem IV.5.3]{fargues2021geometrization}, cf.\ the proof of \cite[Proposition 4.2]{fargues2021geometrization}.

    Having proved that $j_{\ast}^\vee$ commutes with direct sums, we can conclude that $j^{\vee \ast}$ preserves compact objects. The same is true for $i^{\vee\ast}$. Thus, a compact object $B\in D_\et(\mathcal{G}^\vee,\Lambda)$ has perfect stalks.
    For the converse, we already saw that the compact objects in $D_\et(\mathcal{G}^{\vee,\circ},\Lambda)$ are precisely the ones with perfect stalks.
    It is formal that $j_{!}^\vee$ preserves compact objects. Let $K\in D(\Lambda)$ be a perfect complex. We denote the pullback of $K$ to $\mathcal{G}^\vee,\ldots$ again by $K$. Then we have an exact triangle
    \[
      j_{!}^\vee K\to K\to i_{\ast}^\vee K.
    \]
    Thus, $i_{\ast}^\vee K$ is compact if and only if $K$ is compact. By perfectness, we can reduce to the case that $K=\Lambda$. Then we have to see that
    \[
      R\Gamma(\mathcal{G}^\vee,-)
    \]
    commutes with direct sums. This is implied by the statement that
    \[
      R\Gamma(\mathcal{G}^\vee,-)\cong i^{\vee\ast}(-),
    \]
    which in turn is implied by the statement
    \[
      R\Gamma(\mathcal{G}^\vee,j_{!}^\vee(-))=0,
    \]
    which we already used above.
    Now, if $B\in D_{\et}(\mathcal{G}^\vee,\Lambda)$ has perfect stalks, then we have an exact triangle
    \[
      j_{!}^\vee j^{\vee \ast} B\to B\to i_{\ast}^\vee i^{\vee \ast}B
    \]
    and we proved that the outer terms are compact. This finishes the proof.
  \end{proof}

\begin{corollary}
   \label{sec:case-from-bcmathc-2-ft-preserves-finiteness-of-stalks}
   The functors $\mathcal{F}_\psi$, $\mathcal{F}_\psi^\vee$ preserve objects with perfect stalks.
\end{corollary}
\begin{proof}
  This follows from \Cref{fourier-transform-preserves-compact-objects}, \Cref{sec:case-from-bcmathc-4-compact-for-v} and \Cref{sec:case-from-bcmathc-4-compact-for-w}.\end{proof}

\begin{remark}
The conclusion is non-obvious as the stalks of $\mathcal{F}_\psi$ over geometric points $\Spd(C)\to \mathcal{G}^\vee$ are computed as the cohomology of the space
\[
  \mathcal{G}_C \cong \mathbb{D}_C^\diamond,
\]
which is the non-quasi-compact open unit disc over $C$!
\end{remark}

Our next goal is to say something about the rank of these stalks. To do so, we need to give a short reminder on the theory of Swan conductors on adic curves, as developed in \cite{huber2001swan}. From now on in this subsection, $\Lambda$ will be assumed to be a local ring which is a filtered union of finite rings of order prime to $p$\footnote{This assumption comes from our use of results from \cite{ramero2005local} and ensures in particular that $\Lambda$-local systems on a qcqs adic space are trivialized by a finite \'etale cover.}.

Let $(K,|.|)$ be a valued field enjoying the following properties:
\begin{enumerate}
\item $|.|:K \twoheadrightarrow \Gamma \cup \{0\}$ is henselian and defectless in any finite separable extension $L$ of $K$ (meaning that $[L:K]$ is the product of the index of the value group $\Gamma$ of $|.|$ in the value group of the unique extension of $|.|$ to $L$ by the degree of the residue field extension).
\item The set $\{\gamma \in \Gamma,\gamma <1\}$ has a greatest element $\gamma_K$, such that $\Gamma \cong \gamma_K^\Z \times \Gamma_{\rm div}$, where $\Gamma_{\rm div}$ is the subgroup of divisible elements in $\Gamma$.
\end{enumerate}

If $\gamma \in \Gamma$, it can by assumption be written uniquely as
$$
\gamma=\gamma' . \gamma_K^n,
$$
with $n\in \Z$ and $\gamma' \in \Gamma_{\rm div}$. We will write
$$
 \gamma^\sharp :=n, ~  \gamma^\flat:=\gamma^\prime. 
$$
We denote the extension of the multiplicative function $(-)^\sharp\colon \Gamma\to \Z$ to $\Gamma\otimes_{\Z}\Q$ by the same letter.
Let $V$ be a $\Lambda$-linear representation of $\mathrm{Gal}_K$, factoring through $\mathrm{Gal}(L/K)$ for a certain finite Galois extension $L$ of $K$.
Let $P_K$ be the wild inertia subgroup of $\mathrm{Gal}_K$, and let $\mathrm{Gal}_K^{\gamma}, \mathrm{Gal}_K^{\gamma -}\subseteq \mathrm{Gal}_K, \gamma\in \Gamma\otimes_\Z \Q,$ be the ramification groups defined in \cite[Chapter 2]{huber2001swan}. Define $V(1)=V^{P_K}$ and for any $\gamma \in \Gamma \otimes \Q$, $\gamma< 1$, set
$$
V(\gamma)= (\mathrm{span} \{v - \tau.v, v \in V, \tau \in \mathrm{Gal}_K^\gamma\})^{\mathrm{Gal}_K^{\gamma -}}.
$$

We call $\gamma \in \Gamma \otimes \Q$, $\gamma \leq 1$, a \textit{slope} of $V$ if $V(\gamma)\neq 0$. We set
$$
\mathrm{sw}(V)= \sum_{\gamma \in \Gamma \otimes \Q, \gamma \leq 1}\gamma^{\sharp} . \ell_{\Lambda} V(\gamma),
$$
where $\ell_\Lambda M$ denotes the length of a finitely generated $\Lambda$-module $M$. This quantity is called the \textit{Swan conductor} of $V$.
\\

Any complete, discretely valued field $K$ is an example (with $\Gamma_{\mathrm{div}}=\{1\})$ of a valued field satisfying the assumptions (1)-(2) stated above, but for example $\mathbb{C}_p$ is not. For a discretely valued field, the slopes of a $\Lambda$-representation $V$ of $\mathrm{Gal}_K$ in the sense given above are the elements of $\Gamma \otimes \Q$ of the form $q^{-a}$, where $q=|\varpi|_K^{-1}$, where $\varpi$ is a uniformizer of $K$, and $a$ is a slope of $V$ in the usual sense given to this term (see e.g \cite[\S 2.1]{laumon_transformation_de_fourier}). In particular, the Swan conductor $\mathrm{sw}(V)$ just defined agrees with the classical definition of the Swan conductor of $V$.

Another important class of examples of valued fields satisfying the assumptions (1)-(2) is provided by the henselizations of the residue fields at rank $2$ points of an analytic adic curve over a complete algebraically closed non-archimedean field $C$ of residue characteristic $p$, cf. \cite[\S 5]{huber2001swan}. Following Huber, we will adopt the following notations in the sequel:
\begin{itemize}
\item If $a$ is a classical point of $\mathbb{A}_C^1$ and $r\in \Gamma_C$, there is a unique point $p_{a,r}^-$ in the closed subset $\{x, |T-a|_x <r \}$ of $\mathbb{A}_C^1$ which is not in its interior. It is a rank $2$ point. The henselization of the residue field at $p_{a,r}^-$ is a valued field $K$ as above. We will write $\Gamma_K=\Gamma_{p_{0,r}^-}$, $\gamma_K=\gamma_{p_{0,r}^-}$ and we will denote by
$$
\alpha_{a,r}(-)
$$
the Swan conductor function on representations of its Galois group.
\item If $a$ is a classical point of $\mathbb{A}_C^1$ and $r\in \Gamma_C$, there is a unique point $p_{a,r}^+$ in the closure of the open subset $\{x, |T-a|_x \leq r \}$ of $\mathbb{A}_C^1$ which is not in this open subspace. It is a rank $2$ point. The henselization of the residue field at $p_{a,r}^+$ is valued field $K$ as above. We will write $\Gamma_K=\Gamma_{p_{0,r}^+}$, $\gamma_K=\gamma_{p_{0,r}^+}$ and we will denote by
$$
\beta_{a,r}(-)
$$
the Swan conductor function on representations of its Galois group.
\end{itemize}
  
We now review the little bit of the theory of \cite{ramero2005local}, which gives some useful information about the variation with $r$ of the functions $\alpha_{a,r}, \beta_{a,r}$ just introduced. As above fix an algebraically closed, non-archimedean field $C$ whose residue field has characteristic $p$. For $a\in \R_{\geq 0}\cap |C|$ let us fix an element
\[
  \pi_a\in K
\]
such that $|\pi_a|=a$.
For $a,b\in \R_{\geq 0}\cap |C|$ with $a\leq b$ let
\[
  \mathbb{D}_{[a,b]}:=\{ x \ |\ |\pi_a|\leq |x|\leq |\pi_b|\}\subseteq \mathbb{A}^{1,\mathrm{ad}}_{C}
\]
be the associated (topologically open) annulus. Clearly, if $c\in \R_{>0}\cap |C^\times|$, then
\[
  \mathbb{D}_{[a,b]}\cong \mathbb{D}_{[ca,cb]}
\]
via multiplication by $\pi_c$.
If $a>0$, then
\[
  \iota\colon \mathbb{D}_{[a,b]}\cong \mathbb{D}_{[1/b,1/a]}
\]
via inversion.

Let $\mathcal{F}$ be a $\Lambda$-local system on $\mathbb{D}_{[a,b]}$. Then we obtain the functions
\[
  \mathrm{sw}_{\mathcal{F},<}\colon (a,b)\cap |C^\times|\to \Z,\ r\mapsto \alpha_r(\mathcal{F}):=\alpha_{0,r}(\mathcal{F})
\]
(denoted by $\mathrm{sw}^{\natural}(\mathcal{F}, r^+)$ in \cite[4.1.10]{ramero2005local})
and
\[
  \mathrm{sw}_{\mathcal{F},>}\colon (a,b)\cap |C^\times| \to \Z,\ r\mapsto \beta_r(\mathcal{F}):=\beta_{0,r}(\mathcal{F}),
\].
Let
\[
  \delta_{\mathcal{F}}\colon (-\mathrm{log}(b),-\mathrm{log}(a))\cap (-\mathrm{log}(|C^\times|))\to \R_{\geq 0}
\]
be the discriminant function of \cite[4.1.13]{ramero2005local}.
Writing $\delta_{\mathcal{F}}$ as a function of the radius $r$ (and not $-\mathrm{log}(r)$) would be pleasant, but unfortunately it is the function $\delta_{\mathcal{F}}$ having the easier properties as captured by \Cref{sec:slopes-local-systems-1-linearity-convexity-derivative-for-discrimiant-function}.

  We denote by $\partial_r, \partial_l$ the right resp.\ left derivative of a piecewise linear function $(a,b)\cap |C^\times|\to \R$ (note that these derivatives make sense even if we don't require that the piecewise linear function is defined on all of $(a,b)$).

  \begin{theorem}[Ramero]
    \label{sec:slopes-local-systems-1-linearity-convexity-derivative-for-discrimiant-function}
    The map
    \[
      \delta_{\mathcal{F}}\colon (-\mathrm{log}(b),-\mathrm{log}(a))\cap (-\mathrm{log}(|C^\times|))\to \R_{\geq 0}
    \]
    extends uniquely to a piecewise linear, continuous, convex function $(a,b)\to \R_{\geq 0}$ with
    \[
      \partial_r \delta_{\mathcal{F}}(-\mathrm{log}(r))=\mathrm{sw}_{\mathcal{F},<}(r)
    \]
    for any $r\in (a,b)\cap |C^\times|$. Moreover, the function $\mathrm{r}\mapsto \mathrm{sw}_{\mathcal{F},<}(r)$ is non-increasing, i.e., $\mathrm{sw}_{\mathcal{F},<}(r_2)\leq \mathrm{sw}_{\mathcal{F},<}(r_1)$ if $r_1\leq r_2$.
  \end{theorem}
  \begin{proof}
    The first claims are \cite[4.1.15]{ramero2005local} (note that \cite[Proposition 3.3.26]{ramero2005local} does not use the assumption of \cite[4]{ramero2005local} that $K$ is of characteristic $0$). By convexity of $\delta_{\mathcal{F}}$ we can conclude that $\partial_r\delta_{\mathcal{F}}$ is non-decreasing. Composing with the decreasing function $r\mapsto -\mathrm{log} r$, yields that $\mathrm{sw}_{\mathcal{F},<}(r)$ is non-increasing.
  \end{proof}

  As indicated in \cite[3.3.11]{ramero2005local} we can also relate $\delta_{\mathcal{F}}$ and $\mathrm{sw}_{\mathcal{F},>}$. Let us spell out the details.

  \begin{corollary}
    \label{sec:slopes-local-systems-2-discriminant-and-swan-conductor-for-outer-point}
    Assume $a>0$. Then
    \[
      \partial_l\delta_{\mathcal{F}}(-\mathrm{log}(r))=-\mathrm{sw}_{\mathcal{F},>}(r)
    \]
    for $r\in (a,b)\cap |C^\times|$. In particular, $\mathrm{sw}_{\mathcal{F},>}(r)$ is non-decreasing.
  \end{corollary}
  \begin{proof}
    Let
    \[
      \iota\colon \mathbb{D}_{[a,b]}\cong \mathbb{D}_{[1/b,1/a]}
    \]
    be the inversion and set $\mathcal{G}:=\iota_\ast \mathcal{F}$.
    If $r\in (a,b)$ set $\iota(r):=1/r\in (1/b,1/a)$.
    Then
        \[
      \beta_r(\mathcal{F})=\alpha_{\iota(r)}(\mathcal{G})
    \]
    as
    \[
     \mathcal{F}_{p^+_r}\cong (\iota^\ast \mathcal{G})_{p^{+}_r}\cong \mathcal{G}_{\iota(p^+_r)}=\mathcal{G}_{p^-_{\iota(r)}}.
   \]
   In particular,
   \[
     \mathrm{sw}_{\mathcal{F},>}=\mathrm{sw}_{\mathcal{G},<}\circ \iota.
   \]
   Moreover,
   \[
     \delta_{\mathcal{F}}(s)=\delta_{\mathcal{G}}(-s)
   \]
   as $\iota(p^\flat_r)=p^\flat_{\iota(r)}$ for $r\in (a,b)\cap |C^\times|$ and
   \[
     -\mathrm{log}(\iota(r))=-(-\mathrm{log}(r)).
   \]
   Now we calculate for $s\in (-\mathrm{log}(b),-\mathrm{log}(a))\cap (-\mathrm{log}(|C^\times|))$ and $x<0$ sufficiently small
   \[
     \begin{matrix}
       & \delta_{\mathcal{F}}(s+x) \\
       = & \delta_{\mathcal{G}}(-s-x) \\
       \overset{~\Cref{sec:slopes-local-systems-1-linearity-convexity-derivative-for-discrimiant-function}, -x\geq 0}{=} & \delta_{\mathcal{G}}(-s)+\mathrm{sw}_{\mathcal{G},<}(e^{s})(-x) \\
       \overset{e^{s}=\iota(e^{-s})}{=} & \delta_{\mathcal{F}}(s)+(-\mathrm{sw}_{\mathcal{F},>}(e^{-s}))x,\\
     \end{matrix}
   \]
   which implies the first claim. For the last we can argue as in \Cref{sec:slopes-local-systems-1-linearity-convexity-derivative-for-discrimiant-function}.  
 \end{proof}
 
 \begin{remark}
 If $-\mathrm{log}(r)$ is not a breakpoint of $\delta_{\mathcal{F}}$, then
 \[
   \mathrm{sw}_{\mathcal{F},<}(r)=-\mathrm{sw}_{\mathcal{F},>}(r),
 \]
 but otherwise both sides are different. Because of this discrepancy we think that \cite[Proposition 2.9]{wewers_swan_conductors_on_the_boundary_of_lubin_tate_spaces} is (slightly) incorrect as stated because we think that the $\mathrm{sw}_{\mathcal{F}}(s)$ from there agrees with $\mathrm{sw}_{\mathcal{F},>}(e^{-s})$ in our notation.
\end{remark}

This formalism will be applied in the following geometric situation. Let $A$ be a $\Lambda$-local system on the punctured open unit $\D_C^\ast$. 
Assume that the compactly supported cohomology
\[
  H^\ast_c(\D_C^\ast,A)
\]
of $A$ is finite. Let
\[
  \chi_c(\D^\ast_C,A):=\sum\limits_{i=0}^\infty (-1)^i \ell_{\Lambda} H^i_c(\D_C^\ast,A)
\]
denote its Euler characteristic. 

\begin{proposition}
  \label{sec:groth-ogg-shaf-1-euler-characteristic}
  With the above notations and assumption, we have
  \[
    \chi_c(\D^\ast_C,A)=-\alpha(A)-\beta(A).
  \]
  Here, cf.\ \cite[Corollary 10.6]{huber2001swan},
  \[
    \alpha(A):= \alpha_{0,r}(A)
  \]
  for $r$ sufficiently small, and
  \[
    \beta(A):= \beta_{0,r} (A)
  \]
  for $r$ sufficiently close to $1$.
\end{proposition}

\begin{proof}
  We first claim that there exists some $s\in [0,1]\cap |C^\times|$ such that
  \[
    H^\ast_c(\D^\ast_C,A)=H^\ast_c(\B^\ast_{C,r},A)
  \]
  for all $r\in [0,1]\cap |C^\times|, r\geq s$. Indeed,
  for $r^\prime\geq r$ the canonical morphisms
  \[
    H^0_c(\B_{C,r}^\ast,A)\to H^0_c(\B_{C,r^\prime}^\ast, A)
  \]
  and
  \[
    H^0(\B_{C,r^\prime}^\ast,A)\to H^0(\B_{C,r^\prime}^\ast, A)
  \]
  are injective morphisms between finite $\Lambda$-modules. Thus, they are isomorphisms for sufficiently large $r,r^\prime$. From Poincar\'e duality one deduces that
  \[
    H^2_c(\B_{C,r}^\ast,A)\to H^2_c(\B_{C,r^\prime}^\ast, A)
  \]
  is an isomorphism for large $r,r^\prime$, too.
  Moreover,
  \[
    H^1_c(\B_{C,r}^\ast,A)\to H^1_c(\B_{C,r^\prime}^\ast,A)
  \]
  is injective for all $r^\prime\geq r$ because $A$ is locally constant, cf.\ \cite[proof of Corollary 10.6]{huber2001swan}.
  This implies the claim by our assumption on the finiteness of
  \[
    H^\ast_c(\D^\ast_C,A).
  \]
  Thus, we can apply \cite[Corollary 10.6]{huber2001swan} to
  \[
    X:=\B_{C,s}
  \]
  and
  \[
    Y:=\B^\ast_{C,s}=X\cap \D_{C}^\ast.
  \]
  Thus,
  \[
    \chi_c(\D_C^\ast,A)=\chi_c(Y,A)=-\alpha_{0}(A)-\alpha_x(j_\ast A),
  \]
  where $\{x\}=X^c\setminus X$
  (note that $\chi_c(Y)=0$). Now in our notation
  \[
    \alpha(A)=\alpha_0(A)
  \]
  and
  \[
    \beta(A)=\alpha_x(j_\ast(A)),
  \]
  cf.\ \cite[Lemma 8.6.(v)]{huber2001swan}. This finishes the proof.
\end{proof}

We now want to apply these considerations to the Fourier transform for $\mathcal{G}$. Let $\sigma$ be a $\Lambda$-representation of $W_{E}$, corresponding to a local system $\mathbb{L}$ on $\mathrm{Div}^1$. Let $\widetilde{\mathbb{L}}$ be the local system on $\mathcal{G}^\circ$ obtained by pulling back $\mathbb{L}$ along the $E^\times$-torsor $\mathcal{G}^\circ \to \mathrm{Div}^1$. Fix a geometric point $\mathrm{Spd}(C) \to \mathcal{G}^{\vee \circ}$, and an identification 
$$
(\mathcal{G}^{\vee \circ} \times \mathcal{G}^\circ) \times_{\mathcal{G}^{\vee \circ}} \mathrm{Spd}(C) \cong \mathcal{G}_C^\circ \cong \mathbb{D}_C^{\ast \diamond}.
$$
Pulling back $\widetilde{\mathbb{L}}$ to $\mathcal{G}_C^\circ$, we obtain  via this identification a local system on $\mathbb{D}_C^\ast$, which we will simply denote by $\widetilde{\mathbb{L}}_C$. 

Similarly, the sheaf $\alpha^* \mathcal{L}_\psi$ on $\mathcal{G}^{\vee} \times \mathcal{G}$ (the kernel of the Fourier transform) defines by pullback and via the above identification a local system on $\mathbb{D}_C^\ast$, which we will simply denote by $\mathcal{L}_{\psi,C}$.

\begin{corollary}
 \label{sec:groth-ogg-shaf-corollary-to-1-euler-characteristic}
   We use the above notations. Assume that $j_!\widetilde{\mathbb{L}^\vee}\cong Rj_\ast \widetilde{\mathbb{L}^\vee}$. The stalk of 
  $$
  (j^{\vee \circ})^\ast \mathcal{F}_\psi (j_{!}\widetilde{\mathbb{L}})
  $$
  at the geometric point $\mathrm{Spd}(C) \to \mathcal{G}^{\vee\circ}$ is concentrated in degree $0$, and its rank equals
  \[
   \alpha(\widetilde{\mathbb{L}}_C \otimes \mathcal{L}_{\psi,C})+\beta(\widetilde{\mathbb{L}}_C \otimes \mathcal{L}_{\psi,C}).
  \]
\end{corollary}
The assumption on $\widetilde{\mathbb{L}}^\vee$ is satisfied if $\widetilde{\mathbb{L}}^\vee$ does not contain the trivial local system. Indeed, $i^\ast Rj_\ast(\widetilde{\mathbb{L}})\cong R\Gamma(\mathcal{G}^\circ,\widetilde{\mathbb{L}})$ by the proof of \Cref{sec:case-from-bcmathc-4-compact-for-v} and this complex is concentrated in degree $0$ and of Euler characteristic $0$.
 \begin{proof}
 First we prove concentration in degree $0$. Since the Fourier transform commutes with Verdier duality (\Cref{sec:stacks-bc-type-properties-of-very-nice-stacks-plus-fourier-transform}) and $j_!\widetilde{\mathbb{L}}^{\vee}\cong Rj_\ast \widetilde{\mathbb{L}}^\vee$, it suffices to prove that the cohomology of $\widetilde{\mathbb{L}}_C \otimes \mathcal{L}_{\psi,C}$ on $\mathcal{G}^\circ$ is concentrated in degrees $[0,1]$; this follows as $H^i_c(\mathbb{D}^\ast_C,\mathbb{M})=0$ for each $i\notin [0,1]$ and $\mathbb{M}$ a local system on $\mathbb{D}^\ast_C$. 

 The formula for the rank is a direct consequence of \Cref{sec:groth-ogg-shaf-1-euler-characteristic}, which one can apply thanks to \Cref{sec:case-from-bcmathc-2-ft-preserves-finiteness-of-stalks} (the change of sign is due to the shift in the definition of the Fourier transform).
 \end{proof}

To apply \Cref{sec:groth-ogg-shaf-corollary-to-1-euler-characteristic} in some cases, we will compute the slopes of the local systems $\mathcal{L}_{\psi,C}$ and $\widetilde{\mathbb{L}}_C$.

We first focus on $\mathcal{L}_{\psi,C}$. 

\begin{proposition}
\label{swan-of-L-psi}
For each $r$ sufficiently close to $0$, the (unique) slope of $\mathcal{L}_{\psi,C}$ at $p_{0,r}^-$ and $p^{+}_{0,r}$ is $1$. For each $r$ sufficiently close to $1$, the slope of $\mathcal{L}_{\psi,C}$ at $p_{0,r}^+$ is $\mathrm{exp}(-c)r^{-d} \gamma_{p_{0,r}^+}$, for some constants $c,d\in \mathbb{R}_{>0}$. In particular, in the notations of \Cref{sec:groth-ogg-shaf-1-euler-characteristic}, we have
$$
 \alpha(\mathcal{L}_{\psi,C})=0, ~ \beta(\mathcal{L}_{\psi,C})=1.
$$
\end{proposition}
\begin{proof}
  If $\mathcal{F}$ is any $\Lambda$-local system on $\mathbb{D}_C$, then for $r>0$ sufficiently small the restriction of $\mathcal{F}$ to the disc $\mathbb{B}_{C,r}$ will be trivial because the local ring $\mathcal{O}_{\mathbb{D}_C,0}$ of $\mathbb{D}_C$ is strictly henselian. In particular, the slopes of $\mathcal{L}_{\psi,C}$ at $p_{0,r}^-$ and $p^+_{0,r}$ will be $1$.
  
 From the proof of \Cref{description-push-forward-l-psi}, we know that
    \[
      H^\ast_c(\mathbb{D}_C,\mathcal{L}_{\psi,C})=0.
    \]
    This implies that
    \[
      H^\ast_c(\mathbb{B}_{C,r},\mathcal{L}_{\psi,C})=0
    \]
    for each sufficiently large $0<r<1$. By \cite[Corollary 10.4]{huber2001swan} this implies (because $\chi_c(\mathbb{B}_{C,r},\Lambda)=1$) that
    \[
      \beta_{0,r}(\mathcal{L}_{C,\psi})=1
    \]
    for sufficiently large $0<r<1$.
    
    We can by \Cref{sec:slopes-local-systems-1-linearity-convexity-derivative-for-discrimiant-function} and \Cref{sec:slopes-local-systems-2-discriminant-and-swan-conductor-for-outer-point} conclude that the discriminant function
    \[
      \delta_{\mathcal{L}_{\psi,C}}\colon (0,\infty)\cap (-\mathrm{log}(|C^\times|))\to \mathbb{R}_{\geq 0}
    \]
    has the slopes $-1, 0$ and therefore a unique break point $c^\prime\in (0,\infty)$. As $\delta_{\mathcal{L}_{\psi,C}}(s)=0$ for $s>c^\prime$ we can conclude that
    \[
      \delta_{\mathcal{L}_{\psi,C}}(s)=-s+c^\prime
    \]
    for $0<s\leq c^\prime$.
    By \cite[3.3.2]{ramero2005local} we can conclude that if $\gamma=\gamma^\flat\cdot \gamma^\sharp$ is the (unique) slope of $\mathcal{L}_{C,\psi}$, then
    \[
      c^\prime+\mathrm{log}(r)=\delta_{\mathcal{L}_{\psi,C}}(-\mathrm{log}(r))=-d\cdot \mathrm{log}(\gamma^\flat)
    \]
    for some constant $d\in \Z_{>0}$. This implies
    \[
      \gamma^\flat=\mathrm{exp}(-c^\prime/d)r^{-d}
    \]
    if $0<-\log(r)\leq c^\prime$. As $\gamma^\sharp$ is determined by the Swan conductor, we get that
    \[
      \gamma=\mathrm{exp}(-c^\prime/d)r^{-d}\gamma_{p^+_{0,r}}
    \]
    as desired by taking $c=c^\prime/d$.
\end{proof}

Next, we turn to the computation of the slopes of $\widetilde{\mathbb{L}}_C$. Let $F\cong \mathbb{F}_q((t))$ be the field of norms (in the sense of Fontaine-Wintenberger, \cite{wintenberger_corps_des_normes}) attached to the Lubin-Tate extension $E_\infty$ of $E$. We choose the identification 
$$
\mathcal{G}^\circ \cong \mathbb{D}_{\overline{\mathbb{F}_p}}^{\ast \diamond}
$$
induced by the isomorphism between the completion of the perfection of $F$ and $E_\infty$. In this way, $\widetilde{\mathbb{L}}$ corresponds to a representation $V$ of the Weil group of $F$.

If $\sigma$ is of dimension $n=1$, $V$ is trivial, by local class field theory, and thus its Swan conductor is zero. In general, we have the following result.

\begin{proposition}
\label{conjecture-for-the-swan-conductor}
We keep the above notations, and assume moreover that $\Lambda$ is an algebraically closed field, that $\sigma$ is irreducible of dimension $n$ and that its Swan conductor (with respect to $E$) cannot be lowered by twisting by a character. The slope of $V$ (with respect to $F$) is then
$$
\mathrm{sl}(V)=q^{\lfloor \mathrm{sl}(\sigma) \rfloor }(q-1)(\mathrm{sl}(\sigma)-\lfloor \mathrm{sl}(\sigma) \rfloor ) +q^{\lfloor \mathrm{sl}(\sigma)\rfloor}-1
$$
and its Swan conductor (with respect to $F$) is
$$
\mathrm{sw}(V)=n(q^{\lfloor \mathrm{sl}(\sigma) \rfloor }(q-1)(\mathrm{sl}(\sigma)-\lfloor \mathrm{sl}(\sigma) \rfloor ) +q^{\lfloor \mathrm{sl}(\sigma)\rfloor }-1).
$$
\end{proposition}
Here $\lfloor a\rfloor$ denotes the lower Gau\ss\ bracket.
\begin{proof}
  As $\sigma$ is irreducible we see that $V$ has a unique slope.
      We want to compute
$$
\mathrm{sl}(V)= \mathrm{inf} \{u \in \mathbb{R}_+, \mathrm{Gal}_F^{u} \subset \ker(V) \}.
$$
First, we claim that 
$$
\mathrm{inf} \{u \in \mathbb{R}_+, \mathrm{Gal}_E^{u} \cap \mathrm{Gal}_{E_\infty} \subset \ker(\sigma_{|_{\mathrm{Gal}_{E_\infty}}}) \} =\mathrm{sl}(\sigma).
$$
It is clear that the left hand side is smaller than the right hand side. If it were strictly smaller, since $E_\infty$ is the maximal totally ramified abelian extension of $E$, it would mean we could twist $\sigma$ by a character to lower its slope, contradicting our hypothesis. 

We know, cf. \cite[Corollaire 3.3.6]{wintenberger_corps_des_normes}, that for each $u$,
$$  
\mathrm{Gal}_E^{u} \cap \mathrm{Gal}_{E_\infty} = \mathrm{Gal}_F^{\psi_{E_\infty/E}(u)},
$$
where $\psi_{E_\infty/E}$ denotes the inverse Herbrand function of the extension $E_\infty/E$ (i.e. the limit when $m$ goes to infinity of the inverse Herbrand functions $\psi_{E_m/E}$ of the extensions $E_m/E$, whose values stabilize as will be shown by the explicit formulas below).
  
Therefore, we conclude that
$$
\mathrm{sl}(V) = \psi_{E_\infty/E}(\mathrm{sl}(\sigma)).
$$
Let $m\geq 1$. The inverse Herbrand function $\psi_{E_m/E}$ can be computed as follows: it sends $x \leq 0$ to $x$, $x \in [0,m-1]$ to 
$$
q^{\lfloor x \rfloor }(q-1)(x-\lfloor x \rfloor) +q^{\lfloor x\rfloor}-1
$$
and $x\geq m-1$ to
$$
q^{m-1}(q-1)(x-(m-1))+q^{m-1}-1.
$$
In particular, its value at $x\leq m-1$ is independent of $m$, and is the value $\psi_{E_\infty/E}(x)$. Hence, we obtain
$$
\mathrm{sl}(V)= q^{\lfloor \mathrm{sl}(\sigma) \rfloor}(q-1)(\mathrm{sl}(\sigma)-\lfloor \mathrm{sl}(\sigma) \rfloor ) +q^{\lfloor \mathrm{sl}(\sigma)\rfloor }-1.
 $$
The final assertion of the proposition is \cite[Proposition 3.4]{kilic_inequalities_on_swan_conductors} and the fact that all irreducible constituents of $V$ have the same slope. 
\end{proof}

\begin{lemma}
\label{slopes-tilde_E}
Let $0<r<1$, $r\in \Gamma_C$. If $V$ has only one slope, then the slope of $\widetilde{\mathbb{L}}_C$ at $p_{0,r}^-$, resp. at $p_{0,r}^+$, is $r^{\mathrm{sl}(V)} \gamma_{p_{0,r}^-}^{\mathrm{sl}(V)}$, resp. $r^{\mathrm{sl}(V)} \gamma_{p_{0,r}^+}^{-\mathrm{sl}(V)}$. In particular, using the notations of \Cref{sec:groth-ogg-shaf-1-euler-characteristic}, we have
$$
\alpha(\widetilde{\mathbb{L}}_C)= \mathrm{sw}(V), ~ \beta(\widetilde{\mathbb{L}}_C) = -\mathrm{sw}(V).
$$
\end{lemma}
\begin{proof}
    Under the morphism $\Fpbar((t))\to k(p^{\pm}_{0,r})$ the element $t\in \Fpbar((t))$ is mapped to an element of valuation $r\gamma_{p_{0,r}^{\pm}}^{\mp}$. Unraveling the definitions this implies the claim.
\end{proof}

  In particular, we see that the discrimant function for $\widetilde{\mathbb{L}}_C$ recalled above is the line
  \[
    (0,\infty)\to \R,\ s\mapsto \mathrm{sw}(V)\cdot s.
  \]

\begin{corollary}
\label{computation-of-the-rank-fourier-transform-explicit-formula}
Let $\Lambda$ be an algebraically closed field of characteristic $\ell$. Let $\sigma$ be an irreducible $\Lambda$-representation of $W_{E}$ of dimension $n>1$, with corresponding local system $\mathbb{L}$ on $\mathrm{Div}^1$, and assume that the Swan conductor of $\sigma$ cannot be lowered by twisting by a character. Then the stalk of 
  $$
  (j^{\vee \circ})^\ast \mathcal{F}_\psi (j_!\widetilde{\mathbb{L}})
  $$
  at any geometric point $\mathrm{Spd}(C) \to \mathcal{G}^{\vee\circ}$ is concentrated in degree $0$ and has rank
  $$
  n+ n(q^{\lfloor \mathrm{sl}(\sigma) \rfloor }(q-1)(\mathrm{sl}(\sigma)-\lfloor \mathrm{sl}(\sigma) \rfloor) +q^{\lfloor \mathrm{sl}(\sigma)\rfloor }-1).
  $$
 \end{corollary}  
\begin{proof}
We already proved concentration in degree $0$ in \Cref{sec:groth-ogg-shaf-corollary-to-1-euler-characteristic}.

When $r$ is close to $0$, $\mathcal{L}_{\psi,C}$ becomes trivial and thus
$$
\alpha(\mathcal{L}_{\psi,C} \otimes \widetilde{\mathbb{L}}_C) = \mathrm{rk}( \mathcal{L}_{\psi,C}) \cdot \alpha(\widetilde{\mathbb{L}}_C) = \mathrm{sw}(V)
$$
by \Cref{slopes-tilde_E}.
When $r$ is close to $1$, the slope of $\mathcal{L}_{\psi,C}$ at $p_{0,r}^+$ is $\mathrm{exp}(-c) r^{-d} \gamma_{p_{0,r}^+}$ for some constants $c,d\in \mathbb{R}_{>0}$ (by \Cref{swan-of-L-psi}) whereas the slope of $\widetilde{\mathbb{L}}_C$ at $p_{0,r}^+$ is $r^{\mathrm{sl}(V)}. \gamma_{p_{0,r}^+}^{-\mathrm{sl}(V)}$ (\Cref{slopes-tilde_E}). Since
$$
\mathrm{exp}(-c) r^{-d} \underset{r \to 1} \longrightarrow \mathrm{exp}(-c) <1, \quad r^{\mathrm{sl}(V)}  \underset{r \to 1} \longrightarrow 1,
$$
the slope of $\mathcal{L}_{\psi,C}$ is stricly smaller for $r$ close to $1$, and thus
$$
\beta(\mathcal{L}_{\psi,C} \otimes \widetilde{\mathbb{L}}_C) = \mathrm{rk}(\widetilde{\mathbb{L}}_C) \cdot \beta(\mathcal{L}_{\psi,C}) = n. 
$$
This gives the desired result, by applying \Cref{sec:groth-ogg-shaf-corollary-to-1-euler-characteristic} and \Cref{conjecture-for-the-swan-conductor}.
\end{proof}

\begin{remark}
\label{remark-formal-degree-division-algebra}
When $n=2$, the above formula for the rank takes the following simple form. First of all, $\mathrm{sw}(\sigma)=2\mathrm{sl}(\sigma)$. If $\mathrm{sw}(\sigma)$ is even, $\lfloor \mathrm{sl}(\sigma) \rfloor=\mathrm{sw}(\sigma)/2$, and thus \Cref{conjecture-for-the-swan-conductor} gives that the rank is
$$
2q^{\mathrm{sw}(\sigma)/2}.
$$
If $\mathrm{sw}(\sigma)$ is odd, $\lfloor \mathrm{sl}(\sigma) \rfloor =(\mathrm{sw}(\sigma)-1)/2$, and \Cref{conjecture-for-the-swan-conductor} gives that the rank is
$$
(q+1)q^{\mathrm{sw}(\sigma)/2-1/2}.
$$
Let $D^\times$ be the group of $E$-points of the group of units of the unique non-split quaternion algebra over $E$. Let $\rho$ be the smooth irreducible (hence finite-dimensional) representation of $D^\times$ attached to $\sigma$ by composing the local Langlands correspondence for $\GL_2$ over $E$ with the local Jacquet-Langlands correspondence. Carayol, \cite[Proposition 6.5]{carayol_representations_cuspidales_du_groupe_lineaire}, has computed the dimension of $\rho$ in terms of $\sigma$. If $\mathrm{sw}(\sigma)$ is even, 
$$
\dim(\rho) =   2q^{\mathrm{sw}(\sigma)/2}.
$$
If $\mathrm{sw}(\sigma)$ is odd, 
$$
\dim(\rho) =(q+1)q^{\mathrm{sw}(\sigma)/2-1/2}.
$$
We therefore recover the exact same formulas as above. This should not be a coincidence. Indeed, using the notations introduced at the very end of \Cref{sec-an-inductive-construction}, we expect the pullback of
$$ 
\alpha_0^{-1} \circ (\pi_0 \circ j_0^\vee)^\ast (\mathcal{L}_{\mathbb{L}})
$$
to $\mathrm{Div}^1 \subset \mathcal{C}_1^\prime$ to be isomorphic to $\mathbb{L}[1]$. Therefore
$$
(j^{\vee \circ})^\ast \mathcal{F}_\psi (j_!\widetilde{\mathbb{L}}) \in D_\et(\mathcal{G}^{\vee, \circ},\Lambda)
$$
should be isomorphic to the pullback of $\mathcal{A}_{2,\psi}(\mathcal{L}_{\mathbb{L}})[-1]$ along
$$
\mathcal{G}^{\vee, \circ} \to \Bun_2^\prime \times_{\Bun_2} \Bun_2^1 \cong [\mathcal{G}^{\vee, \circ}/E^\times] \to \Bun_2^\prime.
$$
In particular, cf. again \Cref{sec-an-inductive-construction}, $(j^{\vee \circ})^\ast \mathcal{F}_\psi (j_!\widetilde{\mathbb{L}})$ should be the pullback along the map
$$
\mathcal{G}^{\vee, \circ} \cong (\BC(\mathcal{O}(1/2) \backslash \{0\})/\mathrm{SL}_1(D) \to [\mathrm{Spa}(k)/D^\times]
$$
of the sheaf associated to the smooth irreducible representation $\rho$ of $D^\times$.
\end{remark}

\bibliography{biblio}

\begin{thebibliography}{10}

\bibitem{andreychev2021pseudocoherent}
Grigory Andreychev.
\newblock Pseudocoherent and perfect complexes and vector bundles on analytic
  adic spaces.
\newblock {\em arXiv preprint arXiv:2105.12591}, 2021.

\bibitem{bernstein_zelevinsky_representations_of_the_group_glnf}
Joseph Bernstein and Andrei Zelevinskii.
\newblock Representations of the group ${GL}(n,{F})$ where ${F}$ is a
  non-archimedean local field.
\newblock {\em Uspekhi Matematicheskikh Nauk}, 31(3):5--70, 1976.

\bibitem{berthelot_le_theoreme_de_dualite_plate_pour_les_surfaces}
P~Berthelot.
\newblock Le th{\'e}or{\`e}me de dualit{\'e} plate pour les surfaces
  (d'apr{\`e}s {J}. {S}. {M}ilne).
\newblock In {\em Surfaces alg{\'e}briques}, pages 203--237. Springer, 1981.

\bibitem{bhatt2014algebraization}
Bhargav Bhatt.
\newblock Algebraization and {T}annaka duality.
\newblock {\em arXiv preprint arXiv:1404.7483}, 2014.

\bibitem{bhatt_hansen_the_six_functors_for_zariski_constructible_sheaves_in_rigid_geometry}
Bhargav Bhatt and David Hansen.
\newblock The six functors for {Z}ariski-constructible sheaves in rigid
  geometry.
\newblock {\em arXiv preprint arXiv:2101.09759}, 2021.

\bibitem{bhatt_scholze_the_pro_etale_topology_for_schemes}
Bhargav Bhatt and Peter Scholze.
\newblock The pro-{\'e}tale topology for schemes.
\newblock {\em arXiv preprint arXiv:1309.1198}, 2013.

\bibitem{bhatt_scholze_projectivity_of_the_witt_vector_affine_grassmannian}
Bhargav Bhatt and Peter Scholze.
\newblock Projectivity of the {W}itt vector affine {G}rassmannian.
\newblock {\em Invent. Math.}, 209(2):329--423, 2017.

\bibitem{bosch_guentzer_remmert_non_archimedean_analysis}
G{\"u}ntzer-Ulrich Bosch, Siegfried and Reinhold Remmert.
\newblock {\em Non-{A}rchimedean {A}nalysis: a systematic approach to rigid
  analytic geometry}.
\newblock Springer, 1984.

\bibitem{breen_extensions_du_groupe_additif_sur_le_site_parfait}
Lawrence Breen.
\newblock Extensions du groupe additif sur le site parfait.
\newblock In {\em Surfaces Alg{\'e}briques}, pages 238--262. Springer, 1981.

\bibitem{bunke2019twisted}
Ulrich Bunke and Thomas Nikolaus.
\newblock Twisted differential cohomology.
\newblock {\em Algebraic \& Geometric Topology}, 19(4):1631--1710, 2019.

\bibitem{carayol_representations_cuspidales_du_groupe_lineaire}
Henri Carayol.
\newblock Repr{\'e}sentations cuspidales du groupe lin{\'e}aire.
\newblock In {\em Annales scientifiques de l'{\'E}cole Normale Sup{\'e}rieure},
  volume~17, pages 191--225, 1984.

\bibitem{chen_zhu_geometric_langlands_in_prime_characteristic}
Tsao-Hsien Chen and Xinwen Zhu.
\newblock Geometric {L}anglands in prime characteristic.
\newblock {\em Compositio Mathematica}, 153(2):395--452, 2017.

\bibitem{colmez2002espaces}
Pierre Colmez.
\newblock Espaces de {B}anach de dimension finie.
\newblock {\em Journal of the Institute of Mathematics of Jussieu},
  1(3):331--439, 2002.

\bibitem{deligne1973formule}
Pierre Deligne.
\newblock La formule de dualit{\'e} globale. {E}xpos{\'e} {XVIII} of
  {T}h{\'e}orie des topos et cohomologie etale des sch{\'e}mas vol. 3 ({SGA} 4)
  {L}ecture notes in math. 305, 1973.

\bibitem{drinfeld_two_dimensional_l_adic_representations}
V.~G. Drinfeld.
\newblock Two-dimensional $\ell$-adic representations of the fundamental group
  of a curve over a finite field and automorphic forms on ${GL}(2)$.
\newblock {\em American Journal of Mathematics}, 105(1):85--114, 1983.

\bibitem{fargues2021geometrization}
Laurent Fargues and Peter Scholze.
\newblock Geometrization of the local {L}anglands correspondence, 2021.

\bibitem{gabber_ramero_almost_ring_theory}
Ofer Gabber and Lorenzo Ramero.
\newblock {\em Almost ring theory}.
\newblock Springer, 2003.

\bibitem{egaiii2}
Alexander Grothendieck.
\newblock {\'E}l{\'e}ments de g{\'e}om{\'e}trie alg{\'e}brique: {III}.
  {\'e}tude cohomologique des faisceaux coh{\'e}rents, seconde partie.
\newblock {\em Publications Math{\'e}matiques de l'IH{\'E}S}, 17:5--91, 1963.

\bibitem{kaletha_weinstein_on_the_kottwitz_conjecture_for_local_shimura_varieties}
David Hansen, Tasho Kaletha, and Jared Weinstein.
\newblock On the {K}ottwitz conjecture for local shtuka spaces.
\newblock 2021.

\bibitem{huber2001swan}
Roland Huber.
\newblock Swan representations associated with rigid analytic curves.
\newblock {\em Journal f{\"u}r die reine und angewandte Mathematik},
  2001(537):165--234, 2001.

\bibitem{huybrechts2010geometry}
Daniel Huybrechts and Manfred Lehn.
\newblock {\em The geometry of moduli spaces of sheaves}.
\newblock Cambridge University Press, 2010.

\bibitem{kedlaya_liu_relative_p_adic_hodge_theory_foundations}
Kiran~S. {Kedlaya} and Ruochuan {Liu}.
\newblock {Relative $p$-adic Hodge theory: Foundations}.
\newblock {\em ArXiv e-prints}, January 2013.

\bibitem{kilic_inequalities_on_swan_conductors}
Ammar~Yasir Kilic.
\newblock {\em Inequalities on {S}wan conductors}.
\newblock PhD thesis, Universit{\'e} Paris Saclay, 2019.

\bibitem{laumon_duke}
G{\'e}rard Laumon.
\newblock Correspondance de {L}anglands g{\'e}om{\'e}trique pour les corps de
  fonctions.
\newblock {\em Duke Mathematical Journal}, 54(2):309--359, 1987.

\bibitem{laumon_transformation_de_fourier}
G{\'e}rard Laumon.
\newblock Transformation de {F}ourier, constantes d'{\'e}quations
  fonctionnelles et conjecture de {W}eil.
\newblock {\em Publications Math{\'e}matiques de l'IH{\'E}S}, 65:131--210,
  1987.

\bibitem{laumon_premiere_construction_de_drinfeld}
G{\'e}rard Laumon.
\newblock Faisceaux automorphes pour ${G}{L}(n)$ : la premi{\`e}re construction
  de drinfeld.
\newblock {\em arXiv preprint alg-geom/9511004}, 1995.

\bibitem{laumon_transformation_de_fourier_generalisee}
G{\'e}rard Laumon.
\newblock Transformation de {F}ourier g{\'e}n{\'e}ralis{\'e}e.
\newblock {\em arXiv preprint alg-geom/9603004}, 1996.

\bibitem{le_bras_espaces_de_banach_colmez_et_faisceaux_coherents_sur_la_courbe_de_fargues_fontaine}
Arthur-C{\'e}sar Le~Bras.
\newblock Espaces de {B}anach--{C}olmez et faisceaux coh{\'e}rents sur la
  courbe de {F}argues--{F}ontaine.
\newblock {\em Duke Mathematical Journal}, 167(18):3455--3532, 2018.

\bibitem{lurie_spectral_algebraic_geometry}
Jacob Lurie.
\newblock Spectral algebraic geometry.
\newblock Available at \url{http://www.math.harvard.edu/~lurie/}.

\bibitem{lucas_mann_in_progress}
Lucas Mann.
\newblock Ph{D} {T}hesis, {U}niversit\"{a}t {B}onn, in progress.

\bibitem{mathew_the_galois_group_of_a_stable_homotopy_theory}
Akhil Mathew.
\newblock The {G}alois group of a stable homotopy theory.
\newblock {\em Advances in Mathematics}, 291:403--541, 2016.

\bibitem{milne1976duality}
James~S Milne.
\newblock Duality in the flat cohomology of a surface.
\newblock In {\em Annales scientifiques de l'{\'E}cole Normale Sup{\'e}rieure},
  volume~9, pages 171--201, 1976.

\bibitem{patel_de_rham_epsilon_factors}
Deepam Patel.
\newblock de {R}ham $\epsilon$-factors.
\newblock {\em Inventiones mathematicae}, 190(2):299--355, 2012.

\bibitem{ramero1998class}
Lorenzo Ramero.
\newblock On a class of {\'e}tale analytic sheaves.
\newblock {\em Journal of Algebraic Geometry}, 7:405--504, 1998.

\bibitem{ramero2005local}
Lorenzo Ramero.
\newblock Local monodromy in non-{A}rchimedean analytic geometry.
\newblock {\em Publications Math{\'e}matiques de l'IH{\'E}S}, 102:167--280,
  2005.

\bibitem{ramero2007hasse}
Lorenzo Ramero.
\newblock Hasse-{A}rf filtrations in $p$-adic analytic geometry--fourth
  release.
\newblock {\em arXiv preprint math/0703225}, 2007.

\bibitem{saibi1996transformation}
Moussa Saibi.
\newblock Transformation de {F}ourier-{D}eligne sur les groupes unipotents.
\newblock In {\em Annales de l'institut Fourier}, volume~46, pages 1205--1242,
  1996.

\bibitem{scholze_lectures_on_condensed_mathematics}
Peter Scholze.
\newblock Lectures on {C}ondensed {M}athematics.
\newblock available at
  \url{https://www.math.uni-bonn.de/people/scholze/Condensed.pdf}.

\bibitem{scholze_p_adic_hodge_theory_for_rigid_analytic_varieties}
Peter Scholze.
\newblock $p$-adic {H}odge theory for rigid-analytic varieties.
\newblock In {\em Forum of Mathematics, Pi}, volume~1. Cambridge University
  Press, 2013.

\bibitem{scholze_etale_cohomology_of_diamonds}
Peter {Scholze}.
\newblock {{\'E}tale cohomology of diamonds}.
\newblock {\em ArXiv e-prints}, September 2017.

\bibitem{scholze2020berkeley}
Peter Scholze and Jared Weinstein.
\newblock {\em Berkeley Lectures on $p$-adic Geometry}, volume 207.
\newblock Princeton University Press, 2020.

\bibitem{wewers_swan_conductors_on_the_boundary_of_lubin_tate_spaces}
Stefan Wewers.
\newblock Swan conductors on the boundary of {L}ubin-{T}ate spaces.
\newblock {\em arXiv preprint math/0511434}, 2005.

\bibitem{wintenberger_corps_des_normes}
Jean-Pierre Wintenberger.
\newblock Le corps des normes de certaines extensions infinies de corps locaux;
  applications.
\newblock In {\em Annales scientifiques de l'Ecole Normale Superieure},
  volume~16, pages 59--89, 1983.

\end{thebibliography}
\bibliographystyle{plain}

\end{document}